\documentclass{article}
\usepackage{graphicx}
\usepackage{fullpage}
\usepackage{placeins}

\usepackage{amsthm}
\usepackage{amsmath}
\usepackage{amssymb}
\usepackage{amsfonts}
\usepackage{bm}
\usepackage{url}
\usepackage[top=3cm,bottom=3cm,left=2.5cm,right=2.5cm]{geometry}
\usepackage{enumitem}
\usepackage[dvipsnames]{xcolor}
\usepackage{graphicx}
\usepackage{multirow}
\usepackage{booktabs}
\usepackage{soul}
\usepackage{subfigure}
\usepackage{rotating}
\usepackage{longtable}
\usepackage{dsfont}
\usepackage[normalem]{ulem}
\usepackage{xparse}
\usepackage{algorithm}
\usepackage{algpseudocode}
\usepackage{booktabs}
\usepackage{hyperref}
\usepackage{cleveref}
\usepackage{authblk}
\usepackage{array}
\usepackage{cases}

\NewDocumentCommand{\cg}{ O{L} O{L_1} O{L_2} m m m }{C^{#1 #4}_{#2 #5 #3 #6}}
\newcommand{\vrho}[2]{\varrho^{#1}_{#2}}
\newcommand{\bvrho}[2]{\overline{\varrho}^{#1}_{#2}}

\newcommand{\rep}[2]{\rho^{#1}_{#2}}
\newcommand{\blamb}[2]{\bm\lambda^{#1}_{#2}}
\newcommand{\bphi}[2]{\bar\phi^{#1}_{#2}}
\newcommand{\bpsi}[2]{\bar\psi^{#1}_{#2}}

\newcommand{\bPhi}{\bar\Phi^{\bl}}
\newcommand{\bPsi}{\bar\Psi^{\bl}}
\newcommand{\bV}{{\bar V}}

\newcommand{\dd}{{d}}
\newcommand{\sqrtpm}[2]{{s^{#1}_{#2}}}
\newcommand{\card}[1]{|{#1}|}
\newcommand{\dimge}[2]{\dim(V^{#1,#2})}
\newcommand{\dimpi}[2]{\dim(\bV^{#1,#2})}
\newcommand{\dimgeest}[2]{\dim^{\textup{est}}(V^{#1,#2})}
\newcommand{\dimpiest}[2]{\dim^{\textup{est}}(\bV^{#1,#2})}

\newtheorem{theorem}{Theorem}[section]
\newtheorem{proposition}[theorem]{Proposition}

\newtheorem{definition}[theorem]{Definition}
\newtheorem{corollary}[theorem]{Corollary}
\newtheorem{remark}[theorem]{Remark}
\newtheorem{example}[theorem]{Example}
\newtheorem{lemma}[theorem]{Lemma}
\newtheorem{conjecture}[theorem]{Conjecture}

\numberwithin{equation}{section}

\def\N{\mathbb{N}}
\def\K{\mathbb{K}}
\def\Z{\mathbb{Z}}
\def\R{\mathbb{R}}
\def\C{\mathbb{C}}
\def\F{\mathbb{F}}

\def\Nb{{N_{\rm blocks}}}
\def\bc{{\bm c}}
\def\bfe{{\bf e}}

\def\br{{\bm r}}
\def\bk{{\bm k}}
\def\bl{{\bm l}}
\def\bbm{{\bm m}}

\def\bmu{{\bm \mu}}
\def\bn{{\bm n}}
\def\bnu{{\bm \nu}}
\def\bL{{\bm L}}

\def\bR{{\bm R}}

\def\obm{{\overline\bbm}}

\def\mjpp{{\bbm^+_{j,p}}}

\def\mjpm{{\bbm^-_{j,p}}}

\def\calM{\mathcal{M}}

\def\calL{\mathcal{L}}
\def\calS{\mathcal{S}}
\def\calX{\mathcal{X}}
\def\calLG{\mathcal{\calL_{\rm G}}}
\def\Ml{\calM_\bl}

\def\oMl{\overline{\calM}_\bl}
\def\oMlK{\overline{\calM}_{\bl,K}}

\def\MlL{{M}}

\def\Nd{d_{\mathfrak{g}}}

\def\<{\langle}
\def\>{\rangle}

\def\Nat{{N}}

\def\dim{{{\rm dim}}}
\def\Mpm{M^{\pm}}
\def\Mp{M^{+}}
\def\Mm{M^{-}}
\def\Apm{A^{\pm}}
\def\Ap{A^{+}}
\def\Am{A^{-}}
\def\Bpm{B^{\pm}}
\def\Bp{B^{+}}
\def\Bm{B^{-}}
\def\Mup{{\MlL}^{\textup{up}}}

\title{Efficient construction and explicit dimensionality of Lie group-equivariant and permutation-invariant spaces}

\author[1]{Eloïse Barthelemy}
\affil[1]{Université Marie et Louis Pasteur, CNRS, LmB (UMR 6623), F-25000 Besançon, France}

\author[1]{Genevi\`eve Dusson\footnote{genevieve.dusson@math.cnrs.fr}}

\author[1]{Camille Hernandez}
  
\author[2]{Liwei Zhang}
\affil[2]{Institut f\"ur Geometrie und Praktische Mathematik, RWTH Aachen University, Im Süsterfeld 2, 52072 Aachen, Germany }

\date{}

\begin{document}

\maketitle

\begin{abstract}
We give a practical construction of functions of 
$N$ variables that are simultaneously equivariant under a group action and invariant under permutations, or simply group-equivariant, valid for any finite-dimensional space stable under that action.
The method applies to any connected linear Lie group: it uses the Lie algebra to build a matrix 
$M$ whose kernel corresponds exactly to the subspace with the desired symmetry properties, removing the need for Clebsch--Gordan coefficients and  explicit symmetrization over permutations. 
For the groups $SO(3)$ and $SU(2)$,
the special structure of the matrix $M$ yields a linear scaling method which significantly outperforms the typically exponential cost of existing methods in the literature. 
As a byproduct, we obtain the exact dimension of group-equivariant and permutation-invariant spaces. This allows us to compare dimensionalities and show that, for large 
$N$, this dimension is of comparable order to that of the permutation-invariant space alone, while pre-asymptotically it is orders of magnitude lower. Enforcing group-equivariance explicitly, on top of permutation-invariance, therefore yields a substantial computational gain when approximating such functions.
\end{abstract}

\vspace{1em}
\\


\tableofcontents

\section{Introduction}

In this work, we are interested in $\Nat \in \N = \{1,2,\ldots\}$ variable functions satisfying either an equivariance property with respect to some Lie group action (group-equivariance; GE), or an invariance with respect to the permutation of the variables on top of the Lie group equivariance (group-equivariance and permutation-invariance; GE-PI). 
Such functions are fundamental across diverse scientific fields, from physics and chemistry to materials science. They are especially crucial for studying interacting particle systems, where permutation-invariance arises from considering identical particles.
Often Lie group equivariance naturally arises from the physical setting. Group-equivariance related to $SO(3)$ and $O(3)$ appears in atomistic modeling, e.g. for 
interatomic potentials~\cite{Schutt2018-dw, Drautz2019-mn, Batzner2022-az, Batatia2022-zi}, Hamiltonians~\cite{Zhang2022-ay,Li2022-cj,Zhong2023-ar,Suman2025-qz,Qian2025-du,gong2023general}, density matrices~\cite{Wetherell2020-us,Hazra2024-cy,Zhang2025-ve}, wavefunctions~\cite{zhou2024-wf,unke2021se}, friction tensors in Langevin heatbaths~\cite{Sachs2025}, or protein structure~\cite{lu2024dynamicbind,eismann2021hierarchical}. 
GE-PI functions related to $SO^+(1,3)$ are used for the classification of high-energy particles and the regression of their properties (e.g. mass)~\cite{Bogatskiy2020-jt,Munoz2022-ur,Bogatskiy2024-oc}. $SU(n)$ is exploited in quantum mechanics~\cite{Lagrave2021-cw,Favoni2022-zb}.

When approximating such functions, constructing a parametrization respecting inherent symmetries is essential. 
This approach typically reduces the number of required parameters, making the approximation more efficient~\cite{Bachmayr2024-pv}.
As a toy example, the dimension of the space of $N$-variable polynomials with 
given total degree~$d$ is asymptotically about $\Nat^d$ times larger than the corresponding number of totally symmetric polynomials. 
While permutation-invariant functions are well-understood and can be efficiently constructed with symmetric polynomials and variations thereof, 
efficiently constructing functions satisfying both permutation-invariance and group-equivariance is in general difficult; this is the main motivation for this article.

So far, most works build on the construction of so-called many-body equivariant features, consisting of linearly independent basis functions that are equivariant for specific group representations.
These features are then either used as basis functions in linear parametrizations~\cite{Munoz2022-ur,Witt2023-rm,Sachs2025,Zhang2025-ve} or as building blocks of equivariant neural networks~\cite{Batatia2022-zi,geiger2022e3nn,lu2024dynamicbind}. 
For functions of $\Nat$ variables that are group-equivariant but not permutation-invariant, linearly independent equivariant features correspond to so-called generalized Clebsch--Gordan coefficients, introduced in 1962~\cite[Chapter 2]{Yutsis1962-wv}.
A practical method for computing these coefficients was, e.g., proposed in~\cite{Dusson2022-ie} within the Atomic Cluster Expansion (ACE) framework.
However, an extra difficulty arises when permutation-invariance is enforced, 
as removing linear dependencies typically requires evaluating a Gramian matrix through a sum over permutations, whose cost explodes as $\Nat$ grows~\cite{Dusson2022-ie}.
Alternative methods have been proposed such as~\cite{Nigam2020-zj,Batatia2022-zi,
Goff2024-sw,geiger2022e3nn,geiger2024accelerate} for groups like $SO(3)/SU(2)$ and~\cite{Finzi2021-em,Batatia2023-us} in a larger Lie group framework but to the best of our knowledge, the computational cost always scales exponentially
with the space dimension, making the construction infeasible for large $\Nat$.

\medskip

In this article we present an efficient numerical method for extracting the GE or GE-PI subspace of a given finite dimensional space of $N$-variable functions that is stable with respect to the group action.
We characterize GE (resp. GE-PI) functions with components in the given finite-dimensional space in Proposition~\ref{prop:GE_basis} (resp. Proposition~\ref{prop:GEPI_basis}), from which we obtain a basis of the desired subspace.
Our method starts by deriving a linear equation satisfied for every $g\in G$, similarly to~\cite{Batatia2023-us}. Differentiating this linear equation and evaluating it on a set of generators of the Lie algebra gives rise to a simpler and sparse linear system whose kernel is in one-to-one correspondence with the desired subspace.
Importantly, in the GE-PI case, our construction completely avoids explicitly symmetrizing over permutations by relying on count vectors (defined below).
We also propose an alternative recursive construction in Proposition~\ref{prop:GEPI_rec}, the latter 
similarly avoiding an explicit symmetrization over permutations.
Note that, compared to other methods, such as~\cite{Nigam2020-zj,Dusson2022-ie}, we do not need Clebsch--Gordan coefficients, which characterize 2-variable GE functions, as a prerequisite to compute generalized Clebsch--Gordan coefficients. 

For the groups $SU(2)$ and $SO(3)$, we exploit the specific structure of the linear system to solve it with a linear complexity, overcoming the exponential cost of existing methods. We verify this numerically in  Section~\ref{sec:num} by benchmarking our method against four existing packages~\cite{geiger2022e3nn,Batatia2022-zi,Batatia2023-us,geiger2024accelerate}, and show that our method is able to handle a large number of variables $N$ for which other codes are too expensive to even run. 
As a byproduct, we obtain the dimensionality of the GE and GE-PI spaces in Corollary~\ref{prop:ge-su2} and Theorem~\ref{thm:gepi-su2}, and provide asymptotic dimensionalities (see~\eqref{eq:dimge-est} and~\eqref{eq:dimpi-est}).
For rotation groups, explicit GE dimensionalities were previously obtained by one of the authors in~\cite{Dusson2022-ie}, albeit less explicitly (see Remark~\ref{rem:ace}), while GE-PI dimensionalities for specific representations appear in~\cite{Springer1977InvariantTheory,Zhao2005-zx}, though without constructing the corresponding basis, as the proofs rely mostly on characters.

As a direct outcome of exact dimensionalities and dimensionality estimates, we compare the cardinality of bases with different symmetries (no symmetry, PI, GE, GE-PI), complementing the asymptotic results of~\cite{Bachmayr2024-pv}, which were restricted to the PI case. Asymptotically, the number of PI basis functions is comparable to the number of GE-PI basis functions; however, pre-asymptotically, which is of greater practical interest, the numbers of PI and GE basis functions are both much larger than the number of GE-PI basis functions. This result is particularly relevant to the growing discussion of whether known symmetries should be imposed on the model as a constraint, or whether data augmentation is sufficient to replace them. For example,~\cite{Bigi2026-yx} discusses how symmetry-unconstrained models can still be accurate for $SO(3)$, while~\cite{Klinteback2026-xg} shows the high cost of data augmentation. A similar discussion exists for the Lorentz group in~\cite{Li2024-bm}.
    
\medskip

The outline of the article is as follows.
In Section~\ref{sec:equivariant_lie_algebra} we present two methods for obtaining a basis of GE and GE-PI functions: a direct and a recursive one. We apply the direct method to $SO(3)$ and $SU(2)$ groups, extend it to $O(3)$, and provide explicit dimensionality as well as asymptotics in this case in Section~\ref{sec:appl}. We then provide numerical results in Section~\ref{sec:num}, showcasing that the proposed method is more efficient than existing methods from the literature. 
We also provide a complexity analysis and compare GE and GE-PI dimensionalities. 
Finally we present concluding remarks in Section~\ref{sec:conclusion}.

\section{Generating equivariant function spaces using the Lie algebra}
\label{sec:equivariant_lie_algebra}

In this section, we characterize multi-variable group-equivariant functions (both GE and GE-PI) for a certain representation
via the kernel of a specific matrix 
based on the Lie algebra of the group. We start by providing preliminary definitions.

\subsection{Preliminary}

We consider a  connected linear
Lie group $G$ with neutral element $e$ and Lie algebra $\mathfrak g$.
We then consider  a vector space
$\F$
over a field $\K$, typically $\R$ or $\C$ in the examples below, $\rho$ a continuous representation of $G$ of dimension $K$ on $\K$, $\Omega$ a spatial domain, and $\cdot$ a group action of $G$ on $\Omega$:
\[
    \left\{ 
    \begin{array}{lc}
        \forall \br\in \Omega, g\in G,  &  g\cdot \br \in \Omega, \\
       \forall \br\in \Omega,  & e \cdot \br = \br, \\
       \forall g_1,g_2\in G, & 
       (g_1 g_2) \cdot \br = g_1 \cdot (g_2 \cdot \br).
    \end{array}
    \right.
\]
We define the corresponding group action on $\Omega^\Nat$, which is still denoted by $\cdot$, 
as 
\[
    \forall \bR = (\br_1,\ldots,\br_\Nat) \in \Omega^\Nat, \; 
     g \in G, \quad 
    g \cdot \bR = 
    (g \cdot \br_1, 
    \ldots, 
    g \cdot \br_\Nat).
\]
GE  and GE-PI functions are defined as follows.
\begin{definition}[Group-equivariant (GE) function]
    A function $F: \Omega^\Nat \rightarrow \F^K$ is group-equivariant for the representation $\rho$ if 
\begin{equation}\label{eq:GE}
  \forall \bR \in \Omega^\Nat, \;  g \in G,  \quad   F(g \cdot \bR) = \rho(g) \; F(\bR).
\end{equation}
\end{definition}

\begin{definition}[Group-equivariant and permutation-invariant (GE-PI) function]
\label{def:GEPI}
A function $F: \Omega^\Nat \rightarrow \F^K$ is group-equivariant for the representation $\rho$
and permutation-invariant if
\begin{equation}\label{eq:GE-PI}
  \forall \bR \in \Omega^\Nat, \; g \in G, \; \sigma \in S_\Nat, \quad   F(g \cdot \bR_\sigma) = \rho(g) \; F(\bR),
\end{equation}
where $S_\Nat$ denotes the symmetric group of degree $\Nat$, and
\[
    \forall \bR = (\br_1,\ldots,\br_\Nat) \in \Omega^\Nat, 
    \quad 
    \bR_\sigma = (\br_{\sigma_1},\ldots,\br_{\sigma_\Nat}).
\]
\end{definition}

In some contexts, e.g., for functions defined on multisets, the permutation invariance of the functions is hidden in the domain of definition. Precisely, denoting by $\textup{MS}(\Omega)$ the union of multisets of $\Omega$ with arbitrary multiplicity, that is 
\[
    \textup{MS}(\Omega) 
    = \bigcup_{n=1}^{+\infty} 
    \Big\{ 
    \left\{ 
    \br_1, \br_2, \ldots, \br_n
    \right\}, 
    \br_1,\br_2,\ldots, \br_n \in \Omega
    \Big\},
\]
a function $F:\textup{MS}(\Omega)\to\F^K$ satisfying
\[
    \forall \bR \in \textup{MS}(\Omega), \;  g \in G,  \quad   F(g \cdot \bR) = \rho(g) \; F(\bR),
\]
for the representation $\rho$ with dimension $K$ is also said to be GE-PI.

We now define the spaces for which we want to efficiently determine the subspace of GE or GE-PI functions. To do so, we consider one-variable function spaces stable with respect to the group action, and denoted by
\[
    V^l =  {\rm Span}\{ \phi^l_m, \; m\in  \calM_l \}, 
\]
 where the index $l$ lies in an index set $\calL$ which is typically countable and totally ordered (e.g. $\calL$ can be chosen as $\N_0 = \{0,1,\ldots\}$ for the group $SO(3)$, see Section~\ref{sec:compatible_bases} for more examples),
 and
 $\calM_l$ denotes a finite totally ordered index set depending on $l$, typically $\{ 1,2, \ldots, \card{\calM_l} \}$. The functions $\phi^l_m:\Omega \rightarrow \F$ are defined on a spatial domain $\Omega$ in a space endowed with a scalar product denoted by $\langle \cdot, \cdot \rangle$, typically $\langle \cdot, \cdot \rangle_{L^2(\Omega)}$,
and
are assumed for simplicity to be orthonormal, i.e.,
\begin{equation}
\forall l,l' \in \calL, \;
     m,m' \in \calM_l, \;
     m' \in \calM_{l'},
\quad 
\langle\phi^l_m,\phi^{l'}_{m'}\rangle
=\delta_{ll'}\delta_{mm'}.
\label{eq:orthogonality}
\end{equation}
The space $V^l$ being stable with respect to the group action means that
\[
    \forall g \in G,
    \quad \br \mapsto \phi^l_m(g \cdot \br) \in V^l.
\]
This implies that there exists a representation of $G$ of dimension $\card{\calM_l}$, denoted by $\rep{l}{}$, such that
\begin{equation}\label{eq:1b-basis}
   \forall g\in G, \;  \br \in \Omega,\; m\in\calM_l, \quad  \phi^l_m(g \cdot \br) = \sum_{m' \in \calM_l} \rep{l}{m,m'}(g) \phi^l_{m'}(\br),
\end{equation}
which rewrites in matrix form, denoting $\Phi^l=\left(\phi^l_m\right)_{m\in \calM_l}$,  as
\begin{equation}\label{eq:matrix_1b-basis}
   \forall g\in G, \;  \br \in \Omega, \quad  \Phi^l(g \cdot \br) = \rep{l}{}(g) \Phi^l(\br).
\end{equation}
Indeed, by the uniqueness of the coordinates in the basis 
$\Phi^l$, we have
$\rep{l}{}(e) = I_{\calM_l}$, and
\[
    \forall g_1,g_2\in G, \quad \rep{l}{}(g_1 g_2) = \rep{l}{}(g_1) \rep{l}{}(g_2).
\]
We call $\Phi^l$ a compatible basis of $V^l$ for the representation $\rep{l}{}$. 

We then consider for $\bl \in \calL^\Nat$ an $\Nat$-variable function space as the
tensor product space 
\[
    V^{\bl} = \bigotimes_{i=1}^\Nat V^{l_i},
\]
which due to~\eqref{eq:1b-basis} is stable by group action. 
An orthonormal basis for $V^{\bl}$ is 
\begin{equation}
\label{eq:calMbl}
    \Phi^\bl := 
       \left\{  \phi^\bl_\bbm := \prod_{i=1}^\Nat \phi^{l_i}_{m_i}, \quad \bbm \in \calM_\bl \right\}, \quad  \calM_\bl := \bigotimes_{i=1}^\Nat \calM_{l_i},
\end{equation}
where the corresponding scalar product on $V := \bigoplus_{\bl\in\calL^\Nat} V^{\bl}$ is defined as 
\[
  \forall  \phi^\bl_\bbm, \; \phi^{\bl'}_{\bbm'} \in V,
  \quad \langle \phi^\bl_\bbm, \phi^{\bl'}_{\bbm'} \rangle 
    = \prod_{i=1}^\Nat \langle \phi^{l_i}_{m_i}, \phi^{l'_i}_{m'_i} \rangle =\delta_{\bl\bl'}\delta_{\bbm\bbm'}.
\]
Note that we sometimes use the notation $\Phi^\bl$ as the vector containing the elements of the basis, but this should be clear from the context.

Our method for characterizing GE (respectively GE-PI) functions whose components are in $V^{\bl}$ relies on differentiating a linear equation valid for any element $g\in G$. 
Since $G$ is a connected linear Lie group, the continuous representations $\rep{l}{}$ are differentiable, whose differential at the neutral element $e\in G$ is defined as (see~\cite[Theorem 3.28]{Hall2013-af})
\begin{equation}\label{eq:rep-deriv}
    \forall X\in\mathfrak g,\quad  d\rep{l}{}(X) = \left.\frac{d}{dt}\right|_{t=0} \rep{l}{}(\exp(t  X)).
\end{equation}
Although not unique, we fix a set of generators of the Lie algebra $\calX=\{X_\dd\}_{d=1,\ldots,\Nd}$ 
throughout the paper, with a cardinality $\Nd$ which
should be distinguished from the dimensionality of the Lie algebra.
The linear equation then involves terms depending on 
 the values of $d\rep{l}{}$ evaluated at 
$X_\dd\in\calX$ denoted for $l \in \calL$ and $d = 1, \ldots, \Nd$ as
\begin{equation}\label{eq:lie-rep}
    \vrho{l,\dd}{}=[\vrho{l,\dd}{m,m'}]_{m,m' \in \calM_l}:= [d\rep{l}{}(X_\dd)]_{m,m' \in \calM_l}.
\end{equation}

\subsection{Group-equivariant  functions}

In the following we provide necessary and sufficient conditions for characterizing GE functions for a representation $\rep{L}{}$ in the space $[V^\bl]^{\dim(\rep{L}{})}$ for a given $\bl\in\calL^\Nat$.
Unless otherwise specified, we assume $L\in\calL$, so that $\calM_L$ is well-defined and $\dim(\rep{L}{})=\card{\calM_L}$. 
Denoting by $V^{\bl,L}$ the desired space of GE functions, we characterize its elements through the standard intertwining condition in the following proposition~\cite{Finzi2021-em,Hall2013-af}.

\begin{proposition}[GE condition] 
    \label{prop:GE_basis}
    Let $\bl\in\calL^\Nat$, $L\in\calL$.
    For $C^{\bl,L} \in \K^{ \mathcal{M}_L \times \mathcal{M}_\bl}$, a function 
    \begin{equation*}      
        b^{\bl,L} = C^{\bl,L} \Phi^\bl
    \end{equation*}
    belongs to $V^{\bl,L}$ if and only if 
    \begin{equation}\label{eq:diff-rep-GE}
     \forall \dd=1,\ldots,\Nd, \quad d\rep{L}{}(X_d)C^{\bl,L} = C^{\bl,L}d\rep{\bl}{}(X_d).
    \end{equation}
\end{proposition}

\begin{proof}
Similar to~\cite[Equation (10)]{Batatia2023-us}, a function $b^{\bl,L} = C^{\bl,L}\Phi^\bl$ is equivariant with respect to the representation $\rep{L}{}$ if and only if 
\begin{equation}
\label{eq:matrix_GE_sylvester}
        \forall g\in G, \quad \rep{L}{}(g)C^{\bl,L} = C^{\bl,L}\rep{\bl}{}(g),
\end{equation}
where  $\rep{\bl}{}=\displaystyle\bigotimes_{i=1}^\Nat\rep{l_i}{}$ is a tensor product representation.
As $d\rep{L}{}$ and $d\rep{\bl}{}$ are Lie algebra homomorphisms,~\cite[Theorem 3.7]{Hall2013-af} states that~\eqref{eq:matrix_GE_sylvester} holds if and only if
\begin{equation}\label{eq:diff-rep-GE-X}
    \forall X\in\mathfrak g, \quad d\rep{L}{}(X)C^{\bl,L} = C^{\bl,L}d\rep{\bl}{}(X).
\end{equation}
For the same reason, the set
$
    \{X\in\mathfrak g: d\rep{L}{}(X)C^{\bl,L} = C^{\bl,L}d\rep{\bl}{}(X)\}
$
is a Lie subalgebra of $\mathfrak g$, which implies that~\eqref{eq:diff-rep-GE-X} holds true if and only if it is satisfied at all the generators $X_1,\ldots,X_{\Nd}$.
\end{proof}
Using~\cite[Proposition 4.22]{Hall2013-af} and~\eqref{eq:lie-rep}, it is then easy to compute $d\rep{\bl}{}(X_d)$ for $d=1,\ldots, \Nd$, which has a tensor product structure, as
\begin{equation}\label{eq:GE-cond}
         d\rep{\bl}{}(X_d) = 
        \sum_{j=1}^\Nat 
        \left[ 
        \bigotimes_{s=1}^{j-1}I_{\calM_{l_s}}\otimes \vrho{l_j,\dd}{}\bigotimes_{s=j+1}^{\Nat}I_{\calM_{l_s}}
        \right]
        .
\end{equation}
Moreover the Sylvester equations~\eqref{eq:diff-rep-GE} can be easily turned into a single linear system. Indeed define
\begin{equation}
 \label{eq:matM}
     M^{\bl,L} 
     \hspace{-.1cm}
     = 
     \hspace{-.1cm}
     \begin{pmatrix}
     M_1 \\
     M_2 \\
     \vdots \\
     M_{\Nd}
 \end{pmatrix}
    \hspace{-.1cm},
 \textup{ with }
 M_\dd 
    = 
   \hspace{-.05cm}\displaystyle\sum_{j=1}^\Nat 
    \left[ 
   \bigotimes_{s=1}^{j-1}I_{\calM_{l_s}}
   \hspace{-.1cm}\otimes 
   \hspace{-.05cm}
   (\vrho{l_j,\dd}{})^T
   \hspace{-.05cm}
   \bigotimes_{s=j+1}^{\Nat}I_{\calM_{l_s}}
   \right]
   \hspace{-.1cm}
   \bigotimes  I_{\calM_L} 
   \hspace{-.1cm}
   - \bigotimes_{j=1}^\Nat I_{\calM_{l_j}}
   \hspace{-.05cm}
   \otimes \vrho{L,\dd}{},
   \textup{ for }
   \dd = 1,\ldots, \Nd,
 \end{equation}

then the matrix $C^{\bl,L}=\left(c_{\bbm,k}^{\bl,L}\right)_{k\in \mathcal{M}_L,\bbm\in\mathcal{M}_\bl}$ satisfies~\eqref{eq:diff-rep-GE} if and only if the vector
$\left(c_{\bbm,k}^{\bl,L}\right)_{(\bbm,k)  \in \mathcal{M}_\bl\times \mathcal{M}_L}  \in \ker(M^{\bl,L}).$
    For ease of notation, the superscript is sometimes omitted from the matrix $M^{\bl,L}$ when the associated $\bl$ and $L$ are clear from the context. The following corollary is a straightforward consequence.
\begin{corollary}[Dimensionality of $V^{\bl,L}$]
 Let $\bl\in\calL^\Nat$, $L\in\calL$. There holds
\[
    \dimge{\bl}{L}=\dim\big(\ker(M^{\bl,L})\big).
\]
\end{corollary}

Due to the tensor structure of the matrices $M_\dd$ in~\eqref{eq:matM}, their kernel
can be related to the tensor products of eigenvectors of matrices $\vrho{l,\dd}{}$, as presented in the next proposition. 
\begin{proposition}
\label{prop:kernel_sum}
    For any $l\in \calL$, $d \in \{ 1, \ldots, \Nd \}$, let $(\epsilon^{l,d}_j, u^{l,d}_j, v^{l,d}_j) \in \K \times \K^{\calM_l}\times \K^{\calM_l}$,
     $j \in \calM_{l}$ be respectively the eigenvalues, left eigenvectors, and right eigenvectors of the matrices $\vrho{l,\dd}{}$. Then the kernel of the matrix $M_d$ defined in~\eqref{eq:matM} is spanned by all
    \[
        u^{l_1,d}_{j_1} \otimes 
        u^{l_2,d}_{j_2} \otimes 
        \ldots 
        \otimes 
        u^{l_N,d}_{j_N} 
        \otimes v^{L,d}_{j},
    \]
    such that 
    \begin{equation}
    \label{eq:eps}
        \epsilon^{l_1,d}_{j_1} + \epsilon^{l_2,d}_{j_2} + \ldots 
    + \epsilon^{l_\Nat,d}_{j_\Nat}
    -  \epsilon^{L,d}_{j} = 0.
    \end{equation}
\end{proposition}
\begin{proof}
By using the expression of the matrices $M_d$ in~\eqref{eq:matM}, the eigenvectors of the matrices $M_d$
are $u^{l_1,d}_{j_1} \otimes 
        u^{l_2,d}_{j_2} \otimes 
        \ldots 
        \otimes 
        u^{l_N,d}_{j_N} 
        \otimes v^{L,d}_{j}$
        for all possible indices $j_i \in \calM_{l_i}$ for $i=1,\ldots,\Nat$ and $j\in \calM_{L}$,
        with corresponding eigenvalues 
        $\epsilon^{l_1,d}_{j_1} + \epsilon^{l_2,d}_{j_2} + \ldots 
    + \epsilon^{l_\Nat,d}_{j_\Nat}
    -  \epsilon^{L,d}_{j}$.
    Therefore the kernel of $M_d$ corresponds to the space spanned by eigenvectors with zero eigenvalues.
\end{proof}

In practice it would be difficult to figure out when~\eqref{eq:eps} is satisfied especially for large $N$'s and we found it more practical to directly compute a basis of $\ker(M^{\bl,L})$.
Nevertheless combining Proposition~\ref{prop:kernel_sum} and the property
$
    \ker(M^{\bl,L}) = \bigcap_{d=1}^{\Nd} \ker(M_d)
$
can be useful to obtain some explicit dimensionality of GE spaces. 

We end this section with a remark providing a link with Clebsch--Gordan coefficients.

\begin{remark}
\label{rem:CG}
    Let $\bl\in\calL^\Nat$. In the case where $\Nat=2$, that is $\bl = (l_1,l_2)$, the coupling coefficients correspond to the well-known Clebsch--Gordan coefficients. 
    For $\Nat>2$, the corresponding coupling coefficients are called generalized Clebsch--Gordan coefficients~\cite{Yutsis1962-wv,Dusson2022-ie,Batatia2023-us}. 
\end{remark}

\subsection{Permutation-invariant functions}\label{sec:pi}

We now turn to characterizing PI functions which will be the starting point to construct GE-PI functions.
This is essential in practice, since incorporating the permutation invariance into modeling can significantly reduce the dimension of the targeted function space, especially when high-dimensional problems are considered. 
A natural construction is to sum basis functions $\phi_\bbm^\bl$ over permutations, i.e. $\sum_{\sigma\in S_N} \phi_{\bbm}^{\bl} \circ \sigma$, but this leads to some redundancy in the basis, 
as if for a permutation $\pi\in S_N$, $\pi\bl=\bl$,  
the PI basis functions based on $(\bl,\bbm)$ and $(\pi\bl, \pi \bbm)$ become identical. 
This motivates us to introduce the two sets
\[
    \calS_{\bl}=\{\pi \in S_\Nat \; | \;
    \pi\bl=\bl\}\subset S_\Nat,
\]
and 
\begin{equation}\label{eq:GE_class}
    \oMl:=\{\overline{\bbm}\;|\;\bbm\in \calM_{\bl}\},
\end{equation}
where $\overline{\bbm}$ stands for the equivalence class of $\bbm$ with the equivalence relation $\sim$ being defined as 
\begin{equation}
    \label{eq:equivalence_classes}
        \bbm\sim\bbm^{\prime}  \Leftrightarrow \exists \; \pi\in \calS_{\bl}, \quad \bbm^{\prime}=\pi\bbm.
\end{equation}
Thus a first naive basis of PI functions consists of removing the redundancies in the explicit symmetrization over permutations, and is given as follows.
\begin{proposition}[Basis of PI functions in $V^\bl$]\label{prop:basis-pi}
    Let $\bl\in\calL^\Nat$. 
    A basis of the subspace of PI functions in $V^{\bl}$ consists of 
\begin{equation}
\label{eq:PI_Nbody}
       \bPsi:= \left\{ \displaystyle\bpsi{\bl}{\obm}: =\sum_{\sigma\in S_\Nat}\phi^\bl_{\obm}\circ\sigma, \quad \obm\in \oMl \right\}.
\end{equation}
In particular, the dimension of the space of PI functions in $V^\bl$ is $\card{\oMl}$.
\end{proposition}

\begin{proof}
    Since $\left\{ \phi^\bl_\bbm, \bbm\in \calM_\bl \right\}$ is 
    a basis of $V^{\bl}$, a spanning set for PI functions in $V^{\bl}$ is 
$
    \big\{ \displaystyle\bpsi{\bl}{\bbm}: = \sum_{\sigma\in S_\Nat}\phi^\bl_\bbm\circ\sigma, \; \bbm\in \calM_\bl \big\}.
$
Noting that $\bpsi{\bl}{\bbm}=\bpsi{\bl}{\bbm'}$ if there exists $\pi\in \calS_{\bl},$ such that $\bbm^{\prime}=\pi\bbm$, a spanning set 
of PI functions in $V^\bl$ is
$\left\{ \displaystyle\bpsi{\bl}{\obm}, \obm\in \oMl \right\}$.
Moreover, for $\bbm,\bbm'\in \calM_\bl$,
\begin{align*}
    \left\langle\bpsi{\bl}{\bbm},\bpsi{\bl}{\bbm'}\right\rangle =& 
    \Nat!
    \; \Big\langle\phi^\bl_\bbm, \sum_{\sigma\in S_\Nat}\phi^\bl_{\bbm'}\circ\sigma\Big\rangle.
\end{align*}

Noting that permuting the variables corresponds to an inverse permutation on the basis functions indices, and using~\eqref{eq:orthogonality}, we obtain
\begin{align*}
\left\langle\bpsi{\bl}{\bbm},\bpsi{\bl}{\bbm'}\right\rangle
    = & \Nat !\sum_{\sigma\in S_\Nat}\left\langle\phi^\bl_\bbm, \phi^{\sigma\bl}_{\sigma\bbm'}\right\rangle 
    =  \Nat !\sum_{\sigma\in S_\Nat}\delta_{\bl,\sigma\bl}\delta_{\bbm,\sigma\bbm'} 
    =  \Nat !\sum_{\sigma\in \calS_\bl}\delta_{\bbm,\sigma\bbm'}. 
\end{align*}
Using the definition of $\oMl$, we easily obtain that 
$\left\langle\bpsi{\bl}{\bbm},\bpsi{\bl}{\bbm'}\right\rangle = 0$ if and only if $\obm \neq \obm'$, that is $\left\{ \displaystyle\bpsi{\bl}{\obm},\; \obm\in \oMl \right\}$ is an orthogonal basis of PI functions in $V^\bl$ and the dimension of the space of PI functions in $V^\bl$ is the cardinality of $\oMl$.
\end{proof}

Since the evaluation of basis functions defined in~\eqref{eq:PI_Nbody} scales exponentially with $N$ 
due to the sum over permutations, 
a density projection trick was proposed~\cite{Drautz2019-mn,Dusson2022-ie}. It relies on using another set of $N$-variable PI functions to limit the evaluation cost, that are defined by
\begin{equation}
      \forall l\in \calL, \; m\in \calM_l, \quad  
  \bphi{l}{m} : \bR = (\br_1,\ldots,\br_\Nat) \in \Omega^\Nat \mapsto \sum_{i=1}^\Nat \phi^l_m(\br_i).
  \label{eq:sumphi}
\end{equation}
In spirit, this operation symmetrizes the functions with respect to permutations at the one-variable level, going from a one-variable picture to a one-body picture~\cite{Dusson2022-ie,Batatia2023-us}, i.e. the one-body functions involve all variables but do not couple them directly.
The considered one-body PI space is then 
\[
    \bV^l = {\rm Span} \{
    \bphi{l}{m}, \; 
    m \in \calM_l
    \}.
\]
Interestingly, the one-body picture does not change the representation implied by the group action, that is, denoting $\bar\Phi^{l}=\left(\bphi{l}{m}\right)_{m\in \calM_l}$,
\begin{equation}\label{eq:matrix_1b-basis-PI}
   \forall g\in G, \;  \br \in \Omega, \quad  \bar\Phi^{l}(g \cdot \br) = \rep{l}{}(g) \bar\Phi^{l}(\br).
\end{equation}
Similar to $V^{\bl}$, the space of PI functions $\bV^{\bl}$ for $\bl\in \calL^\Nat$ can be defined by forming tensor products of $\bV^{l_i}$, as 
\[
  \bV^{\bl} := \bigotimes_{i=1}^\Nat \bV^{l_i} = {\rm Span} \left\{  \bphi{\bl}{\bbm} := \prod_{i=1}^\Nat \bphi{l_i}{m_i}, \; 
  \bbm \in \calM_\bl \right\}, \quad  \calM_\bl := \bigotimes_{i=1}^\Nat \calM_{l_i}.
\]
As indicated in~\cite{Dusson2022-ie}, the evaluation cost of the $\bphi{\bl}{\bbm}$ scales linearly with respect to $\Nat$.
Similarly as with $\bpsi{\bl}{\bbm}$, some functions in this set are identical. We therefore use again the equivalence classes $\obm\in\oMl$, and obtain the following result.

\begin{proposition}[Basis of PI space $\bV^\bl$]\label{prop:basis-pi2}
    Let $\bl\in\calL^\Nat$. A basis of $\bV^{\bl}$ is
\[
  \bPhi:=  \left\{ \bphi{\bl}{\obm},\; \obm\in\oMl\right\}.
\]
In particular, the dimension of $\bV^\bl$ is $\card{\oMl}$.
\end{proposition}

We defer the proof of this proposition to Appendix~\ref{app:proof_basis_pi2}. 
Note that Definition~\eqref{eq:sumphi} can naturally be extended to an arbitrary number of variables by defining $\bphi{l}{m} : \bR \in \textup{MS}(\Omega) \mapsto \sum_{\br \in \bR } \phi^l_m(\br)$. 
We can therefore also account for multiset functions within the same framework, although Proposition~\ref{prop:basis-pi2} may no longer hold in the multiset setting.

\begin{remark}[Group action]
\label{rem:group_action}
Note that $\bV^\bl$ is stable with respect to the group action, and due to~\eqref{eq:matrix_1b-basis}, and~\eqref{eq:matrix_1b-basis-PI}, the latter acts in the same way over the functions $\phi^\bl_\bbm$, 
 $\bpsi{\bl}{\obm}$ and $\bphi{\bl}{\obm}$. 
\end{remark}

We now introduce a scalar product on the space $\bV^\bl$, which will be useful in the following.

{\begin{definition}[Scalar product on $\bV^\bl$]
    As $\left\{  \bphi{\bl}{\obm},\; \obm \in \oMl \right\}$ is a basis of $\bV^{\bl}$, we can define an inner product on $\bV^{\bl}$ by the relation
    \[
        \langle\bphi{\bl}{\obm},\bphi{\bl}{\obm'}\rangle=\delta_{\obm,\obm'}.
    \]
\end{definition}

We end this section with a discussion over the equivalence classes considered in the PI case. 
First for $\bl$'s containing a single value $l$, i.e., 
    $\bl = (l,l,\ldots,l)\in\calL^N,$
the classes defined in~\eqref{eq:GE_class} are in one-to-one correspondence with ordered $\bbm\in\calM_\bl$, which we use by default as representatives for the corresponding classes. 
For a general $\bl$ possibly containing different values, we assume that it is ordered for notation simplicity, but without loss of generality. 
The index $\bl$ can then be divided into a set of minimal length $\Nb$ of sub-indices $\{\bl^{(j)}\}_{j=1}^\Nb$ of length $|\bl^{(j)}|$ with identical elements, which are defined as follows.
\begin{definition}[Minimal partition]\label{def:min_part}
    Let $\bl\in\calL^\Nat$ be ordered, we say that $(\bl^{(j)})_{j=1}^{\Nb}$ is the minimal partition of $\bl$ if each block $\bl^{(j)}$ has a single repeated value and the blocks have distinct values.
\end{definition}
It immediately follows that 
\begin{equation}\label{eq:sets_partition}
     \calS_{\bl} = \bigotimes_{j=1}^\Nb \calS_{|\bl^{(j)}|},
\quad 
\text{and}
\quad
     \oMl = \bigotimes_{j=1}^\Nb \overline{\calM}_{\bl^{(j)}}.
\end{equation}
For $\obm=(\obm^{(1)}, \ldots, \obm^{(\Nb)})\in\oMl$, we use $\bbm = (\bbm^{(1)}, \ldots, \bbm^{(\Nb)})$ as its representative, where $\bbm^{(j)}$ are ordered for all $j\in\{1,2,\ldots,\Nb\}$. 
Equation~\eqref{eq:sets_partition} indicates that $\oMl$ has a tensor product structure across the blocks, analogous to that of $\calM_\bl$ in~\eqref{eq:calMbl}. As we will see in the next section, this tensor-product structure helps show that PI has an effect only within each block when GE is imposed on top of PI.

\subsection{Group-equivariant and permutation-invariant functions}\label{sec:ge-pi}

We now turn to characterizing GE-PI functions with components in $\bV^\bl$. 
As in the GE case, we can state necessary and sufficient conditions for characterizing GE-PI functions for the representation $\rep{L}{}$ in the space $\left[ \bV^{\bl} \right]^{\card{\calM_L}}$, which we denote by $\bV^{\bl,L}$. 
This result is similar to that of the GE case (Proposition~\ref{prop:GE_basis}) up to a multiplication on the left and right by two matrices, which, as we show below, modifies the tensor structure obtained in~\eqref{eq:GE-cond}.

\begin{proposition}[GE-PI condition]
\label{prop:GEPI_basis}
Let $\bl\in\calL^\Nat$, $L\in\calL$.
For $C^{\bl,L}\in\K^{ \mathcal{M}_L \times \oMl}$, 
a function 
\begin{equation*}      
    b^{\bl,L} = C^{\bl,L} \bPhi
\end{equation*}
belongs to $\bV^{\bl,L}$ if and only if 
\begin{equation}\label{eq:diff-rep-GEPI}
    \forall \dd=1,\ldots,\Nd, \quad d\rep{L}{}(X_\dd)C^{\bl,L} = C^{\bl,L}S^\bl d\rep{\bl}{}(X_\dd)T^\bl,
\end{equation}
where the two matrices $S^\bl\in\K^{\oMl\times\calM_{\bl}}$ and $T^\bl\in\K^{\calM_{\bl}\times\oMl}$ are defined component-wise, for $\obm\in\oMl$ and $\bbm'\in\calM_\bl$, by
    \begin{equation}\label{eq:def_ST}
        S^\bl_{\obm,\bbm'} = \begin{cases}
            1, & \quad\textup{if } \; \obm'= \obm \; \textup{ and }\;  \bbm' \; \textup{ is ordered}, \\
            0, & \quad \textup{otherwise,}
        \end{cases} \quad\quad \textup{and} \quad\quad  T^\bl_{\bbm',\obm} = \begin{cases}
            1, & \quad\textup{if }\;  \obm=\obm', \\
            0, & \quad \textup{otherwise.}
        \end{cases}
    \end{equation}
\end{proposition}
Note that multiplying by $S^\bl$ on the left selects elements indexed by the ordered class representatives and multiplying by $T^\bl$ on the right sums over the elements of the same class.
To prove this proposition, we need a preliminary lemma.  
\begin{lemma}\label{lem:proj-rep}
    The map $S^{\bl}\rep{\bl}{}T^{\bl}$ is a representation of $G$.
\end{lemma}

\begin{proof}
Noting that for $\bbm,\bnu\in \calM_\bl$,
$\left[T^\bl S^\bl\right]_{\bbm,\bnu} =1$ if $\overline{\bnu}=\obm$ and $\bnu$ is ordered and otherwise 0,
    there holds,  for $\bbm\in\calM_\bl$ and $\obm'\in\oMl$, 
    \[
        \forall g\in G, \quad \left[\rep{\bl}{}(g)T^\bl\right]_{\bbm,\obm'}=\displaystyle\sum_{\bmu\in\calM_\bl}\rep{\bl}{\bbm,\bmu}(g)T^\bl_{\bmu,\obm'}= \sum_{\bmu\in\obm'}\rep{\bl}{\bbm,\bmu}(g),
    \]
    and
    \begin{equation*}
        \forall g\in G, \quad \left[T^\bl S^\bl\rep{\bl}{}(g)T^\bl\right]_{\bbm,\obm'} = \displaystyle\sum_{\bnu\in\calM_\bl}\left[T^\bl S^\bl\right]_{\bbm,\bnu}\left[\rep{\bl}{}(g)T^\bl\right]_{\bnu,\obm'} = \sum_{\bmu\in\obm'}\rep{\bl}{\widetilde\bbm,\bmu}(g),
    \end{equation*}
    where $\widetilde\bbm$ is the ordered representative of $\obm$.
    From~\eqref{eq:equivalence_classes} there exists ${\pi}\in\calS_\bl$ such that ${\pi}\bbm=\widetilde{\bbm}$. Since ${\pi}\bl=\bl$ for any ${\pi}\in\calS_\bl$
    and $\pi$ is a bijection on $\obm'$, there holds
    \[
        \forall g\in G, \quad \sum_{\bmu\in \obm'}\rep{\bl}{\widetilde\bbm,\bmu}(g) = 
         \sum_{\bmu\in \obm'}\rep{\bl}{\pi \bbm,\bmu}(g) = 
         \sum_{\bmu\in \obm'}\rep{\bl}{\bbm,\pi^{-1}\bmu}(g) = 
         \sum_{\bmu\in\obm'}\rep{\bl}{\bbm,\bmu}(g),
    \]
    which proves
    \begin{equation}\label{eq:TS_lemma}
        \forall g\in G, \quad T^{\bl}S^{\bl}\rep{\bl}{}(g)T^{\bl}=\rep{\bl}{}(g)T^{\bl}.
    \end{equation}
    Using this equality, there holds
    \[
        \forall g,h\in G, \quad S^{\bl}\rep{\bl}{}(g)T^{\bl}S^{\bl}\rep{\bl}{}(h)T^{\bl}=S^{\bl}\rep{\bl}{}(g)\rep{\bl}{}(h)T^{\bl}=S^{\bl}\rep{\bl}{}(gh)T^{\bl},
    \]
    which shows that $S^{\bl}\rep{\bl}{}T^{\bl}$ is a representation. 
\end{proof}

\begin{proof}[Proof of Proposition~\ref{prop:GEPI_basis}]
Let $\bl\in \calL^N, L\in \calL$.
We consider a function  $b^{\bl,L} = C^{\bl,L}\bPhi \in  [\bV^\bl]^{\calM_L}$, which is already PI.
Using~\eqref{eq:1b-basis}, the group action over $\bPhi$
reads as 
\[
   \forall g\in G,\; \bR\in \Omega^\Nat, \obm \in \oMl, \quad \bphi{\bl}{\obm}(g \cdot \bR) = \sum_{\bbm' \in \calM_\bl} \rep{\bl}{\bbm,\bbm'}(g) \bphi{\bl}{\obm'}(\bR),
\]
or in matrix form, using the matrices $S^\bl$ and $T^\bl$, as
\[
    \bPhi (g\cdot \bR) = S^\bl \rep{\bl}{}(g) T^\bl\bPhi (\bR),
\]
which is the analog of~\eqref{eq:matrix_1b-basis} for the representation $S^{\bl}\rep{\bl}{}T^\bl$.
Thus, applying the same strategy as in the proof of Proposition~\ref{prop:GE_basis}, we have that $b^{\bl,L}$ is equivariant if and only if~\eqref{eq:diff-rep-GEPI} is satisfied.
\end{proof}

Compared to the GE case, the matrix $S^\bl d\rep{\bl}{}(X_\dd)T^\bl$
only has a tensor structure with respect to 
blocks defined in Definition~\ref{def:min_part}
and not anymore with respect to all indices. 
Indeed, from~\eqref{eq:def_ST} and~\eqref{eq:sets_partition}, one can show that
\[
S^\bl = \bigotimes_{j=1}^{\Nb}  
   S^{\bl^{(j)}}, \quad 
T^\bl 
= 
\bigotimes_{j=1}^{\Nb}  
   T^{\bl^{(j)}},  \quad \text{and} \quad 
   S^{\bl^{(j)}} T^{\bl^{(j)}} = I_{\overline{\calM}_{\bl^{(j )}}},
   \quad \text{for $j=1,\ldots,\Nb$}.
\]
Noting that $\rep{\bl}{} = \bigotimes_{j=1}^\Nb \rep{\bl^{(j)}}{}$ 
 we can easily write
 \begin{equation}
 \label{eq:SrhoT}
        S^\bl d\rep{\bl}{}(X_\dd)T^\bl 
    = \displaystyle\sum_{j=1}^\Nb\left[\bigotimes_{s=1}^{j-1} I_{\overline{\calM}_{\bl^{(s)}}}   
    \bigotimes 
   \left( 
   S^{\bl^{(j)}} d\rep{\bl^{(j)}}{}(X_\dd)T^{\bl^{(j)}}
   \right)
\bigotimes_{s=j+1}^{\Nb}I_{\overline{\calM}_{\bl^{(s)}}}
\right].
 \end{equation}
On the good side,
one can now focus on blocks with identical values  first, based on which the matrices 
$ S^\bl d\rep{\bl}{}(X_\dd)T^\bl $ can be particularly simply expressed with count vectors, leading to explicit expressions that we present below. 

Let $\bl\in\calL^\Nat$ consist of $\Nat$ copies of an identical value $l\in\calL$. 
The classes $\obm\in\oMl$ are in one-to-one correspondence with their Parikh vectors or count vectors 
$\blamb{}{\obm}\in\N_0^{\calM_l}$, which are defined as the vectors containing the number of occurrences of the elements in $\calM_l$ (see e.g.~\cite[Supplementary lecture H]{Kozen2007-ot}). 
Note that the sum of the elements in the count vector is equal to $\Nat$, we can identify $\oMl$ with the set
\begin{equation}
\label{eq:lambda}
    \Lambda_{\bl} := 
    \{\blamb{}{}\in \N_{0}^{\calM_l}:\sum_{i\in\calM_l}\blamb{}{i} = \Nat\},
\end{equation}
and we similarly denote $\obm_{\blamb{}{}}\in\oMl$ the corresponding element of $\blamb{}{}\in\Lambda_\bl$.
We give an explicit example to better illustrate this concept.
\begin{example}
    Assume $l=1$, $\calM_l = \{-1,0,1\}$ and $\bl=(l,l,l,l,l)$, then the class $\obm=\overline{(-1,-1,0,1,1)}\in\oMl$ has a count vector 
    $\blamb{}{\obm}=(2,1,2)$, whose components stand for the two occurrences of $-1\in\calM_l$, one occurrence of $0\in\calM_l$ and two occurrences of $1\in\calM_l$, respectively. Conversely, for $\blamb{}{}=(1,3,1)$, $\obm_{\blamb{}{}}=\overline{(-1,0,0,0,1)}\in\oMl$.
\end{example}

Using this, the matrix $ S^\bl d\rep{\bl}{}(X_\dd)T^\bl $ can be simply expressed as follows.
\begin{proposition}
\label{prop:rep_with_count_vec}
Let $\bl = (l,\ldots, l)\in\calL^\Nat$, $L\in\calL$.
    For $d=1,\ldots, \Nd$, denote 
    \[
        \bvrho{\bl,d}{} \; := \; \; S^\bl d\rep{\bl}{}(X_\dd)T^\bl,
    \]
    then there holds for $\obm,\obm' \in \oMl$ 
    \begin{equation}\label{eq:decomposition_SrhoT}
    \bvrho{\bl,d}{\obm,\obm'} \;
    = \begin{cases}
            \sum_{p\in\calM_\bl}\vrho{l,d}{p,p}\blamb{}{\obm,p}, & \quad\textup{if }\; \blamb{}{\obm}=\blamb{}{\obm'},\\
            \vrho{l,d}{q,p}\blamb{}{\obm,q}, & \quad\textup{if }\; \blamb{}{\obm}=\blamb{}{\obm'} - {\bfe}_p + {\bfe}_q ,
            \textup{ for } p,q\in\calM_{l}, \; p\ne q,\\
            0, & \quad \textup{otherwise,}
        \end{cases}
    \end{equation}
    where ${\bfe}_k$ denotes the $k$-th vector of the canonical basis of $\K^{\calM_l}$.
\end{proposition}
\begin{proof}
    First, for $\dd =1,\ldots,\Nd$ and 
    $\obm,\;\obm' \in \oMl$, 
    a direct calculation using~\eqref{eq:GE-cond} gives 
\begin{equation}
\label{eq:system_GEPI}
    \bvrho{\bl,d}{\obm,\obm'}=S^\bl d\rep{\bl}{}(X_\dd)T^\bl [\obm,\obm'] 
    = \sum_{\bmu\in\obm'}
    \sum_{j=1}^\Nat \Bigg( \vrho{l,\dd}{m_j,\mu_j}\prod_{\substack{s=1 \\s\ne j}}^\Nat \delta_{m_s,\mu_s} \Bigg).
\end{equation}
The above terms on the right-hand side can only be non-zero if $\bmu$ and $\bbm$ differ by at most one element.
Once the class $\obm'$ is given, the classes $\obm\in \oMl$ with ordered representative element $\bbm$ that can contribute can be characterized using the count vectors. Indeed, changing one element in a class $\bbm'$  corresponds to modifying the count vector $\blamb{}{\obm'}$ as follows: one element in $\blamb{}{\obm'}$ is decreased by one, another is increased by one. That is, the only classes that can contribute to the sum are either $\obm'=\obm$, or $\obm$'s satisfying
\begin{equation}
        \label{eq:interacting_classes}
     \blamb{}{\obm}
     =  \blamb{}{\obm'} - {\bfe}_p + {\bfe}_q,
\end{equation}
for $p\neq q\in\calM_l$ such that $\blamb{}{\obm',p}> 0$,
and a short calculation, noting that $\bbm$ is the ordered representative of $\obm$,
proves~\eqref{eq:decomposition_SrhoT}.
\end{proof}
Similarly as in the GE case, for a general $\bl\in\calL^\Nat$, we can turn~\eqref{eq:diff-rep-GEPI} into a linear system involving the matrix 
$M^{\bl,L}$, defined by
\begin{equation}
 \label{eq:matM_GEPI}
     M^{\bl,L} 
     \hspace{-.05cm}
     =
     \hspace{-.1cm}
     \begin{pmatrix}
     M_1 \\
     M_2 \\
     \vdots \\
     M_{\Nd}
 \end{pmatrix}
 \hspace{-.1cm},
 \textup{ with }
 M_\dd = 
 \hspace{-.2cm}\displaystyle\sum_{j=1}^\Nb
 \hspace{-.1cm}
 \left[\bigotimes_{s=1}^{j-1} 
 \hspace{-.05cm}
 I_{\overline{\calM}_{\bl^{(s)}}}
 \hspace{-.1cm} \otimes 
 \hspace{-.05cm}
 (\bvrho{\bl^{(j)},d}{})^T  \bigotimes_{s=j+1}^{\Nb}
 \hspace{-.15cm}
 I_{\overline{\calM}_{\bl^{(s)}}}
 \hspace{-.05cm}
 \right]
 \hspace{-.07cm}\otimes
 I_{\calM_L} - 
 \hspace{-.05cm}
 \bigotimes_{j=1}^{\Nb}
 \hspace{-.2cm}
 I_{\overline{\calM}_{\bl^{(j )}}}
 \hspace{-.1cm} \otimes \vrho{L,\dd}{},
 \;
 \dd = 1,\ldots, \Nd,
 \end{equation}
for which the superscript is sometimes omitted when the associated $\bl$ and $L$ are clear. 
Hence, a basis of GE-PI functions can be obtained as follows.

\begin{theorem}
\label{thm:GEPI_basis_psi}
    Let $\bl\in\calL^\Nat$, $L\in\calL$, 
    $\bPhi=\left(\bphi{l}{\obm}\right)_{\obm\in \oMl}$, and $\bPsi=\left(\bpsi{l}{\obm}\right)_{\obm\in \oMl}$.
    Let $C^{\bl,L}_i=\left(c_{\obm,k,i}^{\bl,L}\right)_{k\in \mathcal{M}_L,\bbm\in\oMl}$ with 
    \[
        (c_{\obm,k,i}^{\bl,L})_{\obm\in\oMl,k\in\calM_L}, 
        \quad i=1,\ldots, \dim(\ker(M^{\bl,L})),
    \]
     a basis of $\ker(M^{\bl,L})$, where $M^{\bl,L}$ is defined in~\eqref{eq:matM_GEPI}.
    A basis of GE-PI
    functions for the representation $\rep{L}{}$ in the space $\left[\bV^\bl\right]^{\calM_L}$
    is given by
    \begin{equation}
    \label{eq:GEPIbasis1}
        b^{\bl,L}_i = 
    C^{\bl,L}_i\bPhi
    ,\;\; i=1,\ldots,\dim\big(\ker(M^{\bl,L})\big).
    \end{equation}
    Similarly, a basis of GE-PI
    functions for the representation $\rep{L}{}$ in the space $\left[V^\bl\right]^{\calM_L}$ is
    given by
    \begin{equation}
    \label{eq:GEPIbasis2}
 b^{\bl,L}_i = 
    C^{\bl,L}_i \bPsi,\;\; i=1,\ldots,\dim\big(\ker(M^{\bl,L})\big).
    \end{equation}
\end{theorem} 

\begin{proof}
    Since the group action of $G$ acts the same over the functions $\bpsi{\bl}{\obm}$ and $\bphi{\bl}{\obm}$ and with Proposition~\ref{prop:basis-pi}, the proof is similar for~\eqref{eq:GEPIbasis1} and~\eqref{eq:GEPIbasis2}. It is a straightforward consequence of Proposition~\ref{prop:GEPI_basis} using~\eqref{eq:SrhoT} and~\eqref{eq:decomposition_SrhoT}.
\end{proof}

As in the GE case, the dimensionality of the GE-PI space $\bV^{\bl,L}$ is the dimensionality of the kernel of $M^{\bl,L}$, as stated in the following corollary.
\begin{corollary}[Dimensionality of $\bV^{\bl,L}$]\label{cor:dim_pi}
 Let $\bl\in\calL^\Nat$, $L\in\calL$. There holds
\[
    \dim(\bV^{\bl,L})
    =\dim\big(\ker(M^{\bl,L})\big).
\]
\end{corollary}

As indicated by Proposition~\ref{prop:rep_with_count_vec}, the only two differences between the matrix $\MlL^{\bl,L}$ in the GE-PI case~\eqref{eq:matM_GEPI} and that in the GE case~\eqref{eq:matM} are: (1) the set $\calM_\bl$ being replaced by the set of classes $\oMl$, an even smaller set compared with $\calM_\bl$; and (2) the presence of the count vectors.
Therefore, our construction of the GE-PI space is almost as efficient as that of the GE space. To the best of our knowledge, this has not been the case in any existing work, where the presence of permutation invariance often incurs substantial additional costs~\cite{Batatia2023-us,Dusson2022-ie,geiger2022e3nn}.

\begin{remark}[GE is a special case of GE-PI]
\label{rem:gegepi}
If the components of $\bl$ are pairwise distinct, then $\Nb=\Nat$, $\calS_\bl=\{id\}$, and the equivalence class of every $\bbm\in\Ml$ is $\obm=\{\bbm\}$. Consequently, $\oMl$ is naturally identified with $\calM_\bl$, and~\eqref{eq:matM_GEPI} reduces to~\eqref{eq:matM}. Thus, the GE and GE-PI constructions yield the same coupling coefficients in this case, although their functions are expressed in different underlying bases, respectively $\phi^{\bl}_{\bbm}$ and $\bphi{\bl}{\obm}$ (or $\bpsi{\bl}{\obm}$). In this sense, GE can be regarded as a special case of GE-PI.
To obtain pairwise distinct indices in $\bl$ with possibly identical group representations, one may augment each representation label $\ell_i$ with an auxiliary channel and take $\bl = ((1,\ell_1), (2,\ell_2),\ldots, (N, \ell_N)).$
\end{remark}

\begin{remark}
    When the group $G$ is compact, the linear system associated with finding the kernel of the matrix $M^{\bl,L}$~\eqref{eq:matM_GEPI} can, by augmenting the results in \cite{HofmannMorris1990Weight, HofmannMorris2023CompactGroups}, be further reduced to equations associated with at most two suitably chosen elements $X_1,X_2 \in \mathfrak g$, that is
    $M^{\bl,L} = (M_1 ; M_2)$, with $X_1$ in the Lie algebra of a maximal torus of $G$, whose corresponding matrix $M_1$ can be taken diagonal, and, when $G$ is nonabelian, one additional element $X_2$ outside this toral subalgebra. Note that these two elements do not necessarily form a Lie-generating set of $\mathfrak g$.
\end{remark}

\subsection{Recursive construction}\label{sec:recursive}
In this section, we show how the GE-PI bases can be constructed recursively. This also offers a way to obtain an exact recursive formula for the dimensionality of the considered GE-PI spaces. 
Indeed, a priori it is not so clear whether the tensor product of two GE-PI sub-bases is a basis of the final space or an overcomplete set of functions. To obtain these results, we need a few additional assumptions on the considered group and the representations. We assume that the group $G$ is compact, so that the Haar measure is well-defined~\cite[Chapter 4]{Hall2013-af}, and we consider a complete set of mutually inequivalent irreducible representations $\rep{l}{}$ for $l\in \calL$ with corresponding spaces $V^l$.
Under these assumptions, the space $\bV^\bl$ for $\bl\in\calL^N$ can be decomposed into the direct sum of GE-PI spaces. 
\begin{proposition}[Direct sum of spaces with GE-PI functions]\label{prop:direct_sum_GEPI}
    Suppose $G$ is a compact Lie group,
    $\bl
    \in \calL^\Nat$ and 
    $\bV^\bl$ is stable by group action.
    There holds
    \begin{equation}
    \label{eq:direct_sum_GEPI}
         \bV^{\bl} = \bigoplus_{L \in \calL}  
        \;
        \bigoplus_{k\in \calM_L} [\bV^{\bl,L}]_k,
    \end{equation}
    where $[\bV^{\bl,L}]_k$ denotes the space spanned by the $k$-th component of the functions in $\bV^{\bl,L}$.
\end{proposition}
\begin{proof}
    The space $\bV^\bl$ is stable by group action and the group action defines a representation $\rep{\bl}{}$, which can be blockdiagonalized with irreducible representations $\rep{L}{}$, $L\in\calL$ and corresponding basis functions $v^L_{i,k}$, $i \in \{1, \ldots, d_L\}$ and $k\in \calM_L$, where $d_L$ is the number of copies of $\rep{L}{}$ in the blockdiagonalization.
    Noting that 
    $\bV^{\bl,L} = {\rm Span} \{ [v^L_{i,k}]_{k\in \calM_L}, i\in \{1, \ldots, d_L\}  \}$, the result follows. 
\end{proof}
We now provide a recursive construction of GE-PI bases and a recursive formula for the dimensionality of the GE-PI spaces. Using Remark~\ref{rem:gegepi}, a similar result holds in the GE case, which we omit for conciseness.
\begin{proposition}[GE-PI basis based on recursion]
\label{prop:GEPI_rec}
Let $N,N_1,N_2\in \N$ with $N_1+N_2=N$.
Let $\bl\in\calL^\Nat$ and $L\in\calL$ . 
Suppose $\bl = (\bl^{(1)}, \bl^{(2)})\in\calL^{\Nat_1}\times\calL^{\Nat_2}$, then a spanning set for $\bV^{\bl,L}$ is 
\[
    b^{\bl,L}_{(L_1,i_1,L_2,i_2,j)} = C^{(L_1,L_2),L}_{j} \left(b^{\bl^{(1)},L_1}_{i_1}\otimes b^{\bl^{(2)},L_2}_{i_2}\right),\quad 
    \begin{array}{l}
         L_1,L_2\in\calL,   \\
         i_1\in\{1,\ldots, \dimpi{\bl^{(1)}}{L_1}\}, \\
         i_2\in\{1,\ldots, \dimpi{\bl^{(2)}}{L_2}\},\\
         j \in \{1,\ldots,\dimge{(L_1,L_2)}{L}\},
    \end{array}
\]
    where $b^{\bl^{(1)},L_1}_{i_1}$, 
$b^{\bl^{(2)},L_2}_{i_2}$ are bases of $\bV^{\bl^{(1)},L_1}$ and $\bV^{\bl^{(2)},L_2}$, respectively, and
$\{C^{(L_1,L_2),L}_{j}\}_{j=1,\ldots,\dimge{(L_1,L_2)}{L}}$ is a basis of $\ker(M^{(L_1,L_2),L})$ with $M^{(L_1,L_2),L}$ being defined in~\eqref{eq:matM}. 
Moreover, if $\bl^{(1)} \cap \bl^{(2)} = \varnothing$, this family is a basis of $\bV^{\bl,L}$ and there holds
\begin{equation}\label{eq:dim_GEPI_rec}
    \dimpi{\bl}{L} = \sum_{L_1,L_2\in \calL} \dimge{(L_1,L_2)}{L} \;\dimpi{\bl^{(1)}}{L_1}\;\dimpi{\bl^{(2)}}{L_2}.
\end{equation}
\end{proposition}
\begin{proof}
First, in the non-intersecting case where $\bl^{(1)} \cap \bl^{(2)} = \varnothing$, since $\bV^\bl = \bV^{\bl^{(1)}} \otimes \bV^{\bl^{(2)}}$,~\eqref{eq:direct_sum_GEPI} yields that the scalar components of
\begin{equation}
\label{eq:basisss}
    \left(b^{\bl^{(1)},L_1}_{i_1}\otimes b^{\bl^{(2)},L_2}_{i_2}\right)_{L_1,L_2 \in \calL, i_1 \in \{ 1, \ldots, \dimpi{\bl^{(1)}}{L_1}\}, i_2 \in \{1, \ldots, \dimpi{\bl^{(2)}}{L_2}\} }
\end{equation}
form a basis of $\bV^\bl$. 
Denote $b^{\bl,L}=\displaystyle\sum_{L_1,L_2,i_1, i_2}C^{(L_1,L_2),L}_{i_1,i_2}\cdot\left(b^{\bl^{(1)},L_1}_{i_1}\otimes b^{\bl^{(2)},L_2}_{i_2}\right)\in\left[\bV^\bl\right]^{\calM_L}$ for some coefficients $C^{(L_1,L_2),L}_{i_1,i_2}\in\K^{\calM_{L}\times\calM_{(L_1,L_2)}}$, the group action on $b^{\bl,L}$ reads as 
\begin{align*}
    b^{\bl,L}(g\; \cdot) &= 
    \sum_{L_1,L_2,i_1, i_2}C^{(L_1,L_2),L}_{i_1,i_2}\left(\rep{L_1}{}(g)\otimes\rep{L_2}{}(g)\right)\left(b^{\bl^{(1)},L_1}_{i_1}\otimes b^{\bl^{(2)},L_2}_{i_2}\right).
\end{align*}
Thus, as~\eqref{eq:basisss}
is a basis, $b^{\bl,L}\in \bV^{\bl,L}$ if and only if 
\[
    \forall L_1,L_2, i_1, i_2, \quad \rep{L}{}C^{(L_1,L_2),L}_{i_1,i_2}=C^{(L_1,L_2),L}_{i_1,i_2}\left(\rep{L_1}{}\otimes\rep{L_2}{}\right)=C^{(L_1,L_2),L}_{i_1,i_2}\rep{(L_1,L_2)}{}.
\]
Note that the dependency on $i_1,i_2$ has been removed so that we can form a basis using the same coefficients regardless of $i_1$ and $i_2$. Moreover, these are the same equations as in Proposition~\ref{prop:GE_basis}, with $(L_1,L_2)$ replacing $\bl$, hence the solutions are the same, as stated in the proposition.
In that case, the dimensionality \eqref{eq:dim_GEPI_rec} is then a direct consequence.
If $\bl^{(1)} \cap \bl^{(2)} \neq \varnothing$, we have $\bV^\bl = \bV^{\bl^{(1)}} \otimes \bV^{\bl^{(2)}}$ but the scalar components of the functions 
in~\eqref{eq:basisss} 
    only form a spanning set and not a basis of  $\left[\bV^\bl\right]^{\calM_L}$.
    Therefore, following the steps of the non-intersecting case, we only generate a spanning set of GE-PI functions.
\end{proof}
Note that in the case of intersecting $\bl^{(1)},\bl^{(2)},$ we can numerically obtain a basis by, for example, performing a singular value decomposition on the spanning set given in Proposition~\ref{prop:GEPI_rec}.
In practice, it is not always clear which construction is more efficient between the direct and recursive constructions. However, with the result on GE-PI functions for $\bl\in\calL^N$ with minimal partition $\bl=(\bl^{(1)},\bl^{(2)},\ldots,\bl^{(\Nb)})$, we expect it to be efficient to first compute basis functions for the different blocks $\bl^{(j)}$ of identical values, and then assemble them using Proposition~\ref{prop:GEPI_rec}, especially for large $N$, see Section~\ref{sec:num-rec} for an elementary numerical comparison.

\section{Application to some rotation groups}\label{sec:appl}

We now apply our construction to the groups  $SU(2)$ and $SO(3)$, which naturally appear in quantum mechanics, molecular mechanics, and computer graphics. Specifically, the generators of the underlying Lie algebra are taken as the infinitesimal rotations under the ZYZ convention. We show how the matrices~\eqref{eq:matM} and~\eqref{eq:matM_GEPI} 
simplify for these specific groups, and investigate the algebraic properties of the corresponding GE and GE-PI bases. In particular, we provide their exact dimensionality and prove asymptotic estimates.

\subsection{Representations and Lie algebra}
We first introduce the explicit expression of the Wigner-D matrices with respect to Euler angles, which are irreducible representations of $SU(2)$ and $SO(3)$~\cite{wigner1931gruppentheorie}. 
Let $X_1$ and $X_2$ be the
infinitesimal generators of rotations about the $z$- and $y$-axes, respectively. Then
up to a multiplicative constant,
the commutator $[X_1,X_2]$ is the infinitesimal rotation about the $x$-axis. Hence $\{X_1,X_2\}$ forms a Lie-generating set of both $\mathfrak{su}(2)$ and $\mathfrak{so}(3)$, the Lie algebras of $SU(2)$ and $SO(3)$.

In addition, an element $Q$ of either $SU(2)$ or $SO(3)$ can thus be parameterized by the three Euler angles $(\alpha, \beta, \gamma)\in
[0,4\pi)\times[0,\pi)\times[0,4\pi)$ for $SU(2)$
or 
$[0,2\pi)\times[0,\pi)\times[0,2\pi)$
for $SO(3)$, as
\[
    Q(\alpha,\beta,\gamma) =
    \exp(\alpha X_1)\exp(\beta X_2)\exp(\gamma X_1).
\]
In particular, there holds
\begin{equation}\label{eq:link_derivatives}
    X_1 =
    \left.\frac{\partial}{\partial\alpha}Q(\alpha,0,0)\right|_{\alpha=0} = 
    \left.\frac{\partial}{\partial\gamma}Q(0,0,\gamma)\right|_{\gamma=0},
    \textup{ and }
    X_2 =
    \left.\frac{\partial}{\partial\beta}Q(0,\beta,0)\right|_{\beta=0}.
\end{equation}
Then, the Wigner-D matrix, which we denote by $D^\ell$ instead of $\rep{\ell}{}$ as it is standard, is defined as follows.

\begin{definition}[Wigner-D matrix]\label{def:wigner-D}
For $\ell = 0,\; \frac12, \;1, \;\frac32, \;\ldots$, the Wigner-D matrix $D^\ell$, under the ZYZ convention, is defined as a $(2\ell+1)\times(2\ell+1)$ matrix, whose elements read
\begin{equation}
      D^\ell_{\mu m}\left(Q(\alpha, \beta, \gamma)\right)
      = e^{-\mathrm{i}\mu\alpha}d^\ell_{\mu m}(\beta) e^{-\mathrm{i}m\gamma},
      \label{eq:wigner-D}
\end{equation}
for $\mu,m \in \{-\ell, -\ell+1, \ldots, \ell-1, \ell\}$, where 
\[
    d^\ell_{\mu m}(\beta) = [(\ell+m)!(\ell-m)!(\ell+\mu)!(\ell-\mu)!]^\frac{1}{2}\sum_{s = \max(0, \mu-m)}^{\min(\ell+\mu,\ell-m)} \frac{(-1)^{s}(\cos \frac{\beta}{2})^{2l+\mu-m-2s}(\sin \frac{\beta}{2})^{m-\mu+2s}}{(\ell+\mu-s)!s!(m-\mu+s)!(\ell-m-s)!}.
\]
\end{definition}

In particular, the irreducible representations of $SU(2)$ are the Wigner-D matrices for 
$\ell \in \calL_{SU(2)}$ defined~as
\begin{equation}
    \label{eq:calL}
     \calL_{SU(2)} :=  \left\{\frac{\ell}{2}: \ell\in\N_0 \right\},
\end{equation}
while the irreducible representations of $SO(3)$ correspond to Wigner-D matrices for 
$\ell \in \calL_{SO(3)}$
defined as 
\begin{equation}
\label{eq:calL2}
     \calL_{SO(3)} := \N_0.
\end{equation}
To unify the notation and emphasize the dependency of the index sets for the irreducible representations on the group $G$, we denote the sets as $\calLG$, particularly for $G=SU(2)$ and $G=SO(3)$, in this section. For notational convenience, we define for $\ell\in\calLG$ the set
\begin{equation}
    \label{eq:calMU2}
        \calM_\ell = \{-\ell,-\ell+1,\ldots,\ell-1,\ell\}, 
        \quad \text{ with }
        \quad 
         \card{\calM_\ell} = 2\ell+1.
\end{equation}
To compute the matrices $M^{\bl,L}$ defined in~\eqref{eq:matM} and~\eqref{eq:matM_GEPI} we need
the differential of the Wigner-D representations at the neutral element of the group (identity denoted by~$I$) applied to the Lie-algebra generators $X_1$ and $X_2$, which are given element-wise 
in the following Proposition~\ref{prop:deri_alpha_gamma}.

\begin{proposition}[Derivatives of Wigner-D matrices]
\label{prop:deri_alpha_gamma}
Let $\ell\in\calLG$, $\mu,m \in \calM_\ell$. Then
\begin{align}
\label{eq:deri_alpha}
    & \vrho{\ell,1}{\mu,m} = \left[dD^\ell(X_1)\right]_{\mu,m}
    = \frac{\partial}{\partial\alpha}
    \left.D^\ell_{\mu m}\left(Q(\alpha,0,0)\right)\right|_{\alpha=0}
    = \frac{\partial}{\partial\gamma}
    \left.D^\ell_{\mu m}\left(Q(0,0,\gamma)\right)\right|_{\gamma=0}
    = -\mathrm{i}m\delta_{\mu m},
    \\
\label{eq:deri_beta}
    & \vrho{\ell,2}{\mu,m} =
    \left[dD^\ell(X_2)\right]_{\mu,m}
    = \frac{\partial}{\partial\beta}
    \left.D^\ell_{\mu m}\left(Q(0,\beta,0)\right)\right|_{\beta=0}
    =
    \begin{cases}
        \frac{1}{2}
        \sqrtpm{+}{\ell,m},
        & m=\mu+1,\\
        \frac{1}{2}
        \sqrtpm{-}{\ell,m},
        & m=\mu-1,\\
        0,
        & \textup{otherwise},
    \end{cases}
\end{align}
where 
\begin{equation}\label{eq:def_srqt}
    \sqrtpm{\pm}{\ell,m} = \pm \sqrt{(\ell \mp m +1)(\ell\pm m)}.
\end{equation}
\end{proposition}
While this proposition is a standard result, we provide a proof for completeness in Appendix~\ref{app:Wigner-Dmatricesproof}.
The quantities $\sqrtpm{\pm}{\ell,m}$ will appear often in the rest of this section. In particular, there holds for $\ell \in \calLG$ and $m\in\calM_\ell$
\begin{equation}
\label{eq:akpm}
    \sqrtpm{+}{\ell,m} = -\sqrtpm{-}{\ell,m-1},
\end{equation}
and
\begin{equation}\label{eq:A-diff}  
    (\sqrtpm{+}{\ell,m})^2 - (\sqrtpm{-}{\ell,m})^2= 2m.
\end{equation}

\subsection{One-variable compatible bases}
\label{sec:compatible_bases}

We now introduce some possible one-variable compatible bases for $SU(2)$ and $SO(3)$. Although a large freedom exists in the choice of the bases that are stable by group action, we focus on compatible bases with respect to the irreducible representation $D^\ell$ for $\ell\in\calLG$. Specifically, we introduce bases that are frequently used in real implementations. 
For each case, we need to define the considered spatial domain $\Omega$, the chosen $\F$ and $\K$, the considered group action, and the one-variable compatible bases.

\paragraph{Wigner-D matrices}

First, for all $\ell\in\calLG$, the elements in each column of $D^\ell$ are functions from $\Omega := G$ to $\F := \C$, and form a compatible basis with respect to the representation $D^\ell$ itself. This corresponds to $\K := \C$.
Precisely, let $l=\ell$, $\calM_l=\calM_\ell$, then for any index $\mu\in\calM_\ell$, the space
\[
    \displaystyle V^l_\mu = \textup{Span} 
    \big\{D^\ell_{m\mu}:G\to\C, \; m\in\calM_l \big\}
\]
is stable under the group action defined by matrix multiplication as
\[
    \forall Q, \; \br \in G, \quad 
    Q \cdot \br := Q \br,
\]
since $D^\ell$ is a representation.
Hence, columns of Wigner-D matrices $\{D^\ell_{m\mu}:G\to\C\}_{m\in\calM_\ell}$ form compatible bases for the representations $D^\ell$. 

\paragraph{Spherical harmonics}

In particular, for $\ell\in\N_0$, $\mu=0\in\calM_l$, and for $Q\in SO(3)$ represented by Euler angles $(\alpha,\beta,\gamma)$, there holds
\[
    D^\ell_{m0}(Q) = D^\ell_{m0}(\alpha,\beta,\gamma) = e^{-\textup{i}m\alpha}\cdot (-1)^m\sqrt{\frac{(l-m)!}{(l+m)!}}P_\ell^m(\cos\beta),
\]
where $P_\ell^m$ is the associated Legendre polynomial. Clearly, $D^\ell_{m0}$ is independent of the angle $\gamma$, and is proportional to the complex spherical harmonics of order $\ell$ and angular momentum quantum number $m$: $Y_\ell^m(\alpha,\beta)$.
In other words, the spherical harmonics $\{Y_\ell^m\}_{m\in \calM_\ell}$ form a compatible basis corresponding to $D^\ell$ for $\ell\in\N_0$, the irreducible representations of $SO(3)$. Thus, defining $\Omega := S^2$ the 2-dimensional unit sphere, $\F := \C$, $\K:= \C$, and the group action as matrix-vector multiplication
\[
    \forall Q \in SO(3), \; \br \in S^2, \quad 
    Q\cdot \br := Q\br,
\]
the space for $l:=\ell\in \N_0$ defined by
\[
    V^l := \textup{Span} \big\{Y_l^m:S^2\to\C, \; 
    m\in\calM_l \big\}
\]
is stable under group action. Note that we could also take $\F := \R$, and $\K :=\R$ by considering real spherical harmonics, although another set of irreducible representations would have to be taken into account in this case.

\paragraph{One-variable basis on $\R^3$}\label{paragraph:one_variable}

In practice, it is of particular interest to investigate a compatible basis on $\R^3$. This is often done by padding a 1D radial basis to spherical harmonics. 
Precisely, we take the spatial domain $\Omega$ as $\R^3\backslash \{\mathbf{0}\}$, 
and let $\F := \C$, and $\K :=\C$. Then for $\br=|\br|\hat{\br}\in\R^3\backslash\{\mathbf{0}\}$ with $|\br|> 0$ and 
$\hat{\br} \in S^2$, we define the group action as
\[
    \forall Q\in SO(3), \;  \br \in \R^3\backslash\{\mathbf{0}\}, 
    \quad 
    Q \cdot \br := |\br| (Q \hat{\br}).
\]
Taking $\{P_n\}_{n\in\N_0}\subset L^2(\R)$, an orthogonal basis of $L^2(\R)$, we
define for $l=(n,\ell)\in\N_0\times\N_0=:\calL$ the set $\calM_l := \calM_\ell$, and for $m\in \calM_l$ 
\[
    \phi^l_m=\phi^{(n,\ell)}_m : \br \in \R^3\backslash\{\mathbf{0}\}  \mapsto P_n(|\br|)Y_\ell^m(\hat{\br}).
\]
Then, $\big\{\phi^l_m \big\}_{m\in\calM_l}$
forms a compatible basis of $SO(3)$ corresponding to representation $D^\ell$, as the radial function $P_n$ remains invariant under $SO(3)$ rotations. The space
\[
    V^{l} = \textup{Span}\big\{ \phi^l_m, \; 
    m\in\calM_l \big\}
\]
is therefore stable under the defined group action. 

\paragraph{One-variable basis on higher-dimensional spaces}

In fact, any invariant transformation with respect to rotations 
has no impact on the stability of the corresponding basis. One can take into account higher-dimensional spaces $\Omega = W \times S^2$, with $W$ independent of rotations. 
For example, a particle in $\R^3$ is often described not only by its position, but also by other properties, such as its chemical species. In this case, one could, e.g., take 
\[
    W = \N\times\big(\R\backslash\{0\}\big), \quad \Omega = \N \times \big(\R^3\backslash\{\mathbf{0}\}\big), \quad \F = \C, \quad \K = \C,
\]
with the group action being defined as 
\[
    \forall Q \in SO(3),\; (z,\br) \in \Omega, 
    \quad Q \cdot (z,\br) := 
    (z,Q\br).
\]
Hence, we may write for $l=(n_1,n_2,\ell)\in \N \times \N_0 \times \N_0=:\calL$ the basis index set $\calM_l:=\calM_\ell$, and the compatible basis 
\[
    \{\phi^l_m: (z,\br)
    \in \N \times \big(\R^3\backslash\{\mathbf{0}\}\big) \mapsto F_{n_1}(z)P_{n_2}(|\br|)Y_\ell^m(\hat\br)\}_{m\in\calM_l},
\]
where $F_{n_1}$ are basis functions depending on the chemical species (and possibly other parameters independent of rotations),
which are usually chosen as $\delta$-functions (also called one-hot embedding in some contexts),
$P_{n_2}$ are radial functions as stated above. 
Thus, the
corresponding function space
\[
    V^l = \textup{Span}
    \big\{\phi^l_m, \; m\in\calM_l \big\}
\]
is stable for the representation $D^\ell$. 

\medskip

To summarize, we present in Table~\ref{tab:comp_bases} the aforementioned compatible bases in a unified form.
Note that $\ell$ is always a component of $l$, and the basis index set for $l\in\calL$ satisfies $\calM_l = \calM_\ell$.
In some cases that should be clear from the context, when it comes to operations on $l\in\calL$, we refer to the corresponding operations on the $\ell$-components.

\begin{remark}
    Note that in general the vector space $\F$ and the field $\K$ do not have to be the same. For example, considering functions with two components, such as one spin up and one spin down, we would have $\F$ = $\C^2$ and $\K=\C$.
\end{remark}

\begin{table}[htb!]
\centering
\caption{Summary of one-variable compatible bases for the irreducible representations of $SU(2)$ and $SO(3)$}
\label{tab:comp_bases}
\begin{tabular}{|m{3cm}|m{3.4cm}|m{2cm}|m{2.5cm}|c|}
    \hline
    \centering Compatible basis functions $\phi^l_m$ & \centering {Applicable groups $G$} & \centering{Domain $\Omega$} & \centering{Index set $\calL$} & 
    {$\;\;\; \displaystyle\bigoplus_{l\in\calL} V^l$}
    \\
    \hline
    \centering $D^l_{m\mu}$ & \centering  $SU(2)$, $SO(3)$ & \centering   $G$   &  \centering    $\calLG$ \eqref{eq:calL}, \eqref{eq:calL2}    & -         \\
    \hline
    \centering $Y_l^m$    &  \centering    $SO(3)$      &  \centering  $S^2$ & \centering    $\N_0$       & $L^2(S^2)$ \\
    \hline
    \centering $P_nY_l^m$ & \centering     $SO(3)$      &  \centering $\R^3\backslash\{\mathbf{0}\}$ & \centering $\N_0\times\N_0$ & $L^2(\R^3\backslash\{\mathbf{0}\})$ \\
    \hline
    \centering $F_{n_1} P_{n_2} Y_l^m$ & \centering     $SO(3)$      & \centering  $\N\times\big(\R^3\backslash\{\mathbf{0}\}\big)$ & \centering $\N\times\N_0\times\N_0$ & $L^2(\Omega)$ \\
    \hline
\end{tabular}
\end{table}

In the following, given compatible bases $\{\phi^l_m\}_{m\in\calM_\ell}$ for $l\in \calL$ depending on the chosen group and representations, such as those summarized in Table~\ref{tab:comp_bases} and the corresponding function spaces $V^l$, our goal is to practically find a basis for the GE space $V^{\bl,L}$ or the GE-PI space $\bV^{\bl,L}$, for given $\bl\in\calL^N$ and $L\in\calL_G$, using the strategy described in Section~\ref{sec:equivariant_lie_algebra}. We present the details in the coming subsections.

\subsection{Construction of multi-variable equivariant spaces}
\label{sec:3.3}

In this section we focus on the construction of the matrices~\eqref{eq:matM} and~\eqref{eq:matM_GEPI} for the groups $SU(2)$ and $SO(3)$, and show how their specific structure can be leveraged to compute equivariant spaces efficiently, and to obtain their explicit dimensionality.
Since the GE case can be recovered from the GE-PI case (see Remark~\ref{rem:gegepi}), we focus on~\eqref{eq:matM_GEPI}.
Instead of considering the whole matrix $\MlL$, we deal with each matrix $M_\dd,\; d=1,2$ independently, with the indices $1$ and $2$  corresponding to the Euler angles $\alpha$ (or equivalently, $\gamma$)} and $\beta$, respectively.

\subsubsection{Matrix $\MlL_1$}
Due to~\eqref{eq:deri_alpha}, the matrix $\MlL_1$ is diagonal and
\begin{equation}
\label{eq:M1_PI}
        \forall (\bbm,k)\in {\oMl}\times\calM_L, \qquad \MlL_1[(\bbm,k),(\bbm,k)] = -\left(\sum_{i=1}^N m_i -k\right)\cdot \mathrm{i},
\end{equation}
with $\calM_L$ defined in~\eqref{eq:calMU2} and $\oMl$ defined in~\eqref{eq:GE_class}. For simplicity, we denote $\sum\bbm := \sum_{i=1}^N m_i$ for any vector of scalars $\bbm$ of length $\Nat$. 
Additionally, for $\bl\in \calL^N$, we define $\sum\bl:=\sum\boldsymbol{\ell}$.
The following proposition holds trivially. 
\begin{proposition}\label{prop:ker-M1-GEPI}
    Let $\bl\in \calL^N$, $L\in \calLG$, then for $\MlL_1$ being defined in~\eqref{eq:M1_PI}, there holds
    \[
        \ker(\MlL_1) = \{\bc=(\bc_{\bbm,k})_{(\bbm,k)\in\oMl\times\calM_L}\in\K^{\oMl\times\calM_L}: \bc_{\bbm,k} = 0 \textup{ if } \sum\bbm -k \neq 0\}.
    \]
\end{proposition}
Since $\ker(M)=\ker(M_1)\cap\ker(M_2)$, 
Proposition~\ref{prop:ker-M1-GEPI} immediately implies that 
\begin{equation}\label{eq:kerM-nullpart}
    \forall\bc=(\bc_{\bbm,k})_{(\bbm,k)\in\oMl\times\calM_L}\in\ker(\MlL), \quad \bc_{\bbm,k} \ne 0 \textup{ only if } k=\sum\bbm\in\calM_L.
\end{equation}
Additionally, we have the following result on the dimensionality of the GE-PI space in certain cases. 
\begin{proposition}
Let $\bl\in\calL^N$ and $L \in \calLG$  be such that $(L+\sum\bl)\notin\Z$. There holds
\[ 
    \dimpi{\bl}{L} = 0.
\]
\end{proposition}
\begin{proof}
For $k\in\calM_L$ and $\bbm\in\oMl$, we have $k-L\in\Z$ and $\sum\bbm-\sum\bl\in\Z$. When $\big(L+\sum\bl\big)\notin\Z$, the condition $k = \sum\bbm$ in~\eqref{eq:kerM-nullpart} is never fulfilled, thus $\ker(\MlL)=\{\mathbf{0}\}$. 
Using Corollary~\ref{cor:dim_pi}, we conclude that $\dimpi{\bl}{L} = \dim\big(\ker(\MlL)\big)=0$.
\end{proof}

\subsubsection{Matrix $\MlL_2$}

We now turn to investigate the matrix $\MlL_2$. 
From~\eqref{eq:kerM-nullpart}, we may consider for $\MlL_2$ only the restricted column index set
\begin{equation}\label{nonzero-ker}
    \big\{(\bbm,\sum \bbm): \;  \bbm\in\oMl, \; \sum \bbm \in\calM_L\big\} = \Big\{(\bbm,\sum \bbm):\;\bbm\in\oMl, \;|\sum \bbm|\le L,\;\sum\bbm-L\in\Z\Big\},
\end{equation}
which can be identified with
\[
\overline{\calM}_{\bl,\le L} := \Big\{\bbm\in\oMl:|\sum\bbm|\le L,\;\sum\bbm-L\in\Z\Big\}.
\]
In the following, we show that 
a similar condition applies to the rows of $M_2$ so that 
it can be structured as a sparse block matrix with indices grouped by values of $\sum \bbm$ as in Figure~\ref{fig:block_structure_of_M}(a). 
We then prove that either the upper or the lower half of $M_2$ is sufficient to determine its kernel. 
For numerical efficiency reasons (see Section~\ref{sec:num-complexity}), we then permute rows of $M_2$ 
to define another matrix $\widetilde{M_2}$
(Figure~\ref{fig:block_structure_of_M}(b)), which has the same kernel as $M_2$.
We similarly show that either half of $\widetilde{M_2}$, in particular the upper half $M^{\rm up}$ (Figure~\ref{fig:block_structure_of_M}(c)), is sufficient to determine the desired kernel. The latter will be used in practice (Section~\ref{sec:num}) to efficiently determine the coupling coefficients. 
The blocks appearing in the matrices of Figure~\ref{fig:block_structure_of_M} will be properly introduced below.
\begin{figure}[ht]
    \centering
    \includegraphics[width=\linewidth]{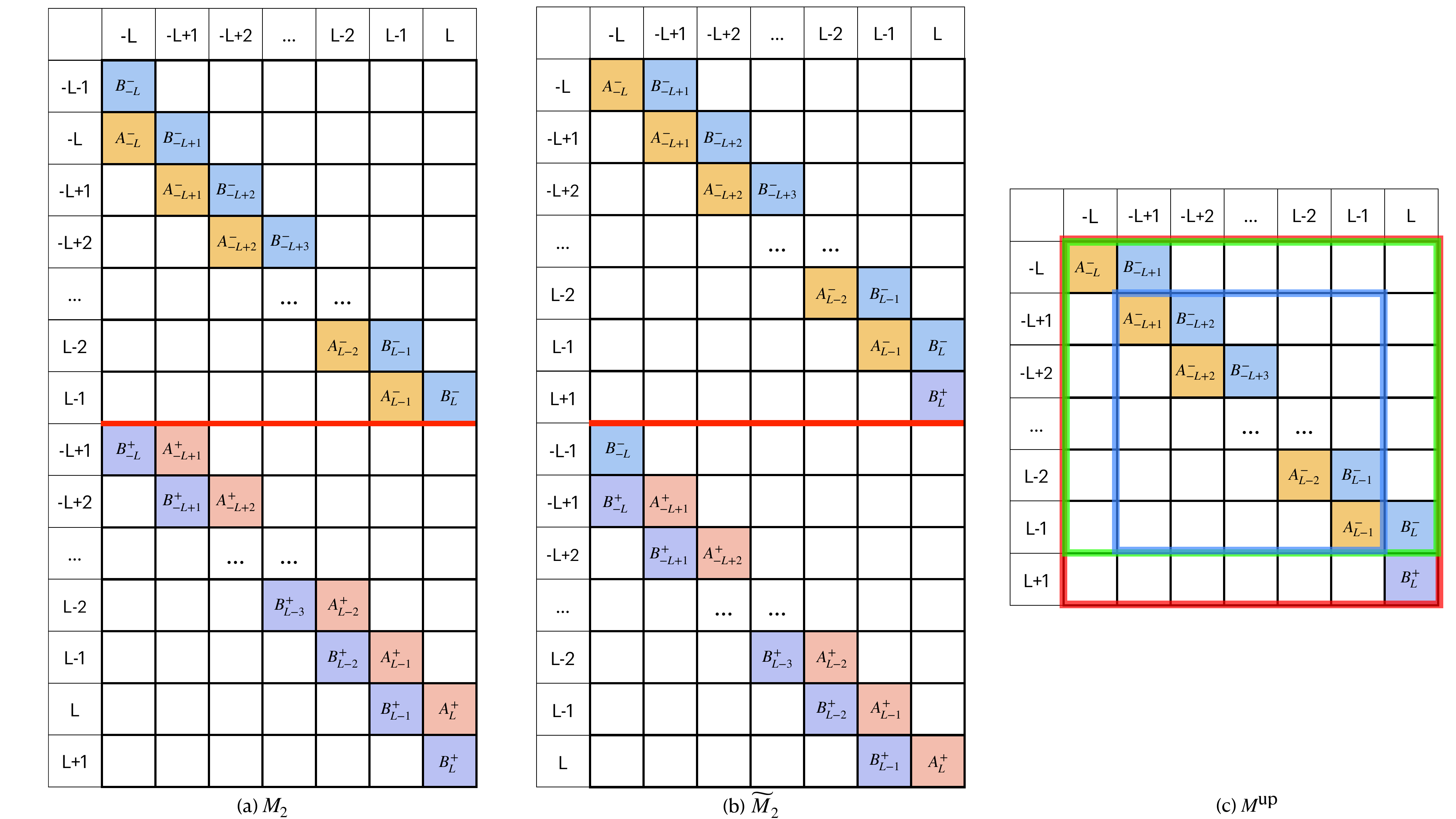}
    \caption{Block structure of the matrix $\MlL_2$: (a) $\MlL_2$: original derivation; (b) $\widetilde{\MlL}_2$: Row reordering of  $\MlL_2$ to enable a fast kernel solver; (c) $\MlL^{\textup{up}}$: upper half of $\widetilde{\MlL}_2$; the final matrix for which we find the kernel. In the plots, the red, green and blue boxes correspond to the matrix in the cases $L<\sum\bl,\; L=\sum\bl$, and $L>\sum\bl$, respectively. All the blocks left blank are zeros.
    }
    \label{fig:block_structure_of_M}
\end{figure}

We start by introducing a notation for the set of indices summing to $K$
\begin{equation}
\label{eq:calMK}
       \oMlK := \{\bbm\in\oMl: \; \sum\bbm=K\}.
\end{equation}
Thus,
\[
\overline{\calM}_{\bl,\le L} = \bigcup_{K\in\calM_L}\oMlK.
\]
As a consequence of~\eqref{eq:matM_GEPI},~\eqref{eq:decomposition_SrhoT}, and~\eqref{eq:deri_beta}, the entries of the matrix $M_2$ for $(\bbm,k)\in \oMl\times\calM_L$, $\bbm'\in \overline{\calM}_{\bl,\le L}$, after omitting a constant factor 1/2 which does not affect the kernel, are given by
\begin{equation}
        \MlL_2[(\bbm,k),\bbm'] = 
    \begin{cases}
        -\sqrtpm{-}{L,\sum\bbm}, & \qquad  \textup{if }\bbm'=\bbm,\; k = \sum\bbm+1, \\
        -\sqrtpm{+}{L,\sum\bbm}, & \qquad  \textup{if }\bbm'=\bbm,\; k = \sum\bbm-1, \\
        \blamb{}{\bbm'^{(j)},p+1}\sqrtpm{-}{\bl^{(j)}_1,p}, & \qquad \textup{if }\bbm' = \mjpp, \; k = \sum\bbm', \\
        \blamb{}{{\bbm'}^{(j)},p-1}\sqrtpm{+}{\bl^{(j)}_1,p}, & \qquad \textup{if }\bbm' = \mjpm, \;  k = \sum\bbm', \\
        0, & \qquad \textup{otherwise},
    \end{cases}
    \label{eq:M2-PI}
\end{equation} 
where, for $j\in\{1,2,\ldots,\Nb\}$ and $p\in\calM_{\bl^{(j)}_1}$ satisfying $p\pm1\in\calM_{\bl^{(j)}_1}$ and $\blamb{}{\bbm^{(j)},p}>0$,
the classes $\bbm^\pm_{j,p}$ are defined from $\bbm$ through 
the count vector relation
\begin{equation}\label{eq:lambda-shift}
    \blamb{}{(\bbm^\pm_{j,p})^{(j)}} = \blamb{}{\bbm{^{(j)}}}-\bfe_p+\bfe_{p\pm 1}, 
    \quad \textup{and} \quad \forall i\ne j, \quad 
    \blamb{}{(\bbm^\pm_{j,p})^{(i)}}=\blamb{}{\bbm^{(i)}}. 
\end{equation}
Note that the following equivalence holds
\begin{equation}\label{eq:equivalence_mjpplus}
    \bbm'=\bbm^\pm_{j,p} \Longleftrightarrow \bbm = \bbm^{'\mp}_{j,p\pm1 }.
\end{equation}
To characterize the row indices that can lead to nonzero entries in the third and fourth lines of~\eqref{eq:M2-PI} more precisely, we prove the following lemma.

\begin{lemma}\label{lem:row_ind}
Let $\bl\in\calL^N$, $L \in \calLG$, and $\bbm\in\oMl$.
There exists $j\in \{ 1,\ldots, \Nb\}$ and $ p\in\calM_{\bl^{(j)}_1}$ such that 
$\mjpp \in \overline{\calM}_{\bl,\le L}$ if and only if $\bbm\in\oMlK$ for some $K\in\calM_{L+1}\backslash\{L,L+1\} $ and $ K<\sum\bl$.
Similarly, there exists $j\in \{ 1,\ldots, \Nb\}$ and $ p\in\calM_{\bl^{(j)}_1}$ such that 
$\mjpm \in \overline{\calM}_{\bl,\le L}$ if and only if $\bbm\in\oMlK$, for some $K\in\calM_{L+1}\backslash\{-L,-L-1\} $ and $ K>-\sum\bl$.
\end{lemma}
\begin{proof}
    We will only show the first statement as the second can be proved analogously.
    Suppose that $\mjpp\in\overline{\calM}_{\bl,\le L}$ for some $j\in \{1,\ldots,\Nb\}$ and $p\in\calM_{\bl^{(j)}_1}$, and set $K:=\sum\bbm$. 
    Then $\bbm\in\oMlK$ and 
    $\sum\mjpp=K+1$. Since $\mjpp\in\overline{\calM}_{\bl,\le L}$, $K\in\calM_{L+1}\backslash\{L,L+1\}$. Moreover, the existence of $\mjpp$ implies that not all the entries of $\bbm$ attain their maximal values, and hence $K<\sum\bl$.
    Conversely, suppose that $\bbm\in\overline{\calM}_{\bl,K}$ for some $K\in\calM_{L+1}\backslash\{L,L+1\} $ and $ K<\sum\bl$. 
    If for all $j\in \{1,\ldots,\Nb\},p\in\calM_{\bl^{(j)}_1}$, $\mjpp \notin \overline{\calM}_{\bl,\le L}$, it means that $\mjpp \notin \overline{\calM}_{\bl}$ since $\sum\mjpp=K+1\in\calM_L$. Thus, every entry of $\bbm$ attains its maximal value, that is, $\bbm = (\ell_1, \ldots, \ell_\Nat)$. Hence, $K=\sum\bbm = \sum\bl$, contradicting $K<\sum\bl$. This proves the result.
\end{proof}
Lemma~\ref{lem:row_ind} implies 
that the indices leading to nonzero rows in the matrix $M_2$ are
\begin{equation}\label{eq:MlK_minus}
    \big\{ (\bbm, {\sum}\bbm+1), \;  \bbm \in \overline{\calM}_{\bl,K} \big\},
\quad K \in \calM_{L+1} \backslash \{ L, L+1\},\quad K<\sum\bl,
\end{equation}
for the terms involving $\bbm^+_{j,p}$ 
and \begin{equation}\label{eq:MlK_plus}
\big\{ (\bbm, {\sum} \bbm-1), 
\; \bbm \in \overline{\calM}_{\bl,K} \big\},
\quad K \in \calM_{L+1} \backslash \{ -L-1, -L\},\quad K>-\sum\bl,
\end{equation}
for the terms involving $\bbm^-_{j,p}$.
For each $K$, since the second component of each element in the sets~\eqref{eq:MlK_minus} and~\eqref{eq:MlK_plus} 
is uniquely determined by its first component $\bbm\in\oMlK$, both sets can be identified with $\oMlK$. We note that the restrictions $K<\sum\bl$ and $K>-\sum\bl$ have an impact only when $L>\sum\bl$.

With the zero rows omitted, the matrix $\MlL_{2}$ can be decomposed, based on the two row index sets~\eqref{eq:MlK_minus} and~\eqref{eq:MlK_plus}, as
\begin{equation}\label{eq:split-M-PI}
    \MlL_{2} = \begin{pmatrix}
         \Mm_{2} \\
         \Mp_{2} \\
     \end{pmatrix},
\end{equation}
with 
$\Mm_{2}=\big[\Mm_{K,K'}\big]_{K\in \calM_{L+1}\backslash\{L,L+1\},~K'\in\calM_L}$ 
and
$\Mp_{2}=\big[\Mp_{K,K'}\big]_{K\in\calM_{L+1}\backslash\{-L,-L-1\},~K'\in\calM_L}$, 
whose rows and columns are grouped by $K=\sum\bbm$ and $K' = \sum\bbm'$, respectively,
and whose blocks are
defined as
\begin{equation*}
\Mpm_{K,K'} = \begin{cases}
    \Apm_{K}, 
    \quad & \textup{if } K'=K, \\
    \Bpm_{K\mp1}, 
    \quad & \textup{if } K'=K\mp 1, \\
    \mathbf{0}, \quad &\textup{otherwise,}
    \end{cases}
\end{equation*}
as illustrated in Figure~\ref{fig:block_structure_of_M}(a).
Here, the matrices $\Apm_{K}\in\R^{\oMlK\times\oMlK}$ are diagonal matrices 
\[
    \Apm_{K} = -\sqrtpm{\pm}{L,K} I_{\card{\oMlK}},
\]
correspond to the first two cases in~\eqref{eq:M2-PI},
while the matrices $\Bpm_{K}\in\K^{\overline{\calM}_{\bl,K\pm1}\times\oMlK}$ 
defined as 
\begin{equation}
\label{eq:B+-}
\Bpm_{K}[\bbm,\bbm'] :=
\begin{cases}
            \blamb{}{{\bbm'}^{(j)},p\mp 1}\sqrtpm{\pm}{\bl^{(j)}_1,p}, & \quad\textup{if }\; 
            \bbm' = \bbm^\mp_{j,p},\\
            0, & \quad \textup{otherwise,}
        \end{cases}
        = 
\begin{cases}
            \big(\blamb{}{\bbm^{(j)},p\mp 1}+1\big)\sqrtpm{\pm}{\bl^{(j)}_1,p}, & \quad\textup{if }\; 
            \bbm = (\bbm')^{\pm}_{j,p\mp 1},\\
            0, & \quad \textup{otherwise,}
        \end{cases}        
\end{equation}
correspond to the third and fourth cases in~\eqref{eq:M2-PI}.
Note that there exists at most one pair of $(j,p)$ such that $\bbm'=\mjpp$ or $\bbm'=\mjpm$ so that the matrices $\Bpm_{K}$ are well-defined.
While in the matrix $\MlL_2$, there appears only $B^\pm_K$, for $ K\in\calM_L$, the definition of the $B^\pm_K$ matrices applies to general $K\in\calM_\infty$, where $\calM_\infty=\{\pm L \;,\; L\in\calL_G\}$, 
viewing $B^\pm_K$ as a generalized matrix with 0 rows (or 0 columns)
when the set $\overline{\calM}_{\bl,K \pm 1}$ (or $\oMlK$)
is empty.
We collect below some useful properties of the $B^\pm_K$ matrices. 

\begin{lemma}[Properties of the $B^\pm_K$ matrices]\label{lem:laddar-pi}
    Let $\bl\in\calL^N$, $L\in\calL_G,$ $K\in\calM_\infty$, 
    and let the matrices $B_K^\pm$ be defined in~\eqref{eq:B+-}, then
    \begin{equation}\label{eq:per-sym-B-pi}
        \big(\Bm_{K}\big)^TW_{K-1} = - W_{K}\Bp_{K-1}, 
    \end{equation}
    and 
        \begin{equation}\label{eq:ladder-pi}
        \big(\Bp_{K}\big)^TW_{K+1}\Bp_{K} - \big(\Bm_{K}\big)^TW_{K-1}\Bm_{K} = -2K\cdot W_K,
    \end{equation}
    with $W_K\in \K^{\oMlK\times\oMlK}$, a diagonal matrix defined as
    \begin{equation}
    \label{eq:WK}
                W_K[\bbm,\bbm] = \prod_{j=1}^\Nb\prod_{p\in\calM_{\bl^{(j)}_1}}
        (\blamb{}{{\bbm}^{(j)},p}) !.
    \end{equation}
\end{lemma}

\begin{proof}
Equation~\eqref{eq:per-sym-B-pi} is a straightforward consequence of~\eqref{eq:B+-} and~\eqref{eq:WK}, using~\eqref{eq:akpm}. It remains to prove~\eqref{eq:ladder-pi}. 
By~\eqref{eq:B+-}, there holds for $K\in \calM_\infty$, and $\bbm, \;\bbm'\in\oMlK$ that
\begin{equation}
    \label{eq:B-B+}
     [\Bm_{K+1} \Bp_{K}][\bbm,\bbm'] =  
 \sum_{j,j'=1}^\Nb 
 \sum_{p,p'} 
 \blamb{}{{\bbm'}^{(j')},p'-1}\left(\blamb{}{\bbm^{(j)},p+1}+1\right)\sqrtpm{-}{\bl^{(j)}_1,p}\sqrtpm{+}{\bl^{(j')}_1,p'}
 \delta_{\mjpp,(\bbm')^{+}_{j',p' - 1}},
\end{equation}
\begin{equation}
    \label{eq:B+B-}
    [\Bp_{K-1} \Bm_{K}][\bbm,\bbm']
    = 
 \sum_{j,j'=1}^\Nb 
 \sum_{p,p'} 
 \blamb{}{{\bbm'}^{(j)},p+1}\left(\blamb{}{\bbm^{(j')},p'-1}+1\right)\sqrtpm{+}{\bl^{(j')}_1,p'}\sqrtpm{-}{\bl^{(j)}_1,p}
 \delta_{(\bbm')^{-}_{j,p+1},\bbm^{-}_{j',p'}},
\end{equation}
where in the above sums, $p\in\calM_{\bl^{(j)}_1} \backslash \{ \bl^{(j)}_1 \}$, $p'\in\calM_{\bl^{(j')}_1} \backslash \{ -\bl^{(j')}_1 \}$ such that $\blamb{}{\bbm'^{(j')},p'-1} > 0$ and $\blamb{}{\bbm^{(j)},p} > 0$
 for~\eqref{eq:B-B+}, and $\blamb{}{\bbm'^{(j)},p+1} > 0$ and 
$\blamb{}{\bbm^{(j')},p'} > 0$ for~\eqref{eq:B+B-}. 

We suppose first that $\bbm\ne\bbm'$. By~\eqref{eq:equivalence_mjpplus}, the nonzero terms in~\eqref{eq:B-B+} and~\eqref{eq:B+B-} correspond to the same indices $(j,j',p,p')$. For each such index,~\eqref{eq:lambda-shift} gives
\[
\blamb{}{\bbm'^{(j')},p'-1}
(\blamb{}{\bbm^{(j)},p+1}+1)
=
\blamb{}{\bbm'^{(j)},p+1}
(\blamb{}{\bbm^{(j')},p'-1}+1),
\]
hence the corresponding summands in~\eqref{eq:B-B+} and~\eqref{eq:B+B-} also coincide. 
Therefore, $\Bm_{K+1}\Bp_{K} - \Bp_{K-1}\Bm_{K}$ is a diagonal matrix. 

For $\bbm=\bbm'$, a direct calculation, using a change of indices on $p$ and equations~\eqref{eq:akpm} and~\eqref{eq:A-diff}, gives that
\[
    (\Bm_{K+1}\Bp_{K} - \Bp_{K-1}\Bm_{K})[\bbm,\bbm] 
    = \sum_{j=1}^\Nb \sum_{p\in\calM_{\bl^{(j)}_1}} (\blamb{}{\bbm^{(j)},p+1}-\blamb{}{\bbm^{(j)},p})(\sqrtpm{-}{\bl^{(j)}_1,p})^2 = 2\sum_{j=1}^\Nb \sum_{p\in\calM_{\bl^{(j)}_1}} p\blamb{}{\bbm^{(j)},p} = 2K.
\]
Thus, 
\begin{equation}\label{eq:ladder-pi-inner}  
    \Bm_{K+1}\Bp_{K} - \Bp_{K-1}\Bm_{K} = 2K\cdot I_{\overline{\calM}_{\bl,K}}.
\end{equation}
We complete the proof by multiplying both sides of~\eqref{eq:ladder-pi-inner} by $W_K$ on the left, and using~\eqref{eq:per-sym-B-pi}.
\end{proof}

We can then show that either half of the matrix $M_2$ is sufficient to determine its kernel.
\begin{proposition}
\label{prop:matrix_final_pi}
    For $\bl\in \calL^N$, $L\in \calLG$, there holds
    \[
        \ker(\MlL_{2}) = \ker(\Mm_{2}) = \ker(\Mp_{2}).
    \]
\end{proposition}
\begin{proof}
We show that $\big(\Mm_{2}\big)^TW^-\Mm_{2} = 
    \big(\Mp_{2}\big)^TW^+\Mp_{2}$,
where 
\[
\begin{split}
    W^- = \textup{Diag}(W_K)_{K\in\{-L-1,\ldots,L-1\}}, \\
    W^+ = \textup{Diag}(W_K)_{K\in\{-L+1,\ldots,L+1\}}.
\end{split}
\]
Clearly, $W^\pm$ are positive definite. Due to the block structure of $W^\pm$ and $\Mm_{2}$ (respectively $\Mp_{2}$), the matrix $\big(\Mm_{2}\big)^TW^-\Mm_{2}$ (resp. $\big(\Mp_{2}\big)^TW^+\Mp_{2}$) is also a block matrix with row and column indices taking values from $\calM_L$. Specifically, $\forall K,K'\in\calM_L$, noting that $\sqrtpm{-}{L,L}=\sqrtpm{+}{L,-L}=0$,
\[
    \Big[\big(\Mm_{2}\big)^TW^-\Mm_{2}\Big]_{K,K'} = 
    \begin{cases}
        \big(\Bm_K\big)^TW_{K-1}\Bm_K + \big(\Am_K\big)^TW_{K}\Am_K, \quad & \textup{if } K'=K, \\
        \big(\Bm_{K}\big)^TW_{K-1}\Am_{K-1}, \quad & \textup{if } K'=K-1, \\
        \big(\Am_{K}\big)^TW_{K}\Bm_{K+1}, \quad & \textup{if } K'=K+1, \\
        \mathbf{0}, \quad & \textup{otherwise,}
    \end{cases}
\]
and 
\[
    \Big[\big(\Mp_{2}\big)^TW^+\Mp_{2}\Big]_{K,K'} = 
    \begin{cases}
        \big(\Bp_K\big)^TW_{K+1}\Bp_K + \big(\Ap_K\big)^TW_{K}\Ap_K, \quad & \textup{if } K'=K, \\
        \big(\Ap_{K}\big)^TW_{K}\Bp_{K-1}, \quad & \textup{if } K'=K-1, \\
        \big(\Bp_{K}\big)^TW_{K+1}\Ap_{K+1}, \quad & \textup{if } K'=K+1, \\
        \mathbf{0}, \quad & \textup{otherwise.}
    \end{cases}
\]
Then, we see from equations~\eqref{eq:akpm} and~\eqref{eq:per-sym-B-pi} that their corresponding off-diagonal blocks are identical, and from equations~\eqref{eq:akpm} and~\eqref{eq:ladder-pi} that their corresponding diagonal blocks also coincide.
Thus, we have shown
$\big(\Mm_{2}\big)^TW^-\Mm_{2} = \big(\Mp_{2}\big)^TW^+\Mp_{2}$. 
As a result, for any $\bc\in\ker(\Mm_{2})$,
\[
    \|\Mp_{2}\bc\|_{W^+}^2 = \bc^T(\Mp_{2})^TW^+\Mp_{2}\bc=\bc^T(\Mm_{2})^TW^-\Mm_{2}\bc=\|\Mm_{2}\bc\|_{W^-}^2=0,
\]
which means $\bc\in\ker(\Mp_{2})$, and vice versa. Consequently, $\ker(\Mm_{2})=\ker(\Mp_{2})$. We complete the proof by using \eqref{eq:split-M-PI}.
\end{proof}

To characterize the kernel of $M_2$, we first reorder its row blocks as shown in Figure~\ref{fig:block_structure_of_M}(b), which highly simplifies the numerical computation of coupling coefficients (cf. Section~\ref{sec:num-complexity}).
We denote the rearranged matrix by $\widetilde{\MlL}_2$ which, as illustrated in Figure~\ref{fig:block_structure_of_M}(b), also
admits a natural partition into two parts
and we show below that the upper part is sufficient to determine the kernel of $M_2$.

\begin{proposition}\label{prop:ker-M2-GEPI}
Let $\bl\in\calL^N$, $L \in \calLG$ and let $\MlL_2$ be given by~\eqref{eq:M2-PI}. 
The matrix 
\begin{equation}
    \label{eq:Mup}
    \Mup=\big[\Mup_{K,K'}\big]_{K\in\calM_{L+1}\backslash \{L,-L-1\},~K'\in\calM_L},
\end{equation}
with each block being defined as
        \begin{equation*}
        \Mup_{K,K'} = \begin{cases}
            \Am_{K}, 
            \quad & \textup{if } K'=K, \\
            \Bm_{K+1}, 
            \quad & \textup{if } K'=K+1, \\
            \Bp_{L}, 
            \quad & \textup{if } K'=K-1=L, \\
            \mathbf{0}, \quad &\textup{otherwise.}
            \end{cases}
        \end{equation*}
has full rank and satisfies 
\[
    \ker(\Mup) =  \ker(\MlL_2).
\]
\end{proposition}
\begin{proof}
Using~\eqref{eq:ladder-pi} with $K=L+1$, 
we obtain 
\begin{equation*}
    \big(\Bm_{L+1}\big)^TW_{L}\Bm_{L+1} = \big(\Bp_{L+1}\big)^TW_{L+2}\Bp_{L+1} + 2(L+1)\cdot W_{L+1},
\end{equation*}
which 
is positive definite because $W_L$, $W_{L+1}$ and $W_{L+2}$ defined in~\eqref{eq:WK} are diagonal matrices with strictly positive elements. Hence, $\Bm_{L+1}$ has full rank. We obtain that $\Bp_{L}$ has full rank using~\eqref{eq:per-sym-B-pi}.
Thus, since the matrices $\Am_{K}$ 
(yellow blocks on Figure~\ref{fig:block_structure_of_M}(a)) are scaled identities with non-zero diagonal elements for $K\in\calM_L\setminus \{L\}$, the matrix $\Mup$ has full rank.
Since $\Mup$ consists of a subset of the rows of $\MlL_2$, we have $\ker(\MlL_2)\subset\ker(\Mup)$, which combined with Proposition~\ref{prop:matrix_final_pi} leads to
\begin{equation}\label{eq:dim-bound-pi}
    \dim\big(\ker(\Mup)\big)\ge \dim\big(\ker(\MlL_2)\big)=\dim\big(\ker(\Mm_2)\big).
\end{equation}
Since $\card{\overline{\calM}_{\bl,L+1}} = \card{\overline{\calM}_{\bl,-L-1}}$, $\Mup$ and $\Mm_2$ have the same size, so that~\eqref{eq:dim-bound-pi}, together with the rank–nullity theorem, indicates that $\text{rank}(\Mup)\le\text{rank}(\Mm_2)$. However, as $\Mup$ has full rank, $\Mm_2$ must have full rank, i.e., $\text{rank}(\Mup)=\text{rank}(\Mm_2)$. Therefore
\[
    \dim\big(\ker(\Mup)\big)=\dim\big(\ker(\Mm_2)\big) = \dim\big(\ker(\MlL_2)\big).
\]
We conclude by recalling that $\ker(\MlL_2)\subset\ker(\Mup)$.
\end{proof}

\begin{remark}[Size of the matrix $\Mup$]\label{remark:size-M2-GEPI}
From the expression of the matrix $\Mup$ given in~\eqref{eq:Mup}, 
the number of columns of the matrix $\Mup$ is always 
$\displaystyle\sum_{K\in\calM_L}\card{\oMlK}$
while the number of rows is always 
$\displaystyle\sum_{K\in\calM_{L+1} \backslash \{ L, L+1\}}\card{\oMlK}$.
In particular, when $L=\sum\bl$, the sets $\overline{\calM}_{\bl,\pm(L+1)}$ are empty and thus $\Mup_{L+1,L}$ vanishes, leaving only the green box as shown in Figure~\ref{fig:block_structure_of_M}(c), 
and when $L>\sum\bl$, all the sets $\oMlK$ with $|K|>\sum\bl$ are empty, leaving only a square matrix (the blue box in Figure~\ref{fig:block_structure_of_M}(c)).
Hence, depending on the relative ordering of $L$ and $\sum\bl$, 
\begin{equation*}
    \displaystyle\sum_{K\in\calM_{L+1} \backslash \{ L, L+1\}}\card{\oMlK} = 
    \begin{cases}
    \displaystyle {\sum_{K\in\calM_L}\card{\oMlK}} + \card{\overline{\calM}_{\bl,L+1}}-\card{\overline{\calM}_{\bl,L}}, & \quad \textup{if }L<\sum\bl, \\
    \displaystyle \sum_{K\in\calM_L}\card{\oMlK}-\card{\overline{\calM}_{\bl,L}}, & \quad \textup{if }L=\sum\bl, \\
    \displaystyle \sum_{K\in\calM_L}\card{\oMlK}, & \quad \textup{if }L>\sum\bl.
\end{cases}
\end{equation*}
\end{remark}
We can now characterize the GE-PI space $\bV^{\bl,L}$, including its dimensionality and a basis.
\begin{theorem}\label{thm:gepi-su2}
Let $\bl\in\calL^N$, $L \in \calLG$, $\bar\Phi^{\bl,L}=\left(\bphi{l}{\bbm}\right)_{\bbm\in\overline{\calM}_{\bl,\le L}}$. Let $C^{\bl,L}_i=\left(c_{\bbm,i}^{\bl,L}\delta_{K,\;\sum\bbm}\right)_{{K\in\calM_L,\; \bbm\in\overline{\calM}_{\bl,\le L}}}$ with
\[
(c_{\bbm,i}^{\bl,L})_{\bbm\in\overline{\calM}_{\bl,\le L}},\quad i=1,\ldots,\dim\big(\ker(\Mup)\big)
\]
a basis of $\ker(\Mup)$, where $\Mup$ is given in~\eqref{eq:Mup}.
Then a basis of $\bV^{\bl,L}$ is given by 
\begin{equation*}
     b^{\bl,L}_i = 
    C^{\bl,L}_i\bar\Phi^{\bl,L},\quad i=1,\ldots,\dim\big(\ker(\Mup)\big).
\end{equation*}
In addition,
\begin{equation}\label{eq:dim_su2_pi}
    \dimpi{\bl}{L}=\dim\big(\ker(\Mup)\big)=\card{\overline{\calM}_{\bl,L}} - \card{\overline{\calM}_{\bl,L+1}}.
\end{equation}
\end{theorem}
\begin{proof}
This is a direct result of Theorem~\ref{thm:GEPI_basis_psi} and Propositions~\ref{prop:ker-M1-GEPI} and~\ref{prop:ker-M2-GEPI}. Specifically, $\dimpi{\bl}{L}$ can be determined by using the fact that $\Mup$ is full rank, whose size is given in Remark~\ref{remark:size-M2-GEPI}.
\end{proof}
\begin{remark}
For $\bl$'s that
consist of $\Nat$ copies of an identical value $l\in\calL$,
the dimensionality formula~\eqref{eq:dim_su2_pi} 
coincides with the classical Cayley--Sylvester formula for symmetric powers~\cite[Theorem~3.3.4, p.~52; Exercise~3.3.6(1), p.~53]{Springer1977InvariantTheory} 
and appears in~\cite{Zhao2005-zx}. It is here derived directly using the kernel structure of the coupling system for general $\bl\in\calL^\Nat$. To the best of our knowledge, the result for general $\bl\in\calL^\Nat$ has not previously been given. 
\end{remark}

As indicated in Remark~\ref{rem:gegepi}, the GE space can be identified with a special case of the GE-PI space when the components of $\bl$ are pairwise distinct, in which case the equivalence class of every $\bbm\in\Ml$ is $\obm=\{\bbm\}$ and we can therefore identify $\oMlK$ with $\calM_{\bl,K}$ and define $\calM_{\bl,\le L}$ analogously to $\overline{\calM}_{\bl,\le L}$. The following result is then an immediate corollary of Theorem~\ref{thm:gepi-su2}. 

\begin{corollary}\label{prop:ge-su2}
Let $\bl\in\calL^\Nat$ have pairwise distinct components and let
$L\in\calLG$. Let $\Phi^{\bl,L}=\left(\phi^{\bl}_{\bbm}\right)_{\bbm\in\calM_{\bl,\le L}}$
and $C^{\bl,L}_i=\left(c_{\bbm,i}^{\bl,L}\delta_{K,\;\sum\bbm}\right)_{{K\in\calM_L,\; \bbm\in\calM_{\bl,\le L}}}$ with
\[
(c_{\bbm,i}^{\bl,L})_{\bbm\in\calM_{\bl,\le L}},\quad i=1,\ldots,\dim\big(\ker(\Mup)\big)
\]
a basis of $\ker(\Mup)$, where $\Mup$ is given in~\eqref{eq:Mup}.
Then a basis of $V^{\bl,L}$
is given by 
\begin{equation}\label{eq:GE-basis-su2}
     b^{\bl,L}_i = 
    C^{\bl,L}_i\Phi^{\bl,L},\quad i=1,\ldots,\dim\big(\ker(\Mup)\big).
\end{equation}
In addition, 
\begin{equation}\label{eq:dim_su2}
    \dimge{\bl}{L}=\dim\big(\ker(\Mup)\big)=\card{\calM_{\bl,L}} - \card{\calM_{\bl,L+1}}.
\end{equation}
\end{corollary}

\subsection{Dimensionality}

We now turn to investigating the dimensionality of the GE and GE-PI spaces, with exact formulas when possible and asymptotics. We start with an example where the dimensionality is explicit.
\begin{example}
    In the specific case where $\bl=(1,\ldots,1)\in\N^N$ or $\bl=\left((\bn_i,\boldsymbol{\ell}_i)\right)_{i=1}^{\Nat}\in\calL^N$ with $\boldsymbol{\ell}=(1,\ldots,1)\in\N^N,\; \bn=(n,\ldots,n)\in\N^N$, we have for $L\in\calLG$ that 
    \[
        \card{\overline\calM_{\bl,L}} =\left\{
    \begin{array}{ll}
        \lfloor\frac{\Nat-L}{2}\rfloor+1 & \textup{if } 0\le L\le \Nat \\
        0 & \textup{otherwise},
    \end{array}
    \right.
    \]
    and as a consequence, 
    \[
        \dimpi{\bl}{L} = \left\{
    \begin{array}{ll}
        1, & \text{if } L\le \Nat \text{ and } \Nat+L\text{ is even}, \\
        0, & \text{otherwise}.
    \end{array}
\right.
    \]
Indeed, the elements $\bbm\in\oMl$ are in one-to-one correspondence with their count vectors $\blamb{}{\bbm}=(\blamb{-1}{\bbm},\blamb{0}{\bbm},\blamb{1}{\bbm})$ which are the triplets of non-negative integers satisfying the two equations 
\[
    \blamb{-1}{\bbm}+\blamb{0}{\bbm}+\blamb{1}{\bbm}=\Nat\qquad \text{and} \qquad \blamb{1}{\bbm}-\blamb{-1}{\bbm}= L.
\]
Therefore a straightforward combinatorial proof on the number of count vectors $\blamb{}{\bbm}$ shows the formula for $\card{\overline\calM_{\bl,L}}$ and the dimensionality of the GE-PI space is then a direct consequence of~\eqref{eq:dim_su2_pi}.
\end{example}
We can also obtain an explicit expression of the dimensionality of the GE space using~\eqref{eq:dim_su2}.
\begin{proposition}[Exact GE dimensionality]
    For $\bl\in\calL^N$ and $L\in\calLG$, we have
    \begin{equation}\label{eq:explicit_dim_GE}
        \dimge{\bl}{L}=\sum_{{\mathbf S}\subseteq \{1,2,\ldots,N\}}
        (-1)^{\left( 1+\sum \card{{\mathbf S}} \right)} 
        \;\binom{L+\sum\bl-\sum_{s\in{\mathbf S}}(2l_s+1)+N-1 }{N-2}.
    \end{equation}
\end{proposition}
\begin{proof} By definition of $\calM_{\bl,L}$, $\card{\calM_{\bl,L}}$ is the coefficients of the term $x^L$ in the polynomial
$
    \prod_{i=1}^N \left(\sum_{m=-l_i}^{l_i} x^m \right).
$
Therefore, by the binomial theorem,
\begin{equation*}
    \card{\calM_{\bl,L}}=\sum_{{\mathbf S}\subseteq \{1,2,\ldots,N\}}(-1)^{\sum \card{{\mathbf S}}} \binom{L+\sum\bl-\sum_{s\in{\mathbf S}}(2l_s+1)+N-1}{ N-1}.
\end{equation*}
We conclude using~\eqref{eq:dim_su2} and the binomial formula $\binom{n}{k}+\binom{n}{k+1}=\binom{n+1}{k+1}$.
\end{proof}
\begin{remark}[Link with~\cite{Dusson2022-ie}]
\label{rem:ace}
A result on the dimensionality of the space $V^{\bl,0}$ for $G=SO(3)$ is provided in~\cite[Proposition 5]{Dusson2022-ie} as the cardinality of the set
\[
\bigg\{ \bL = (L_2,L_3,\ldots,L_N)\in \N^{N-1}, 
      \quad 
      |l_1-l_2| \le L_2 \le l_1+l_2, \\
       \forall \; 3\le i\le N, \;
      |L_{i-1}-l_i| \le L_i \le {L_{i-1}+l_i}, L_N =0
      \bigg\}.
\]
We note that~\eqref{eq:explicit_dim_GE} is more explicit and more general than what was obtained in~\cite{Dusson2022-ie}.
\end{remark}
\begin{example}\label{exam:dim_cg_su2}
As an immediate result of~\eqref{eq:explicit_dim_GE}, the dimensionality of the Clebsch--Gordan coefficients for $SU(2)$ and $SO(3)$ can at most be 1.
Indeed let $L,L_1,L_2\in\calL_G$, then 
\begin{equation*}
        \dimge{(L_1,L_2)}{L}= \begin{cases}
        1, & \textup{if } L+L_1+L_2\in\Z, \; |L_1-L_2|\le L\le L_1+L_2,\\
        0, & \textup{otherwise}.
    \end{cases}
\end{equation*}
\end{example}
This example allows us to provide the following Corollary on the dimensionality of the GE-PI space. 
\begin{corollary}
    Let $N,N_1,N_2\in \N$ with $N_1+N_2=N$.
    Let $\bl\in\calL^\Nat$ and $L\in\calL_{SO(3)}$. 
    Suppose $\bl = (\bl^{(1)}, \bl^{(2)})\in\calL^{\Nat_1}\times\calL^{\Nat_2}$ with $\bl^{(1)} \cap \bl^{(2)} = \varnothing$, then
    \begin{equation}\label{eq:rec-su2-gepi}
        \dimpi{\bl}{L} = \sum_{L_1=0}^{\sum\bl^{(1)}}\;\sum_{L_2=|L-L_1|}^{\min\{\sum\bl^{(2)},L_1+L\}}\dimpi{\bl^{(1)}}{L_1} \; \dimpi{\bl^{(2)}}{L_2}.
    \end{equation}
\end{corollary}
With~\eqref{eq:rec-su2-gepi}, the GE-PI dimensionalities for non-identical $\bl\in\calL^N$ can be accessed recursively via the knowledge of the dimensionalities of identical $\bl$'s. 
We provide in Appendix~\ref{app:table_of_dim} the GE-PI dimensionalities for some typical values of identical $\bl$ and $L$.
The same way, for $\bl\in\calL^N$ and $L\in\calL_{SO(3)}$, the recursive formula~\eqref{eq:dim_GEPI_rec} for the dimensions, adapted to the GE case, becomes
\begin{equation}\label{eq:dim_recursive_su2}
    \dimge{\bl}{L} = \sum_{L_1=0}^{\sum\bl^{(1)}}\;\sum_{L_2=|L-L_1|}^{\min\{\sum\bl^{(2)},L_1+L\}}\dimge{\bl^{(1)}}{L_1} \; \dimge{\bl^{(2)}}{L_2}.
\end{equation}
Note that~\eqref{eq:rec-su2-gepi} and~\eqref{eq:dim_recursive_su2} can be adapted for $SU(2)$ by summing over integers and half-integers $L_1$ and $L_2$ with the same boundaries.
We also provide dimensionalities of the GE space for different values of $\bl$ and~$L$ in Appendix~\ref{app:table_of_dim}. 

\medskip

We then turn to investigating the asymptotic dimensions of equivariant spaces, starting with the GE-PI case. 
Suppose $(\bl^{(j)})_{j=1}^{\Nb}$ is the minimal partition of $\bl$ whose $j$-th block has length $N_j$, then the set $\overline\calM_{\bl,L}$ is isometric to
\begin{equation}\label{eq:Lamb_blL}
    \{\blamb{}{}\in \Lambda_\bl:\sum_{j=1}^{\Nb}\sum_{k=-\bl^{(j)}_1}^{\bl^{(j)}_1}\blamb{(j)}{k}\cdot k = L\},
\end{equation}
which is actually the set of solutions to the equations
\begin{equation}\label{eq:ge-pi-int-part}
      \forall j=1,\ldots,\Nb,
      \quad \displaystyle \sum_{k=-\bl^{(j)}_1}^{\bl^{(j)}_1}\blamb{(j)}{k}  = N_j, \quad  \text{and} \quad
        \displaystyle \sum_{j=1}^{\Nb}\sum_{k=-\bl^{(j)}_1}^{\bl^{(j)}_1}\blamb{(j)}{k}\cdot k  = L.
\end{equation}
Hence, $\card{\overline\calM_{\bl,L}}$ is the number of solutions to~\eqref{eq:ge-pi-int-part}. 
In particular, for $\Nb=1$,
$\card{\overline\calM_{\bl,L}}$ is the coefficient of the term $x^{L+Nl}$ in the Gaussian binomial coefficient 
$
   \displaystyle \left[ \begin{matrix} N+2l \\ 2l \end{matrix} \right]_x = 
        \prod_{i=1}^{2l}\frac{1-x^{N+i}}{1-x^{i}},\quad x\neq 1.
$
With the properties of the Gaussian binomial coefficient being well-studied, an asymptotic estimate can be obtained for identical $\bl$'s. Indeed, denoting 
\[
    \textup{Var}_\bl = \frac{Nl(N+2l+1)}{6}, 
\]
we obtain using~\cite[(45)]{takacs1986asymptotic} that
\[
    \lim_{\Nat\to+\infty} \left(\frac{\card{\overline\calM_{\bl,L}}}{\binom{N+2l }{ 2l }} - \frac{\exp{(-L^2/(2\textup{Var}_\bl))}}{\sqrt{2\pi \textup{Var}_\bl}}\right) = 0,
\]
which yields
\[
  \dimpi{\bl}{L} = \binom{N+2l }{ 2l }\left(\frac{2L+1}{2\sqrt{2\pi} (\textup{Var}_\bl)^{3/2}}+o\left(\frac{1}{\Nat^{3/2}}\right)\right).
\]
Numerically, we observe a faster decay in the remainder of this estimation, which we present as the following conjecture, as its proof goes beyond the scope of this manuscript.
Note that the estimation below seems stronger than the results presented in a recent paper~\cite{melczer2020counting}, which require that both $l$ and $N$ go to infinity.
\begin{conjecture}[Asymptotic dimensionality for identical $\bl$]\label{conj:asym-dim-gepi-1block}
    Let $\bl = (l,\ldots,l)\in\calL^N$ containing identical nonzero values $l$, $L\in\calL_G$ such that $L+\sum\bl\in\Z$. 
    Then for $L\ll\sqrt{\textup{Var}_\bl}$,
    \begin{equation*}
                \dimpi{\bl}{L} = \binom{N+2l }{ 2l }\left(\frac{2L+1}{2\sqrt{2\pi} (\textup{Var}_\bl)^{3/2}}+O\left(\frac{1}{\Nat^{5/2}}\right)\right).
    \end{equation*}
\end{conjecture}
When extending to the multi-block case, we have that for general $\bl$, $\card{\overline\calM_{\bl,L}}$ is the coefficients of the term $x^{L+\sum_{j=1}^\Nb N_{j}\bl^{(j)}_1}$ in the following
\[
    \prod_{j=1}^\Nb\left[ \begin{matrix} N_{j}+2\bl^{(j)}_1 \\ 2\bl^{(j)}_1 \end{matrix} \right]_x = 
    \prod_{j=1}^\Nb\prod_{i=1}^{2\bl^{(j)}_1}\frac{1-x^{N_{j}+i}}{1-x^i},\quad x\neq 1.
\]
In addition, the sums of the blocks can be considered as independent variables, which lead to the following estimation with the increase of $\Nb$. 
\begin{proposition}[Asymptotic dimensionality for general $\bl$]\label{prop:asym-dim-gepi}
    For $\bl\in\calL^N$, $L\in\calLG$, $L+\sum\bl\in\Z$, where $\bl = (\bl^{(j)})_{j=1}^\Nb$ is the minimal partition of $\bl$ 
    with $N_j$ being the length of its $j$-th block, define
    \[
        \textup{Var}_\bl = \sum_{j=1}^\Nb \frac{N_{j}\bl^{(j)}_1(N_j+2\bl^{(j)}_1+1)}{6}.
    \]
    If there exist $\kappa_{\min},\kappa_{\max}> 0$ independent of $\Nb$ such that
    \begin{equation}\label{eq:bounded-ll-pi}
    \forall j=1,\ldots,\Nb, \quad \kappa_{\min} \le N_j\bl^{(j)}_1\le \kappa_{\max},\;
    \end{equation}
    then for $L\ll\sqrt{\textup{Var}_\bl}$, there holds for $\Nb$ sufficiently large
    \begin{equation}
        \label{eq:dimpi-est}
        \dimpi{\bl}{L} = \prod_{j=1}^\Nb \binom{N_j+2\bl^{(j)}_1 }{ 2\bl^{(j)}_1}
        \left(\frac{2L+1}{2\sqrt{2\pi}(\textup{Var}_\bl)^{3/2}}+O\left(\frac{1}{\Nb^{5/2}}\right)\right).
    \end{equation}
\end{proposition}
\begin{proof}
By definition, $\card{\overline\calM_{\bl,L}}$ is the number of $\obm\in\oMl$ such that $\sum\obm=L$. Using~\cite[(42),(43)]{takacs1986asymptotic}, 
$\sum\obm$ is a random variable
with variance $\textup{Var}_{\bl}$
and fourth order cumulant 
\[
    C_{4\bl} = -\sum_{j=1}^\Nb \frac{N_j \bl^{(j)}_1 (N_j+2\bl^{(j)}_1+1)(N_j^2+2N_j
    \bl^{(j)}_1+N_j+4{\bl^{(j)}_1}^2+2\bl^{(j)}_1)}{60}.
\]
Under the condition~\eqref{eq:bounded-ll-pi}, we have that $O(\textup{Var}_{\bl})=O(\textup{C}_{4\bl})=O(\Nb)$ for $\Nb$ sufficiently large. 
As a consequence of~\cite[Ch.~VII, Theorem~12]{petrov1975-si}, we have
\begin{equation*}
    P
    \left(\sum\obm=L{\Big|}\obm\in\oMl
    \right) = 
    \frac{1}{\sqrt{2\pi \textup{Var}_{\bl}}}
    \exp\left(\frac{-L^2}{2\textup{Var}_{\bl}}\right)
        \Big(1 + \frac{\textup{C}_{4\bl}}{24 \textup{Var}_{\bl}^{2}}H_4
        \left(\frac{L}{\sqrt{\textup{Var}_{\bl}}}\right) + O\left(\frac{1}{\Nb^{2}}\right)\Big),
\end{equation*}
where $H_4(x) = x^4-6x^2+3$ stands for the Hermite polynomial of order $4$.
Since
$    \card{\oMl}=\prod_{j=1}^\Nb \binom{N_j+2\bl^{(j)}_1 }{ 2\bl^{(j)}_1}, $
there holds
\begin{equation*}
    \card{\overline\calM_{\bl,L}} = \prod_{j=1}^\Nb \binom{N_j+2\bl^{(j)}_1 }{ 2\bl^{(j)}_1}
    \frac{1}{\sqrt{2\pi \textup{Var}_{\bl}}}
    \exp\left(\frac{-L^2}{2\textup{Var}_{\bl}}\right)
        \left(1 + \frac{\textup{C}_{4\bl}}{24 \textup{Var}_{\bl}^{2}}H_4\left(\frac{L}{\sqrt{\textup{Var}_{\bl}}}\right) + O\left(\frac{1}{\Nb^{2}}\right)\right).
\end{equation*}
Noting that for $L\ll \sqrt{\textup{Var}_{\bl}}$
\[
    H_4 \left(\frac{L}{\sqrt{\textup{Var}_{\bl}}} \right) - H_4 \left(\frac{L+1}{\sqrt{\textup{Var}_{\bl}}}\right) = O\left(\frac{1}{\Nb}\right),
\]
and 
\[
    \exp\left(\frac{-L^2}{2\textup{Var}_{\bl}}\right) -  \exp\left(\frac{-(L+1)^2}{2\textup{Var}_{\bl}}\right)
    = 
    \frac{2L+1}{2\textup{Var}_{\bl}} + O\left(\frac{1}{\Nb^2}\right), 
\]
we easily obtain~\eqref{eq:dimpi-est}. 
\end{proof}
As Conjecture~\ref{conj:asym-dim-gepi-1block} and Proposition~\ref{prop:asym-dim-gepi} suggest, when the order of equivariance $L$ is not too large, we can estimate $\dimpi{\bl}{L}$ by 
\begin{equation}\label{eq:est-gepi}
\dimpiest{\bl}{L} = \prod_{j=1}^\Nb \binom{N_j+2\bl^{(j)}_1 }{ 2\bl^{(j)}_1}\left(\frac{2L+1}{2\sqrt{2\pi}(\textup{Var}_\bl)^{3/2}}\right), 
\end{equation}
in the asymptotic regime. In addition, it suggests that the dimensionalities of the GE-PI spaces scale linearly with respect to $L$, for moderate $L$.
Recovering the GE case by imposing blocks of length $1$, that is, $N_j=1$ for all $j\in\{1,\ldots,\Nb\}$ and $\Nb=\Nat$, we obtain the asymptotic estimate for the GE case as
\begin{equation}
    \label{eq:dimge-est}
    \dimge{\bl}{L} = \prod_{i=1}^N\big(2l_i+1\big)
    \left(\frac{2L+1}{2\sqrt{2\pi}(\textup{Var}_\bl)^{3/2}}+O\left(\frac{1}{\Nat^{5/2}}\right)\right),
\end{equation}
and can thus estimate $\dimge{\bl}{L}$ for moderate $L$ by
\begin{equation}\label{eq:est-ge}
\dimgeest{\bl}{L} = \frac{(2L+1)\big(\prod_{i=1}^N(2l_i+1)\big)}{2\sqrt{2\pi}(\textup{Var}_\bl)^{3/2}}.
\end{equation}

\subsection{Extension to $O(3)$}

In a variety of physical implementations, the $O(3)$ symmetry is of greater interest than the $SO(3)$ one described above, i.e., considering not only rotation equivariance, but also reflection equivariance. Whereas $O(3)$ is not a connected group, our approach can still be applied to its two connected components separately. 

First, since $O(3)=\{\pm1\}\times SO(3)$, each element $R\in O(3)$ can be written as $R=sQ$ where $s:=\det(R)\in\{\pm1\}$ and $Q\in SO(3)$.
The irreducible representations of $O(3)$ are then defined for $R=sQ\in O(3)$ as
\begin{equation}
\label{eq:Dlpm}
        D^{\ell,+}(R) = D^\ell(Q),
    \qquad
    D^{\ell,-}(R) = sD^\ell(Q),
\end{equation}
for $\ell\in\N_0$, where $D^\ell$ stands for the Wigner-D matrix of degree $\ell$ (see Definition~\ref{def:wigner-D}). 
Given $l\in\calL$ with component $\ell\in\N_0$, all the compatible bases $\Phi^l=\{\phi^l_m\}_{m\in\calM_l}$ mentioned in Section~\ref{sec:compatible_bases} corresponding to the representation $D^\ell$ satisfy
\begin{equation}\label{eq:equiv-1p-basis-O3}
    \Phi^l(R\cdot\br) = \Phi^l(sQ\cdot\br) = s^\ell\Phi^l(Q\cdot\br)=s^\ell D^\ell(Q)\Phi^l(\br).
\end{equation}
Therefore, they are also compatible bases of the group $O(3)$ for
the representation $D^{\ell,+}$ when $\ell$ is even, and $D^{\ell,-}$ when $\ell$ is odd. While other compatible bases may exist, we restrict ourselves to these specific bases. 

The GE functions in $\left[V^\bl\right]^{2L+1}$ and the GE-PI functions in $\left[\bV^\bl\right]^{2L+1}$ for the representation $D^{L,\pm}$ are strongly connected to the equivariant functions for the representation $D^{L}$. For notation simplicity, we denote by $V^{\bl,L,\pm}$ the space for GE functions in $\left[V^\bl\right]^{2L+1}$ corresponding to $D^{L,\pm}$. Likewise, we write $\bV^{\bl,L,\pm}$ for the spaces of the GE-PI functions in $\left[\bV^\bl\right]^{2L+1}$ for the representations $D^{L,\pm}$. 

\begin{proposition}\label{eq:equiv-basis-o3}
    Let $\bl\in\N_0^\Nat$, 
    $L \in \N_0$. If $\sum\bl$ is even, then
    \[
        V^{\bl,L,+} = V^{\bl,L}, \quad \bV^{\bl,L,+} = \bV^{\bl,L}, 
    \]
    and
    \[  
        \dimge{\bl}{L,-} = \dimpi{\bl}{L,-}=0.
    \]
    Similarly, if $\sum\bl$ is odd, then
    \[
        V^{\bl,L,-} = V^{\bl,L}, \quad \bV^{\bl,L,-} = \bV^{\bl,L},
    \]
    and
    \[  
        \dimge{\bl}{L,+} = \dimpi{\bl}{L,+}=0.
    \]
\end{proposition}
\begin{proof}
We only prove the GE case with $\sum\bl$ even, as the proofs for all other cases are similar. 
First, restricting the group action to elements of $SO(3)$ and using~\eqref{eq:Dlpm}, it is clear that $V^{\bl,L,\pm}\subset V^{\bl,L}$.
Second, let  $F\in V^{\bl,L}$. Expressing $F$ in the basis $\{b^{\bl,L}_i\}_i$ with coefficients $\{f_i\}_i\subset\K$,  the equivariant property~\eqref{eq:equiv-1p-basis-O3} of the basis indicates that for $\br\in\Omega^\Nat$ and $R=sQ\in O(3)$ with $s=\pm 1$ and $Q\in SO(3)$ 
\[
     F(R\cdot\br) = \sum_i f_i \; b^{\bl,L}_{i}(R\cdot\br) = s^{\sum\bl} D^L(Q) \sum_i f_i \; b^{\bl,L}_{i}(\br) = D^L(Q) F(\br) = D^{L,+}(R) F(\br).
\]
Therefore, 
$F\in V^{\bl,L,+}$, and thus $V^{\bl,L,+}=V^{\bl,L}$. In addition, $F$ can be in $V^{\bl,L,-}$ only when $F = \mathbf{0}$. This completes the proof. 
\end{proof}

Proposition~\ref{eq:equiv-basis-o3} implies that the spaces of the GE and GE-PI functions for the $O(3)$ representations $D^{L,\pm}$, $V^{\bl,L,\pm}$ and $\bV^{\bl,L,\pm}$, are either spanned by basis functions given in Corollary~\ref{prop:ge-su2} and Theorem~\ref{thm:gepi-su2}, or are trivial zeros, depending on the parity of $\sum\bl$.

\section{Numerical results}
\label{sec:num}

In this section, we numerically implement the method described in Section~\ref{sec:equivariant_lie_algebra} on the groups $SO(3)$ and $SU(2)$ to validate our theoretical claims in the preceding Section~\ref{sec:appl}, and to showcase the efficiency of the proposed method, particularly for the GE-PI case. 
For simplicity, we only present numerical results for the group $SO(3)$ as the results for $SU(2)$ are extremely similar.
To ensure reproducibility of the results and facilitate use of the proposed method, the source code implementing the methods described in this work, as well as the code that generates all the figures in this manuscript, have been archived on Zenodo~\cite{zhang2026zenodo}. In addition to the archived repository, the proposed method has also been integrated into \href{https://github.com/ACEsuit/EquivariantTensors.jl}{\emph{EquivariantTensors.jl}}~\cite{ortner2024et}, a Julia package providing the core computational kernels and infrastructure for building equivariant tensor layers.

\subsection{Efficiency}\label{sec:efficiency}

We first illustrate the efficiency gain in obtaining the equivariant bases using our method, which is mainly due to the chosen approach for dealing with the permutations, and the full exploitation of the specific structure of the matrix $\MlL^{\bl,L}$ defined in~\eqref{eq:matM_GEPI}, specifically for $SU(2)$ and $SO(3)$. 
As Theorem~\ref{thm:gepi-su2} shows, it is sufficient to compute the kernel of $\MlL^{\textup{up}}$ given in~\eqref{eq:Mup}, whose computational cost is analyzed below.

\subsubsection{Complexity of the algorithm}\label{sec:num-complexity}

To compute the coupling coefficients, three main steps have to be performed: (1) constructing the set $\calM_\bl$ (GE case) or the set of classes $\oMl$ (GE-PI case), (2) building the matrix $\Mup$, and (3) finding its kernel. In practice, it appears that the first step is typically negligible, as it amounts to less than 1\% of the overall cost. For instance, in the GE-PI case, when $\bl = (8,8,8,8,8,8)$, $L = 20$, generating the classes takes around 0.3 $ms$, constructing the matrix $\Mup$ takes 16 $ms$, and finding the kernel of $\Mup$ takes 45 $ms$.
We therefore focus below on the complexity of the last two steps only, and for simplicity only in the GE-PI case, the GE case being similar.

First, to build $\Mup$, one needs to loop over the row indices $\bbm\in\overline\calM_{\bl,K}$ for $K\in\calM_{L+1}\backslash\{-L-1,L\}$. Noting that each row contains at most $2\sum_{j=1}^\Nb\bl^{(j)}_1+1$ nonzero entries, the whole cost of step (2) is bounded by
\begin{equation}\label{eq:cost2}
    \text{\texttt{cost}}_{\rm 2} \lesssim \left(2\sum_{j=1}^\Nb\bl^{(j)}_1+1\right)\sum_{K\in\calM_{L+1}\backslash\{L,L+1\}}\card{\overline{\calM}_{\bl,K}}.
\end{equation}

The computational cost of step (3) depends on the structure of the matrix $\Mup$. As can be seen on Figure~\ref{fig:block_structure_of_M}(c), $\Mup$ is a supertriangular block matrix, that is, a block matrix with non-zero blocks only on the main diagonal and the super-subdiagonal. Moreover, the first $2L$ diagonal blocks are scaled identities, and the last block $B_L^+$ can in fact be made upper triangular, up to some proper convention of the ordering of its rows and columns. Therefore, the complexity of obtaining the kernel of $\Mup$, using a back-substitution approach, is bounded by
\begin{equation}\label{eq:cost3}
    \text{\texttt{cost}}_{3} \lesssim 
    \dimpi{\bl}{L}
    \left(2\sum_{j=1}^\Nb\bl^{(j)}_1+1\right)
    \sum_{K\in\calM_{L+1}\backslash\{L,L+1\}}\card{\overline{\calM}_{\bl,K}}.
\end{equation}
Consequently, the total cost to generate a basis of GE-PI functions, neglecting step (1), is bounded by
\begin{equation}\label{eq:cost-all}
    \text{\texttt{cost}}^{\rm GE-PI} \lesssim 
    \left(\dimpi{\bl}{L}+1\right)
    \left(2\sum_{j=1}^\Nb\bl^{(j)}_1+1\right)
    \sum_{K\in\calM_{L+1}\backslash\{L,L+1\}}\card{\overline{\calM}_{\bl,K}}.
\end{equation}

 We provide in Figure~\ref{fig:cost_versus_classes_breakdown}  a breakdown of the computational cost as a function of the number of classes (left) for a large variety of choices for parameters $\bl= (l,\ldots,l)$ of length $\Nat$ and equivariance order $L$.  On the right of Figure~\ref{fig:cost_versus_classes_breakdown},
we plot the cost of step (3) for the same parameters as a function of the product of the number of classes and the dimensionality, from which we observe a linear complexity, as expected from~\eqref{eq:cost2} and ~\eqref{eq:cost3}. Notice that for any given $\bl$, the computational cost of our method~\eqref{eq:cost-all} is asymptotically proportional to the size of the coupling coefficients $\bc^{\bl,L}$, and is hence asymptotically optimal.

\begin{figure}[h!]
    \centering
    \includegraphics[width=.9\linewidth]{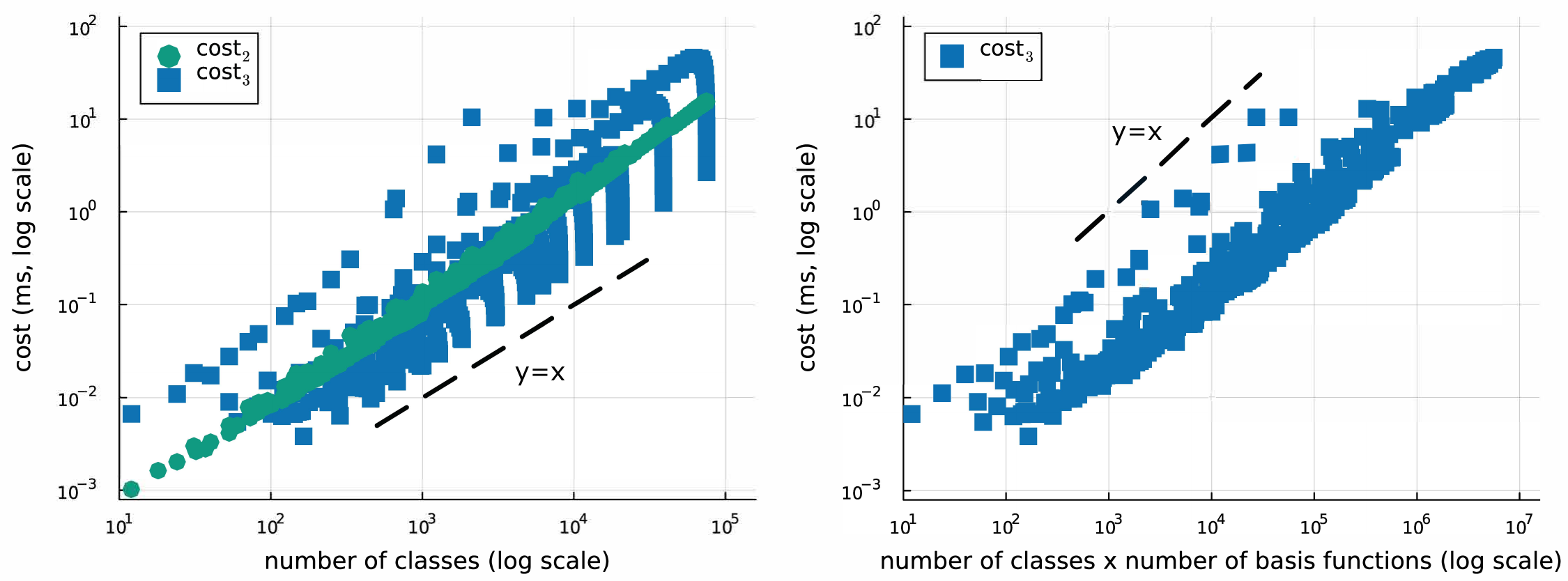}
    \caption{Breakdown of construction time (ms) of the GE-PI bases versus the number of classes in the matrix $\Mup$~\eqref{eq:M2-PI}, or the product of the number of classes and the number of basis functions.  
    }
    \label{fig:cost_versus_classes_breakdown}
\end{figure}

\subsubsection{Numerical efficiency}

Our focus in this subsection lies in the comparison of the efficiency of constructing the GE-PI spaces $\bV^{\bl,L}$ for a given $\bl$ and for arbitrary $L\in\calLG$, as the computation of GE-PI bases is in general very expensive, and more challenging than GE bases, which can also be computed using generalized Clebsch--Gordan coefficients~\cite{Dusson2022-ie}. Also, while the proposed method and our code work for general $\bl\in\calL^\Nat$, it is impractical to carry out a thorough comparison for all possible $\bl$ due to the large freedom in the parameter choices (e.g. $\bl\in\calL^N$ can be chosen to have arbitrary length, and the elements therein vary freely either). Thus, we restrict ourselves to 
$\bl=(l,l,\ldots,l)\in\calL^\Nat$, especially noting that bases for general $\bl$ can be constructed from such identical $\bl$, using~Proposition~\ref{prop:GEPI_rec}. 
Like in some other contexts~\cite{Dusson2022-ie,Batatia2022-zi}, we sometimes refer $\Nat$ to the \emph{correlation order} and $l$ the \emph{polynomial degree} if $l\in\N_0$.

\begin{figure}[b!]
    \centering
    \includegraphics[width=.9\linewidth]{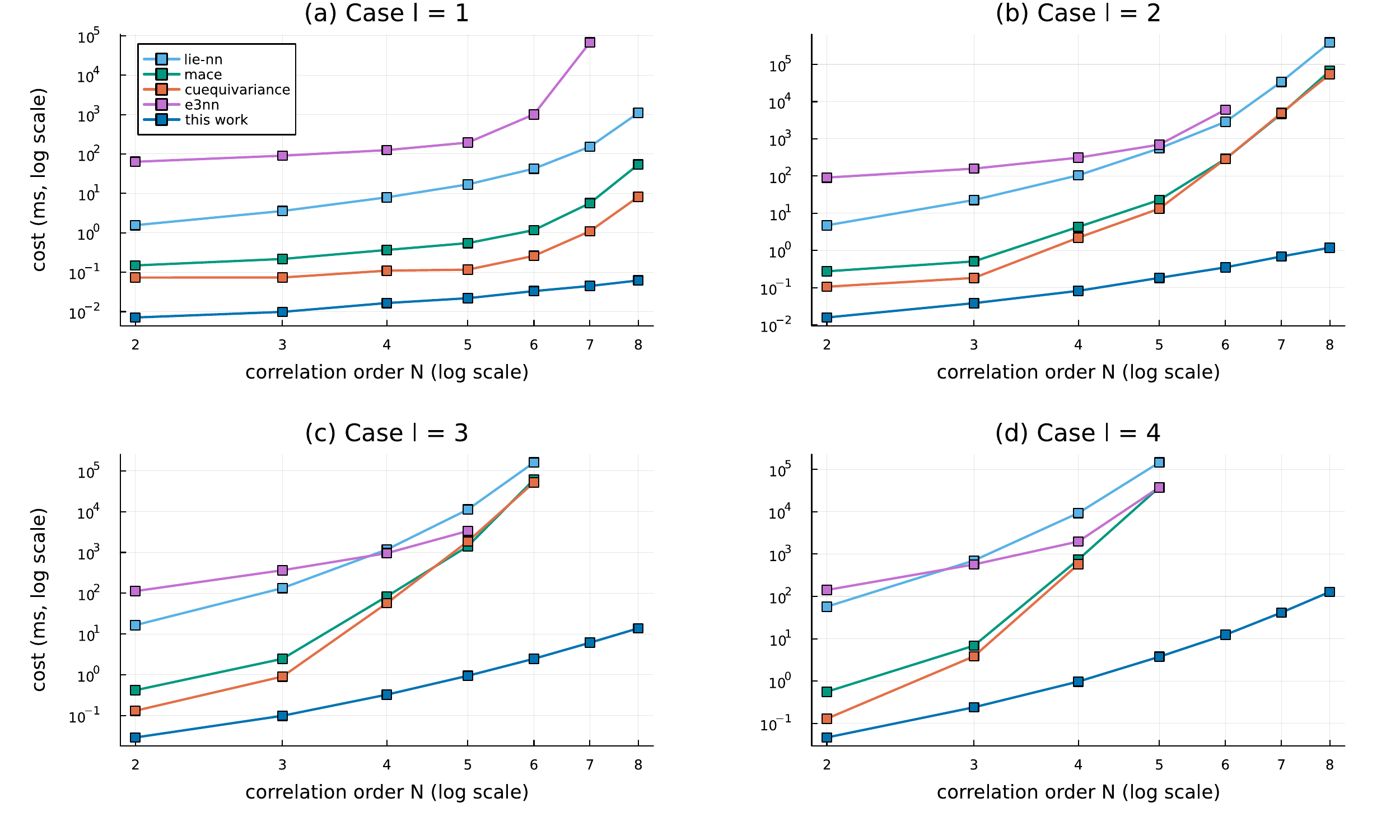}
    \caption{Run time for constructing GE-PI bases using different packages for $l=1, 2, 3, 4$ and for correlation orders $N$'s from $1$ to $8$.}
    \label{fig:cost_fixed_l}
\end{figure}

In the following, we benchmark our method against four well-known existing packages that perform similar calculations.
The first package is E3NN~\cite{geiger2022e3nn}, introduced in the context of machine-learning interatomic potentials (MLIPs), which provides a framework for generating equivariant neural networks. 
The second one is Lie-NN~\cite{Batatia2023-us},
which is theoretically presented for any compact Lie group, and tested in practice for a variety of  groups like $SO(3)$, $SU(2)$, and $SO^+(1,3)$. 
The third one is MACE~\cite{Batatia2022-zi},  a widely known foundation model being used to fit MLIPs with equivariant message passing graph neural networks.
Finally, Cuequivariance~\cite{geiger2024accelerate} is released by Nvidia  and provides CUDA-accelerated building blocks for equivariant neural networks.
All numerical experiments were performed on a MacBook Pro equipped with an Apple M3 Pro processor and 18 GB of memory, and all calculations were restricted to a single CPU thread. Our implementation was run using Julia~1.12, while the comparison packages were run using Python~3.12.4. The precise package versions are recorded in the archived code base~\cite{zhang2026zenodo}, which also contains scripts that compare the dimensions of the resulting spaces and verify that the coupling coefficients generated by the different packages span the same space after the appropriate change of basis (except for Lie-NN, which sometimes generates only a subspace thereof due to a truncation in their setting).

In Figure~\ref{fig:cost_fixed_l}, we set the polynomial degrees $l$ to be different fixed values ($l=1,2,3,4$, respectively), and show the relationship between the run time measured in $ms$ required for constructing bases of $\bV^{\bl,L},$  for $ L\in\{0,1,\ldots,Nl\}$ and the correlation order $N$. Similarly, in Figure~\ref{fig:cost_fixed_N} we fix the correlation order ($N = 3,4,5,6$) and showcase how the run time varies while increasing the polynomial degree $l$. We truncate each method when the construction time for the GE-PI function spaces $\bV^{\bl,L}$ exceeds $10^6\; ms$.

\begin{figure}[h!]
    \centering
    \includegraphics[width=.84\linewidth]{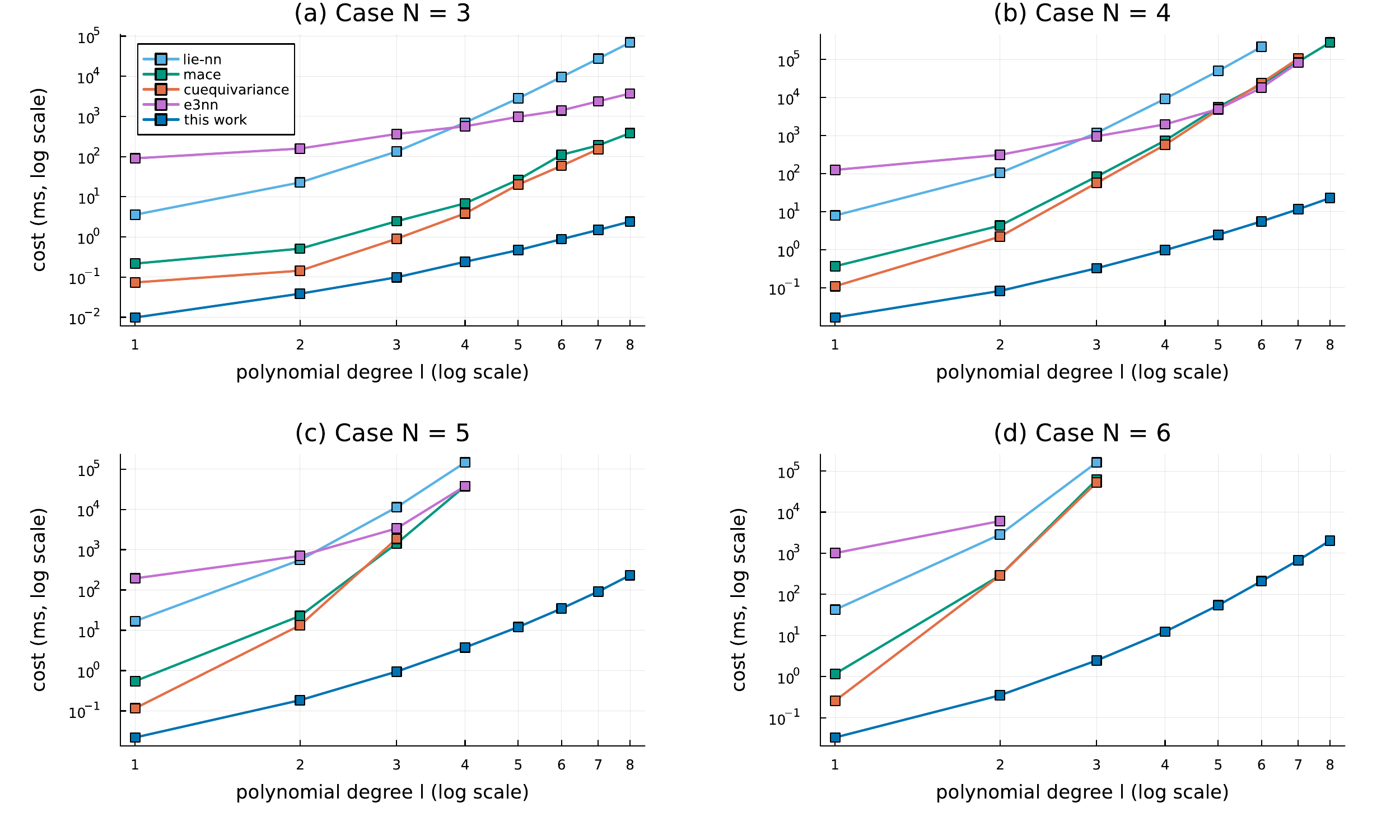}
    \caption{Run time ($ms$) for constructing the GE-PI bases using different packages for correlation orders $N = 3, 4, 5, 6$ and for polynomial degrees $l$'s from 1 to 8.}
    \label{fig:cost_fixed_N}
\end{figure}

Both Figure~\ref{fig:cost_fixed_l} and Figure~\ref{fig:cost_fixed_N} are presented on a log-log scale, from which we not only see that the construction cost of our method outperforms others in most of the cases, but also that in all the scenarios, our method appears to be 
of polynomial complexity, while other packages exhibit super-algebraic scaling.
In particular when $l=4$, none of the existing packages succeed in constructing the GE-PI basis for correlation order $N\ge 6$ within $10^6 \; ms$ while our method allows its construction in around $10^2 \; ms$. The advantages become more and more significant with the increase of either the correlation order $N$ or the polynomial degree~$l$.
We mention that despite the difference in programming languages used by our implementation and existing packages, the qualitative scaling of the construction cost remains the same.

\begin{figure}[b!]
    \centering
    \includegraphics[width=.85\linewidth]{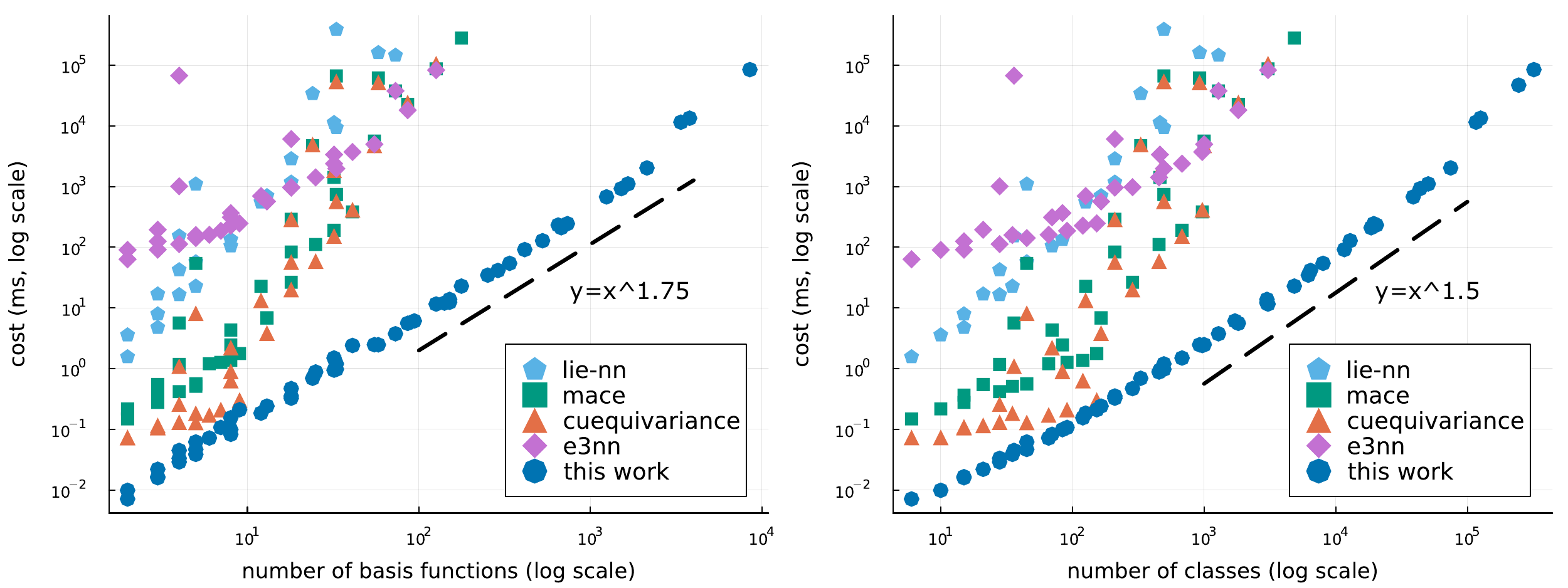}
    \caption{Construction time ($ms$) of the GE-PI bases versus the number of basis functions, or the number of classes in the matrix $\Mup$~\eqref{eq:Mup}.
    }
    \label{fig:cost_versus_classes}
\end{figure}

To see the scaling difference even more clearly, we show in Figure~\ref{fig:cost_versus_classes} the construction time as either a function of the total number of basis functions $ \sum_{L=0}^{\sum\bl}\dimpi{\bl}{L}$ or that of the number of classes
 $\sum_{K=-L-1}^{L-1}\card{\overline\calM_{\bl,K}}$.
Let us mention that we compared the cost of constructing all the equivariant bases for a given $\bl\in\calL^N$ since the tested packages provide all of these equivariant bases simultaneously. In our implementation, each individual $L$-equivariant basis can be constructed independently, allowing us to investigate the cost of each basis individually, as presented in Figure~\ref{fig:cost_versus_classes_breakdown}.
A similar conclusion can be drawn from Figure~\ref{fig:cost_versus_classes}, that is, the cost of the proposed method scales polynomially with respect to the number of basis functions, or the number of classes, while the cost of the existing implementations seems to increase exponentially.

In summary, by using the Lie algebra of the underlying rotation groups, and by exploring the simplified structure of the resulting matrix $\MlL$, we avoid the typical exponential complexity arising from permutations, and achieve a method for constructing GE-PI spaces at an almost linear cost. 

\subsubsection{Recursive construction}\label{sec:num-rec}

We now evaluate and compare the computational efficiency of the direct and recursive constructions. Due to the large parameter choice in the recursive case, we restrict ourselves to the following setting: we consider $\bl$'s that can be divided into two non-intersecting blocks, each of length $N$, which gives a total of $2N$ variables. The $\ell$ values in both blocks are set to the same value. In particular, we take $N = 3$ with $\boldsymbol{\ell} = (\ell,\ell,\ell,\ell,\ell,\ell)$ and $\bn = (1,1,1,2,2,2)$, and analyze the construction time across varying polynomial degrees $\ell\in\{1,2,\ldots,8\}$ and orders of equivariance $L\in\{0,1,\ldots,N\ell\}$. The direct and the recursive construction time, measured in \emph{ms}, is illustrated in Figure~\ref{fig:cost_recursive}.
The solid lines in Figure~\ref{fig:cost_recursive}(a) represent the total computational cost of all equivariant bases for a given degree $l$, calculated using both the direct and recursive methods. 
The overall computational costs are comparable in magnitude in all cases, indicating that the recursive approach does not lead to a strict reduction in total complexity. However, a breakdown of the costs with respect to different $L$'s shows that they are distributed differently between the two methods. In particular, the recursive construction is substantially more efficient for small values of $L$, which correspond to real implementation scenarios, as illustrated in Figure~\ref{fig:cost_recursive}(b). To emphasize this effect, we single out one of the most common cases, $L=0$, shown by dashed lines in Figure~\ref{fig:cost_recursive}(a), which highlights the significant efficiency gains that the recursive method can provide in practical applications.
A deeper comparison is left for future work.

\begin{figure}[h!]
    \centering
    \includegraphics[width=.8\linewidth]{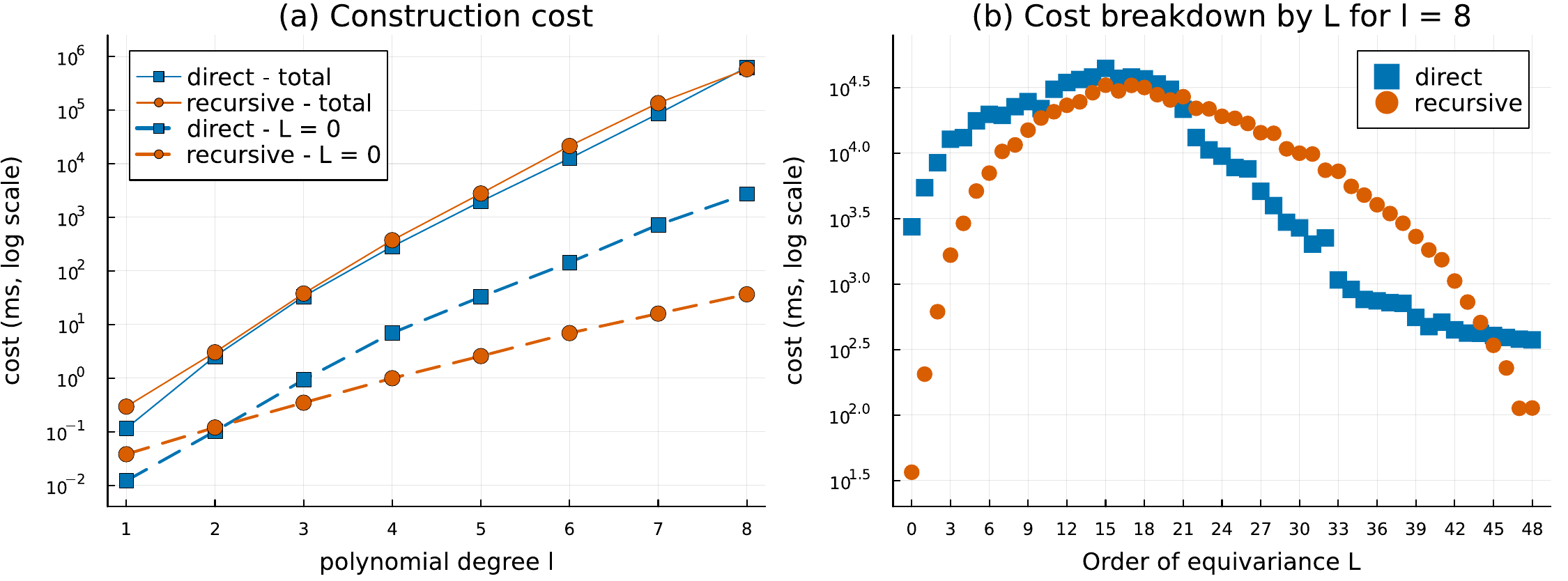}
    \caption{Comparison of recursive and direct construction time (\emph{ms}) for the GE-PI bases for $\boldsymbol{\ell} = (\ell,\ell,\ell,\ell,\ell,\ell)$ and $\bn = (1,1,1,2,2,2)$: (a) Runtime for constructing either all equivariant bases or a specific equivariant basis ($L=0$); (b) Cost breakdown by equivariance order $L$.
    }
    \label{fig:cost_recursive}
\end{figure}

\subsection{Dimensionality comparison}\label{sec:dimensionalities}
In this subsection, we investigate the dimension of GE and GE-PI spaces, or more precisely, 
the trend in the growth of the space dimensions composed of functions with different symmetries. 
Similarly to above, we focus on $\bl\in\calL^N$ that have identical components $l$, unless explicitly specified otherwise. In this setting, the dimensions $\dimge{\bl}{L}$ and $\dimpi{\bl}{L}$ are determined by three parameters: the polynomial degree $l$, the correlation order $N$ and the order of equivariance $L$.

\paragraph{Linear scaling with respect to $L$}

\begin{figure}[t!]
    \centering
    \includegraphics[width=.8\linewidth]{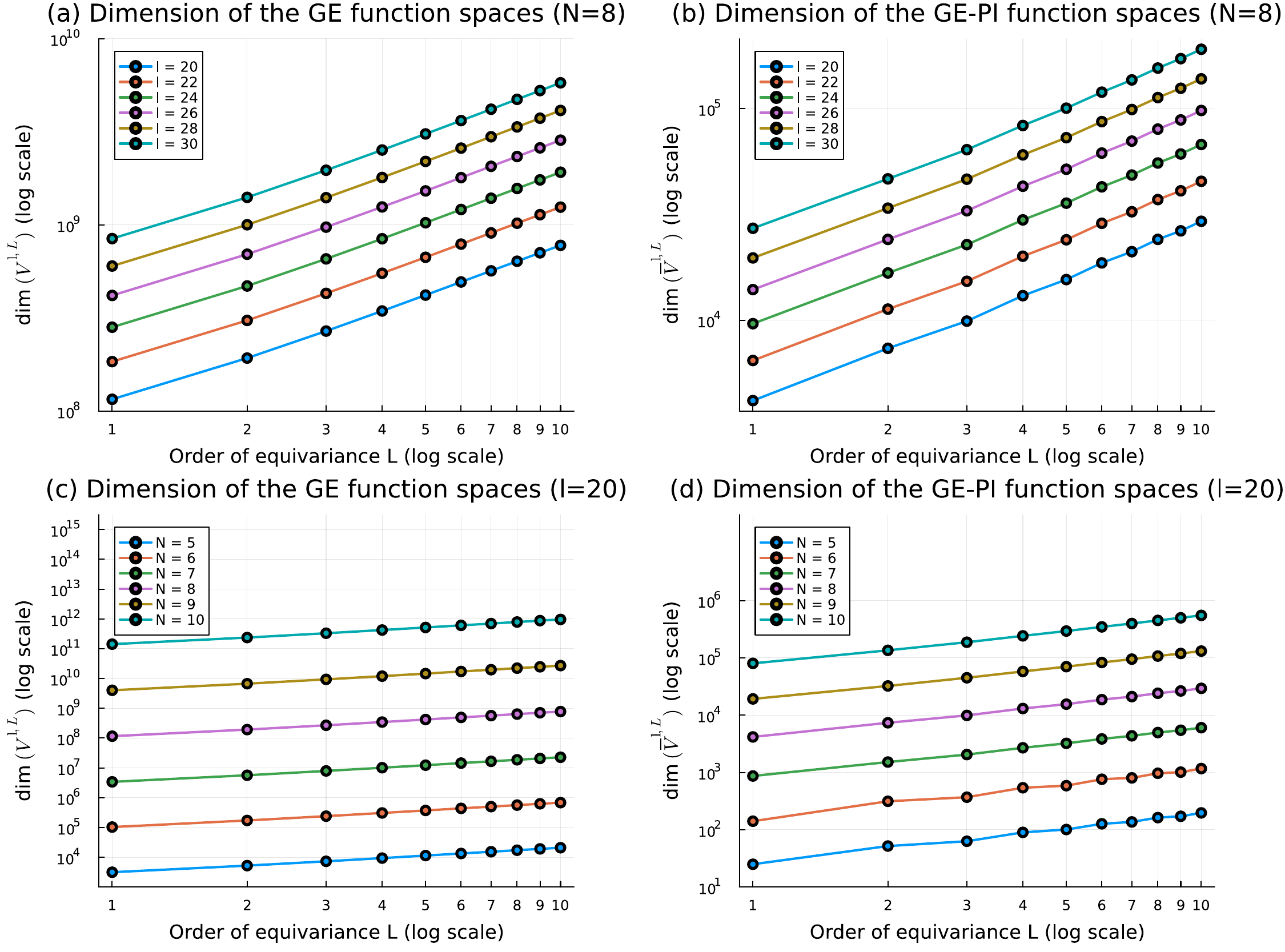}
    \caption{Relationship between the dimensions of the GE space $V^{\bl,L}$ and the GE-PI space $\bV^{\bl,L}$ and the order of equivariance $L$ with $\bl = (l,l,\ldots,l)\in\calL^N$ for varying polynomial degrees $l$ and correlation orders $N$. Both axes are displayed on a logarithmic scale. 
    }
    
    \label{fig:Linear_scale_in_L}
\end{figure}

The estimations~\eqref{eq:est-ge} and~\eqref{eq:est-gepi} suggest
that the dimensionalities of $V^{\bl,L}$ and $\bV^{\bl,L}$ scale linearly with respect to the order of equivariance $L$,
as observed on Figure~\ref{fig:Linear_scale_in_L}, 
when $L$ is not too large. In practice, this is indeed the regime of interest.
In the first two panels of Figure~\ref{fig:Linear_scale_in_L}, we fix the correlation order $N=8$, and show the change of GE and GE-PI dimensions with respect to $L$ for several different $l$, while in the last two panels, similar plots for fixed $l=20$ and a variety of $N$ are displayed. 

Owing to the linear scaling of the dimensions in $L$, we consider hereafter only  $L=0$, the invariant case, which is arguably the case of most interest. The numerical results we will present in the rest of the manuscript can naturally be extended to moderate $L$ because of the linear relation verified above.

\FloatBarrier

\paragraph{Asymptotic estimates}
Despite having explicit formulas for the GE and GE-PI dimensions (cf. equations~\eqref{eq:dim_su2} and~\eqref{eq:dim_su2_pi}), they become increasingly costly to access as $\Nat$ increases. One may rely on the estimations given in~\eqref{eq:dimge-est} and~\eqref{eq:dimpi-est} when $\Nat$ is unaffordably large, or use the estimations to access the magnitude of the dimensions more conveniently. 
To validate them, we present the normalized errors of the estimated dimensions and the exact dimensions with respect to $\Nat$,
\[
  \text{\texttt{err}}_{\rm GE} :=  \frac{\dimge{\bl}{L} - \dimgeest{\bl}{L}}{\left(2l+1\right)^\Nat}, 
  \qquad
    \text{\texttt{err}}_{\rm GE-PI} :=   \frac{\dimpi{\bl}{L} - \dimpiest{\bl}{L}}{\binom{N+2l }{ 2l }},
\]
for $\bl = (l,l,\ldots,l)\in\calL^\Nat$ for different values of $l$ in the first two panels of Figure~\ref{fig:dim-est}.
As indicated by the first two panels of Figure~\ref{fig:dim-est}, the 
 normalized errors decays at an algebraic rate of $N^{-5/2}$ across different values of $l$, as expected.

\begin{figure}[htb!]
    \centering
    \includegraphics[width=.95\linewidth]{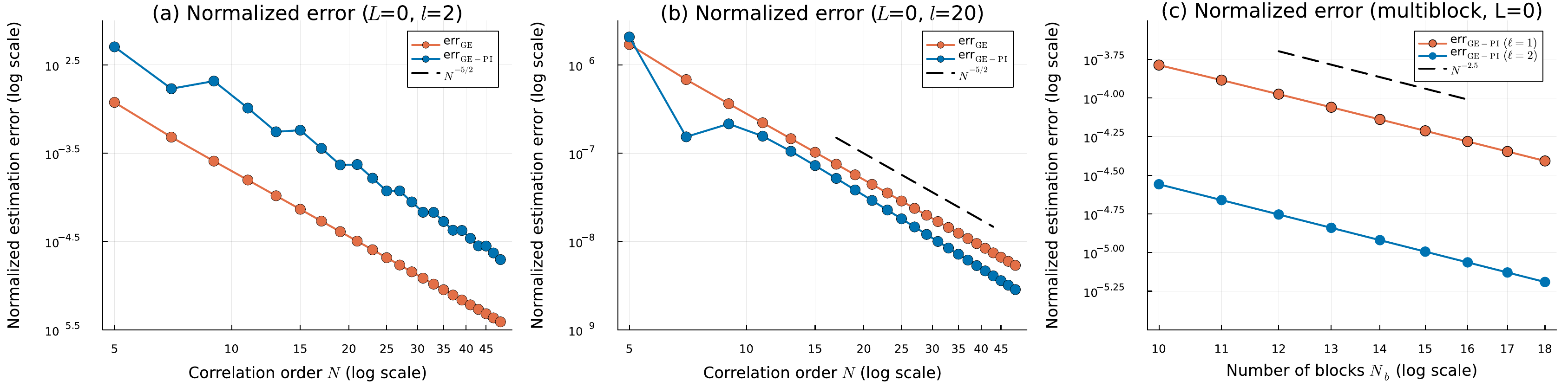}
    \caption{Normalized errors for the exact and estimated dimensions of the GE space $V^{\bl,0}$ and the GE-PI space $\bV^{\bl,0}$ for $\bl = (l,l,\ldots,l)\in\calL^\Nat$ with different $l$ and $\Nat$ (panels (a) and (b)), and for $\bl = (\bl^{(1)},\bl^{(2)},\ldots,\bl^{(\Nb)})$, where $\bl^{(j)} = ((n_j,\ell),(n_j,\ell))\in\calL^2$, with $n_j$ all distinct, and $\ell=1,2$ (panel (c)). 
    } 
    \label{fig:dim-est}
\end{figure}

Note that we have estimations in dimensionalities for general $\bl$ as well, as shown in Proposition~\ref{prop:asym-dim-gepi}. To demonstrate that the observed behaviour is not specific to the minimal setting, we present in the third panel of Figure~\ref{fig:dim-est} the results for a set of more general $\bl = (\bl^{(1)},\bl^{(2)},\ldots,\bl^{(\Nb)})$, where $\bl^{(j)} = ((n_j,\ell),(n_j,\ell))\in\calL^2$, with $n_j$ all distinct, and $\ell=1,2$
The trend suggests a similar conclusion to that of before. Up to this stage, we have all the tools to fully investigate the dimensionalities in all possible cases. 

\paragraph{GE vs GE-PI}
Finally we compare dimensions of function spaces exhibiting different symmetries. For comparison, the dimension of $V^\bl$ itself is obviously $\prod_{i=1}^N (2l_i+1)$, and hence, 
\[
    \dim(\left[V^\bl\right]^{2L+1}) = (2L+1)\prod_{i=1}^N (2l_i+1).
\]
Besides, we see from Proposition~\ref{prop:basis-pi2} that the dimensionality of the PI function space $\left[\bV^\bl\right]^{2L+1}$ is 
\[
    \dim(\left[\bV^\bl\right]^{2L+1}) = (2L+1)\prod_{j=1}^\Nb \binom{N_j+2\bl^{(j)}_1}{ 2\bl^{(j)}_1}.
\]
We are interested in comparing the two dimensions above with $\dimge{\bl}{L}$ and $\dimpi{\bl}{L}$, in both the asymptotic and pre-asymptotic regimes. Owing to the linear scaling in $L$, we fix $L=0$ as before. 
The dimensionalities of the spaces with different symmetries are shown in Figure~\ref{fig:dim_comparison},
for correlation order $N=6,\;20,\;40,\;$ and $80$ respectively, in four separate plots, in each of which the polynomial degree $l$ ranges from $4$ to $20$, covering both the asymptotic and pre-asymptotic regimes.

\begin{figure}[t!]
    \centering
    \includegraphics[width=.86\linewidth]{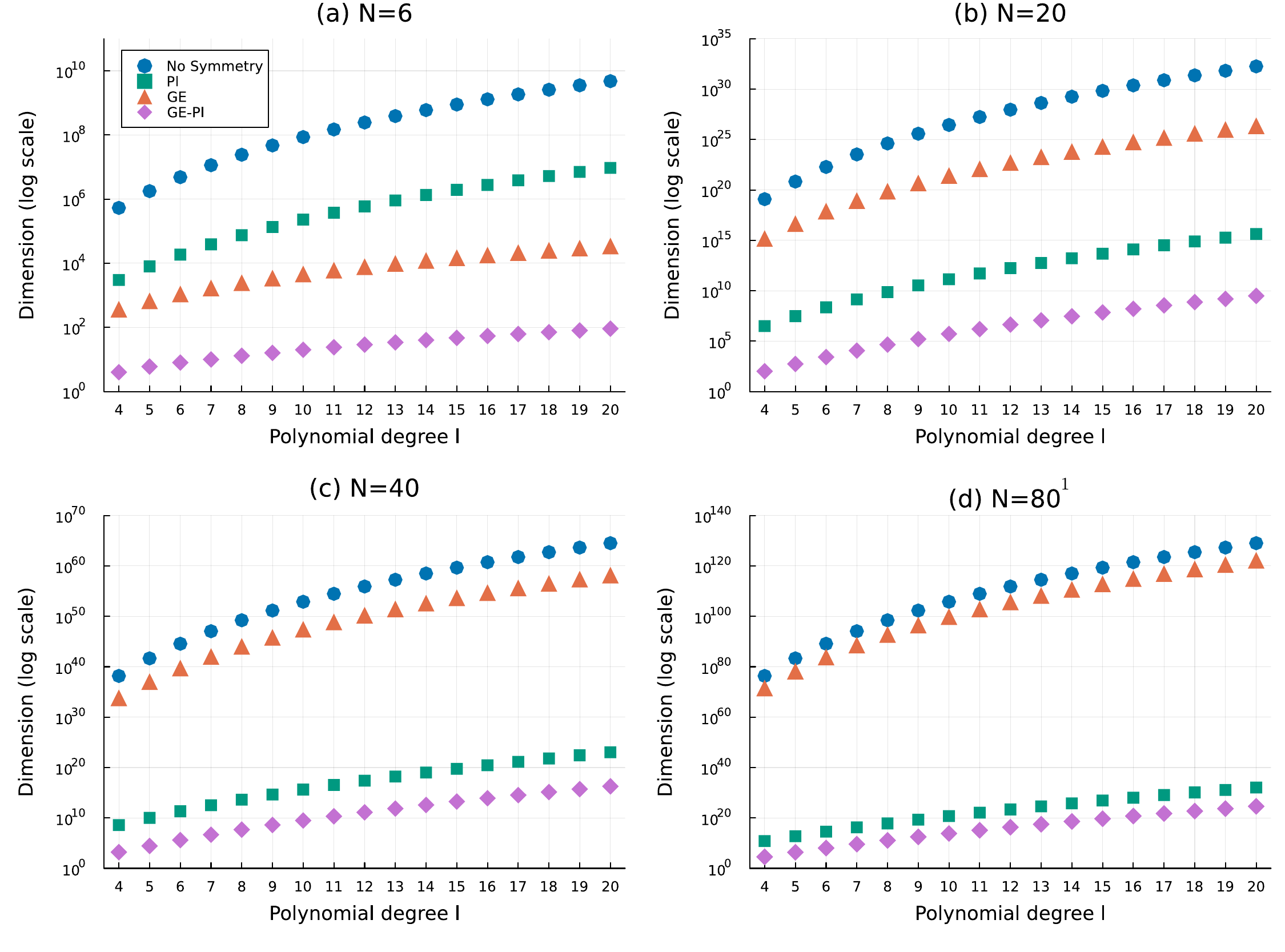}
    \caption{Dimensionality of spaces of functions in $V^{\bl}$ having different symmetries with $\bl = (l,l,\ldots,l)\in\calL^N$: (a) $N=6$; (b) $N=20$; (c) $N=40$; and (d) $N=80$. }
    \label{fig:dim_comparison}
\end{figure}

As can be seen from Figure~\ref{fig:dim_comparison}, the space without any symmetry and that with both symmetries exhibit the largest and smallest dimensions, respectively. When the correlation order is small, GE turns out to be of equal importance, if not more, than PI for dimension reduction. However, as $\Nat$ increases, the PI dimension becomes increasingly closer to the GE-PI dimension, making PI the dominating source of dimension reduction. 
Tables exhibiting precise dimensionality examples are given in Appendix~\ref{app:table_of_dim}. 
These results suggest that it is in general important for dimensionality reduction to include both permutation and group-equivariance symmetry, starting from permutation invariance.

\section{Conclusion and outlook}
\label{sec:conclusion}

In this article, we have introduced a direct numerical method for efficiently constructing group-equivariant (GE) and group-equivariant and permutation-invariant (GE-PI) spaces, valid for arbitrary linear connected Lie groups. Furthermore we have presented an alternative recursive approach valid for compact groups. These methods are based on the use of Lie algebra to derive a linear system whose kernel elements are in one-to-one correspondence with either GE or GE-PI functions. Using the specific form of the Lie algebra for $SO(3)$ and $SU(2)$, the linear system is further simplified, and can be very efficiently solved exploiting sparsity.
Indeed, our method scales almost linearly with respect to the number of basis functions, largely outperforming existing packages in the literature, which typically scale exponentially. 
On top of this, we manage to provide the explicit dimensionality of the spaces of interest. 
Comparing GE and GE-PI dimensionalities shows how important it is in practice to consider the permutation-invariance to reduce the computational cost of approximating GE-PI functions.
It is interesting to note that the structures for the linear systems are very similar in the GE and GE-PI cases, showcasing that imposing PI on top of GE using our method does not require extra computational cost; indeed the GE case can be seen as a particular case of GE-PI, in the extreme scenario that compatible bases remove the permutation-invariance constraint. 

As an immediate application, this method can be used
as a GE-PI feature generator for a broad range of existing architectures in scientific machine learning, such as those for interatomic potentials in chemistry and materials science (see e.g. the equivariant neural networks E3NN~\cite{geiger2022e3nn}, NequIP~\cite{Batzner2022-az}, and MACE~\cite{Batatia2022-zi}). 
At this point we leave for future work the detailed derivation of the linear system structure for groups beyond $SU(2)$ and $SO(3)$, such as $SO^+(1,3)$ and $SU(n)$, but we expect a similar sparsity to arise.
Another natural question is whether this work can be extended beyond permutation invariance.
It is first interesting to note that our derivation for the GE-PI case relies on starting from a known PI basis, and that starting from a basis of permutation antisymmetric functions, e.g. Slater determinants, our strategy directly applies to obtain group-equivariant and permutation anti-invariant functions.
Second, combining the work~\cite{Olver2026-tt} obtaining irreducible representations of tensor representations of the symmetric group, which is similar to what is proposed in this paper but considering the symmetric group instead of a Lie group, 
we could in principle obtain group-equivariant and permutation-equivariant functions for more general representations of the symmetric group, which we also leave for future work.

\FloatBarrier

\section*{Acknowledgements}

We would like to thank Christoph Ortner for insightful discussions. 
This work has received funding from the ANR NUMERIQ, project number ANR-24-CE46-2255. 
This work has been supported by the EIPHI Graduate school (contract ANR-17-EURE-0002) and by the Région Bourgogne Franche-Comté.
L.Z. acknowledges funding by  the Deutsche Forschungsgemeinschaft (DFG, German Research Foundation) - Project number 442047500 through the Collaborative Research Center “Sparsity and Singular Structures” (SFB 1481). 

\bibliographystyle{siam}
\bibliography{biblio}

\appendix

\section{Proof of Proposition~\ref{prop:basis-pi2}}\label{app:proof_basis_pi2}
\begin{proof}[\bf{Proof of Proposition~\ref{prop:basis-pi2}}]
First, since the functions $\bphi{\bl}{\bbm}$ are identical for all $\bbm$ in a class $\overline{\bbm} \in\oMl$
    the functions $\bphi{\bl}{\bbm},\;\bbm\in\oMl$ are clearly a spanning set of $\bV^{\bl}$.
We now need to show that the basis functions $\left\{ \bphi{\bl}{\bbm},\;\bbm\in\oMl\right\}$ are linearly independent.
Assume that
there exist coefficients $\bc \in \K^{|\oMl|}$ such that 
$\displaystyle\sum_{\bbm \in\oMl}c_\bbm\bphi{\bl}{\bbm}=0$. This means that we have for any $\br_1,\ldots,\br_N\in \Omega$,
\begin{equation}
    \label{eq:linearindependence}
    \displaystyle\sum_{\bbm \in\oMl}c_\bbm\bphi{\bl}{\bbm}(\br_1,\ldots,\br_\Nat) = 0.
\end{equation}
We are going to show by induction on $d \in \N$ that for any $\bR = (\br_1,\ldots,\br_\Nat)\in\Omega^\Nat$ and for all $1\le n_1\le\ldots\le n_\Nat\le \Nat$ such that $\card{\{n_1,\ldots,n_\Nat \}  }=d$, there holds
\begin{equation}
        \displaystyle\sum_{\bbm \in\oMl}c_\bbm \sum_{\bm{\alpha}\in \calS_{\{n_1,\ldots,n_\Nat\}}}\phi^\bl_\bbm
    (\br_{\alpha_1}, \ldots, \br_{\alpha_\Nat})=0,
    \label{eq:induction_step}
\end{equation}
where $\calS_{\{n_1,\ldots,n_\Nat\}}=\{({\sigma(n_1)},\ldots,{\sigma(n_\Nat)}),\;\sigma\in S_\Nat\}$ is the set of all different tuples $(n_1,\ldots,n_\Nat)$ in any order.

First, for $d=1$, 
for 
$\bR=(\br,\dots,\br)$, with $\br\in\Omega$,
\begin{align*}
    \displaystyle\sum_{\bbm \in\oMl}c_\bbm\bphi{\bl}{\bbm}(\br,\ldots,\br)
    &=\displaystyle\sum_{\bbm \in\oMl}c_\bbm\prod_{i=1}^\Nat\left(\sum_{j=1}^{\Nat}\phi^{l_i}_{m_i}\left(\br\right)\right)\\
    &=\displaystyle\sum_{\bbm \in\oMl}c_\bbm\prod_{i=1}^\Nat\left(\Nat\phi^{l_i}_{m_i}\left(\br\right)\right)\\
    &=\Nat^\Nat\displaystyle\sum_{\bbm \in\oMl}c_\bbm \phi^\bl_\bbm(\br,\ldots,\br).
\end{align*}
Using~\eqref{eq:linearindependence}, we obtain
\[
    \sum_{\bbm \in\oMl}c_\bbm \phi^\bl_\bbm(\br,\ldots,\br) = 0.
\]
Since $\br$ is arbitrary in $\Omega$, for $1\le n_1=\ldots=n_\Nat=n\le\Nat$,

\[
    \displaystyle\sum_{\bbm \in\oMl}c_\bbm \sum_{\bm{\alpha}\in \calS_{\{n_1,\ldots,n_\Nat\}}}\phi^\bl_\bbm
    (\br_{\alpha_1}, \ldots, \br_{\alpha_\Nat})=\displaystyle\sum_{\bbm \in\oMl}c_\bbm \phi^\bl_\bbm(\br_{n},\ldots,\br_{n})
    =0,
\]
which proves the result for $d=1$.

Now assume that~\eqref{eq:induction_step} holds for all $\bR\in\Omega^\Nat$ and all ordered $\bn\in\{1,\ldots,\Nat\}^\Nat$ having at most $d-1$ distinct values.
Let $\bk = \{k_1,\ldots,k_\Nat\}$ with $1\le k_1\le\ldots\le k_{\Nat}\le d$ such that $\card{\bk}=d$. 
For $\br_1,\ldots, \br_{d} \in \Omega$, denote $\bR_{\bk}=(\br_{k_1},\ldots,\br_{k_{\Nat}})$. Writing explicitly the repetition of variables
\[
  \bphi{\bl}{\bbm}(\bR_{\bk})=\prod_{i=1}^\Nat\left(\sum_{j=1}^{\Nat}\phi^{l_i}_{m_i}\left(\br_{k_j}\right)\right)=\prod_{i=1}^\Nat\left(\sum_{n=1}^{d}\blamb{}{\bk,n}\phi^{l_i}_{m_i}\left(\br_{n}\right)\right),
\]
where $\blamb{}{\bk} \in \N^{d}$ denotes the count vector of $\bk$ for the dictionary $\{1,\ldots,d\}$ (see~\cite[Supplementary lecture H]{Kozen2007-ot}).  
Switching the product and the sum, we obtain 
\begin{align*}
    \bphi{\bl}{\bbm}(\bR_{\bk})&=
    \displaystyle\sum_{1\le n_1,\ldots,n_\Nat\le d}\left(\prod_{i=1}^\Nat \blamb{}{\bk,n_i}\right)\left(\prod_{i=1}^\Nat\phi^{l_i}_{m_i}\left(\br_{{n_i}}\right)\right)\\
    &=
    \displaystyle\sum_{1\le n_1,\ldots,n_\Nat\le d}
    \left(\prod_{i=1}^{d} (\blamb{}{\bk,i})^{\blamb{}{\bn,i}}\right)
    \phi^\bl_\bbm\left(\br_{n_{1}},\ldots,\br_{{n_\Nat}}\right).
\end{align*}
The count vectors $\blamb{}{\bn}$ being independent of the order of the elements in $\bn$, 
we can restrict the first sum on ordered elements, and separate the sum depending on the number of different indices in $\bn$, that is,
\begin{align*}
      \bphi{\bl}{\bbm}(\bR_{\bk})&=
      \displaystyle\sum_{1\le n_1\le\ldots\le n_\Nat\le d}\left(\prod_{i=1}^{d} (\blamb{}{\bk,i})^{\blamb{}{\bn,i}}\right)\sum_{\bm{\alpha}\in \calS_{\{n_1,\ldots,n_\Nat\}}}\phi^\bl_\bbm
    (\br_{\alpha_1}, \ldots, \br_{\alpha_\Nat}) \\
    &= \displaystyle
    \sum_{d' = 1}^d \;
    \sum_{\substack{1\le n_1\le\ldots\le n_\Nat\le d\\\card{\bn}=d'}}\left(\prod_{i=1}^{d} (\blamb{}{\bk,i})^{\blamb{}{\bn,i}}\right)
    \sum_{\bm{\alpha}\in \calS_{\{n_1,\ldots,n_\Nat\}}}\phi^\bl_\bbm
    (\br_{\alpha_1}, \ldots, \br_{\alpha_\Nat}).
\end{align*}
Taking the linear combination of functions $\bphi{\bl}{\bbm}$  and using the induction step~\eqref{eq:induction_step}, the terms involving less than $d$ different values of $\br$ vanish,
\begin{align*}
    \displaystyle\sum_{\bbm \in\oMl}c_\bbm\bphi{\bl}{\bbm}(\bR_{\bk})
    &=\displaystyle\sum_{\substack{1\le n_1\le\ldots\le n_\Nat\le d\\\card{\bn}=d}}\left(\prod_{i=1}^{d} (\blamb{}{\bk,i})^{\blamb{}{\bn,i}}\right)\displaystyle\sum_{\bbm \in\oMl}c_\bbm\sum_{\bm{\alpha}\in \calS_{\{n_1,\ldots,n_\Nat\}}}\phi^\bl_\bbm
    (\br_{\alpha_1}, \ldots, \br_{\alpha_\Nat}) = 0,
\end{align*}
using~\eqref{eq:linearindependence}.
Noting that the Bernstein--Vandermonde matrix $\displaystyle\left(\prod_{i=1}^{d} (\blamb{}{\bk,i})^{\blamb{}{\bn,i}}\right)_{\substack{\text{ordered } \bn,\bk\in\{1,\ldots,\Nat\}^\Nat\\\card{\bk}=\card{\bn}=d}}$ is invertible~\cite[(2.10)]{Ainsworth2011-pk}, 
we obtain multiplying by the inverse of the matrix that, for any ordered $\bn\in\{1,\ldots,\Nat\}^\Nat$, with $\card{\bn}=d$ 
\[
    \displaystyle\sum_{\substack{1\le k_1\le\ldots\le k_\Nat\le d\\\card{\bk}=d}}\displaystyle\sum_{\bbm \in\oMl}
    c_\bbm\sum_{\bm{\alpha}\in \calS_{\{n_1,\ldots,n_\Nat\}}}\phi^\bl_\bbm
    (\br_{\alpha_1}, \ldots, \br_{\alpha_\Nat})=0.
\]
Since the terms inside the sum are independent of $\bk$, we can divide by the number of ordered $\bk\in\{1,\ldots,\Nat\}^\Nat$ having $d$ different coordinates, we finally obtain that for any $1\le n_1\le\ldots\le n_\Nat\le d$ such that $\card{\{n_1,\ldots,n_\Nat \}  }=d$ and any $\br_1,\ldots,\br_d \in \Omega$
\[
  \displaystyle\sum_{\bbm \in\oMl}c_\bbm\sum_{\bm{\alpha}\in \calS_{\{n_1,\ldots,n_\Nat\}}}\phi^\bl_\bbm
    (\br_{\alpha_1}, \ldots, \br_{\alpha_\Nat})=0,
\]
which concludes the induction, using that the labeling of the variables is arbitrary.

In particular for $d=\Nat$ and $\bn=(1,\ldots,\Nat)$, we have

\[
    \displaystyle\sum_{\bbm \in\oMl}c_\bbm  \displaystyle\sum_{\sigma\in S_\Nat}\phi^\bl_\bbm\left(\br_{\sigma_1},\ldots,\br_{\sigma_\Nat}\right)
    = \displaystyle\sum_{\bbm \in\oMl}c_\bbm \bpsi{\bl}{\bbm}(\br_{1},\ldots,\br_{\Nat})=0.
\]
From the linear independence of the $\bpsi{\bl}{\bbm}$, we conclude that the $\bphi{\bl}{\bbm}$ are linearly independent.
\end{proof}

\section{Proof of Proposition~\ref{prop:deri_alpha_gamma}}\label{app:Wigner-Dmatricesproof}

\begin{proof}[Proof of Proposition~\ref{prop:deri_alpha_gamma}]
From~\eqref{eq:link_derivatives} there holds
\[
    \vrho{\ell,1}{\mu,m} = \left[dD^\ell(X_1)\right]_{\mu,m}
    = \frac{\partial}{\partial\alpha}
    \left.D^\ell_{\mu m}\left(Q(\alpha,0,0)\right)\right|_{\alpha=0}
    = \frac{\partial}{\partial\gamma}
    \left.D^\ell_{\mu m}\left(Q(0,0,\gamma)\right)\right|_{\gamma=0},
\]
and 
\[
    \vrho{\ell,2}{\mu,m} =
    \left[dD^\ell(X_2)\right]_{\mu,m}
    = \frac{\partial}{\partial\beta}
    \left.D^\ell_{\mu m}\left(Q(0,\beta,0)\right)\right|_{\beta=0}.
\]
Taking the derivative of $D^l_{\mu m}$ with respect to $\alpha$ in~\eqref{eq:wigner-D}, we have that
\[
    \frac{\partial D^l_{\mu m}}{\partial \alpha}(\alpha, \beta, \gamma)  = -\mathrm{i}\cdot \mu e^{-\mathrm{i}\mu\alpha}d^l_{\mu m}(\beta)e^{-\mathrm{i}m\gamma}.
\]
Inserting $(\alpha,\beta,\gamma) = (0,0,0)$ and noticing that $d_{\mu m}^l(0) = \delta_{\mu m}$, we arrive at~\eqref{eq:deri_alpha}.

Differentiating the formula for the Wigner-D matrix with respect to $\beta$ gives
\begin{align*}
    \frac{\partial D^l_{\mu m}}{\partial \beta} & (\alpha, \beta, \gamma) = e^{-\mathrm{i}(\mu\alpha+ m\gamma)} [(l+m)!(l-m)!(l+\mu)!(l-\mu)!]^{\frac12} \\ &\sum_s \frac{(-1)^{s}}{2}\frac{(\cos\frac{\beta}{2})^{2l + \mu - m - 2s - 1}(\sin\frac{\beta}{2})^{m-\mu+2s-1}[(-2l-\mu+m+2s)(\sin\frac{\beta}{2})^2 + (m-\mu+2s)(\cos\frac{\beta}{2})^2]}{(l+\mu-s)!s! (m-\mu+s)!(l-m-s)!}.
\end{align*}
Let $(\alpha,\beta,\gamma) = (0,0,0)$, we see that the above derivative can be non-zero only when $m-\mu+2s-1=0$ or $m-\mu+2s+1=0$. Since $s\in \{\max(0, \mu-m),\ldots,\min(l+\mu,l-m)\}$, the latter is never true since its left hand side is always positive. In addition, the previous conditions can be fulfilled only when $|m-\mu| = 1$. Otherwise, the derivative with respect to $\beta$ at the origin will be zero.

If $m-\mu=1$, then only $s=0$ can contribute to the summation. Thus 
\begin{align*}
    \frac{\partial D^l_{\mu m}}{\partial \beta}(I)
    & = [(l+m)!(l-m)!(l+\mu)!(l-\mu)!]^{\frac{1}{2}}\frac{(-1)^{0}}{2}\frac{1}{(l+\mu)!0!1!(l-m)!} \\
    & = \frac{1}{2}\frac{[(l+m)!(l-m)!(l+m-1)!(l-m+1)!]^{\frac{1}{2}}}{(l+m-1)!(l-m)!} \\
    & = \frac{1}{2}[(l+m)(l-m+1)]^{\frac{1}{2}}.
\end{align*}
Similarly, when $m-\mu=-1$, only $s=1$ contributes to the summation and 
\begin{align*}
    \frac{\partial D^l_{\mu m}}{\partial \beta}(I)
    & = [(l+m)!(l-m)!(l+\mu)!(l-\mu)!]^{\frac{1}{2}}\frac{(-1)^{1}}{2}\frac{1}{(l+\mu-1)!1! 0!(l-m-1)!} \\
    & = - \frac{1}{2}\frac{[(l+m)!(l-m)!(l+m+1)!(l-m-1)!]^{\frac{1}{2}}}{(l+m)!(l-m-1)!} \\
    & = - \frac{1}{2}[(l-m)(l+m+1)]^{\frac{1}{2}},
\end{align*}
which completes the proof of~\eqref{eq:deri_beta}.
\end{proof} 

\section{Table of dimensionalities}\label{app:table_of_dim}
In this appendix, we provide tables that contain exact values of $\dimge{\bl}{L}$ (denoted by GE) and $\dimpi{\bl}{L}$ (denoted by GE-PI) for some typical $\bl=(l,l,\ldots,l)\in\calL^N$, which is the minimal unit of constructing the equivariant bases (cf. Proposition~\ref{prop:GEPI_rec}). 
In this setting, the dimensions $\dimge{\bl}{L}$ and $\dimpi{\bl}{L}$ are determined by three parameters: the degree $l$, the correlation order $N$ and the order of equivariance $L$. We show tables for fixed $N$, or fixed $l$.

\begin{table}[h]
\resizebox{\textwidth}{!}{%
\begin{tabular}{|c|cc|cc|cc|cc|cc|cc|cc|cc|}
\hline
\multirow{2}{*}{$L/\Nat$} &
 \multicolumn{2}{c|}{1} &
 \multicolumn{2}{c|}{2} &
 \multicolumn{2}{c|}{3} &
 \multicolumn{2}{c|}{4} &
 \multicolumn{2}{c|}{5} &
 \multicolumn{2}{c|}{6} &
 \multicolumn{2}{c|}{7} &
 \multicolumn{2}{c|}{8} \\ \cline{2-17} 
&
 \multicolumn{1}{c|}{GE-PI} &
 GE &
 \multicolumn{1}{c|}{GE-PI} &
 GE &
 \multicolumn{1}{c|}{GE-PI} &
 GE &
 \multicolumn{1}{c|}{GE-PI} &
 GE &
 \multicolumn{1}{c|}{GE-PI} &
 GE &
 \multicolumn{1}{c|}{GE-PI} &
 GE &
 \multicolumn{1}{c|}{GE-PI} &
 GE &
 \multicolumn{1}{c|}{GE-PI} &
 GE \\ \hline
0 &
 \multicolumn{1}{c|}{0} &
 0 &
 \multicolumn{1}{c|}{1} &
 1 &
 \multicolumn{1}{c|}{0} &
 1 &
 \multicolumn{1}{c|}{1} &
 3 &
 \multicolumn{1}{c|}{0} &
 6 &
 \multicolumn{1}{c|}{1} &
 15 &
 \multicolumn{1}{c|}{0} &
 36 &
 \multicolumn{1}{c|}{1} &
 91 \\ \hline
1 &
 \multicolumn{1}{c|}{1} &
 1 &
 \multicolumn{1}{c|}{0} &
 1 &
 \multicolumn{1}{c|}{1} &
 3 &
 \multicolumn{1}{c|}{0} &
 6 &
 \multicolumn{1}{c|}{1} &
 15 &
 \multicolumn{1}{c|}{0} &
 36 &
 \multicolumn{1}{c|}{1} &
 91 &
 \multicolumn{1}{c|}{0} &
 232 \\ \hline
2 &
 \multicolumn{1}{c|}{} &
  &
 \multicolumn{1}{c|}{1} &
 1 &
 \multicolumn{1}{c|}{0} &
 2 &
 \multicolumn{1}{c|}{1} &
 6 &
 \multicolumn{1}{c|}{0} &
 15 &
 \multicolumn{1}{c|}{1} &
 40 &
 \multicolumn{1}{c|}{0} &
 105 &
 \multicolumn{1}{c|}{1} &
 280 \\ \hline
3 &
 \multicolumn{1}{c|}{} &
  &
 \multicolumn{1}{c|}{} &
  &
 \multicolumn{1}{c|}{1} &
 1 &
 \multicolumn{1}{c|}{0} &
 3 &
 \multicolumn{1}{c|}{1} &
 10 &
 \multicolumn{1}{c|}{0} &
 29 &
 \multicolumn{1}{c|}{1} &
 84 &
 \multicolumn{1}{c|}{0} &
 238 \\ \hline
4 &
 \multicolumn{1}{c|}{} &
  &
 \multicolumn{1}{c|}{} &
  &
 \multicolumn{1}{c|}{} &
  &
 \multicolumn{1}{c|}{1} &
 1 &
 \multicolumn{1}{c|}{0} &
 4 &
 \multicolumn{1}{c|}{1} &
 15 &
 \multicolumn{1}{c|}{0} &
 49 &
 \multicolumn{1}{c|}{1} &
 154 \\ \hline
5 &
 \multicolumn{1}{c|}{} &
  &
 \multicolumn{1}{c|}{} &
  &
 \multicolumn{1}{c|}{} &
  &
 \multicolumn{1}{c|}{} &
  &
 \multicolumn{1}{c|}{1} &
 1 &
 \multicolumn{1}{c|}{0} &
 5 &
 \multicolumn{1}{c|}{1} &
 21 &
 \multicolumn{1}{c|}{0} &
 76 \\ \hline
6 &
 \multicolumn{1}{c|}{} &
  &
 \multicolumn{1}{c|}{} &
  &
 \multicolumn{1}{c|}{} &
  &
 \multicolumn{1}{c|}{} &
  &
 \multicolumn{1}{c|}{} &
  &
 \multicolumn{1}{c|}{1} &
 1 &
 \multicolumn{1}{c|}{0} &
 6 &
 \multicolumn{1}{c|}{1} &
 28 \\ \hline
7 &
 \multicolumn{1}{c|}{} &
  &
 \multicolumn{1}{c|}{} &
  &
 \multicolumn{1}{c|}{} &
  &
 \multicolumn{1}{c|}{} &
  &
 \multicolumn{1}{c|}{} &
  &
 \multicolumn{1}{c|}{} &
  &
 \multicolumn{1}{c|}{1} &
 1 &
 \multicolumn{1}{c|}{0} &
 7 \\ \hline
8 &
 \multicolumn{1}{c|}{} &
  &
 \multicolumn{1}{c|}{} &
  &
 \multicolumn{1}{c|}{} &
  &
 \multicolumn{1}{c|}{} &
  &
 \multicolumn{1}{c|}{} &
  &
 \multicolumn{1}{c|}{} &
  &
 \multicolumn{1}{c|}{} &
  &
 \multicolumn{1}{c|}{1} &
 1 \\ \hline
\end{tabular}
}
\caption{Dimensions of $V^{\bl,L}$ (GE) and $\bV^{\bl,L}$ (GE-PI) for $\bl = (1,1,\ldots,1)$ with length $\Nat$, and for $L\in\{0,1,\ldots,\Nat\}$.}
\label{tab:l1}
\end{table}

\begin{table}[htb!]
\resizebox{\textwidth}{!}{%
\begin{tabular}{|c|cc|cc|cc|cc|cc|cc|cc|cc|}
\hline
\multirow{2}{*}{$L$/$\Nat$} & \multicolumn{2}{c|}{1} & \multicolumn{2}{c|}{2} & \multicolumn{2}{c|}{3} & \multicolumn{2}{c|}{4} & \multicolumn{2}{c|}{5} & \multicolumn{2}{c|}{6} & \multicolumn{2}{c|}{7} & \multicolumn{2}{c|}{8} \\ \cline{2-17} 
 & \multicolumn{1}{c|}{GE-PI} & GE & \multicolumn{1}{c|}{GE-PI} & GE & \multicolumn{1}{c|}{GE-PI} & GE & \multicolumn{1}{c|}{GE-PI} & GE & \multicolumn{1}{c|}{GE-PI} & GE & \multicolumn{1}{c|}{GE-PI} & GE & \multicolumn{1}{c|}{GE-PI} & GE & \multicolumn{1}{c|}{GE-PI} & GE \\ \hline
0 & \multicolumn{1}{c|}{0} & 0 & \multicolumn{1}{c|}{1} & 1 & \multicolumn{1}{c|}{0} & 1 & \multicolumn{1}{c|}{2} & 7 & \multicolumn{1}{c|}{0} & 31 & \multicolumn{1}{c|}{3} & 175 & \multicolumn{1}{c|}{0} & 981 & \multicolumn{1}{c|}{4} & 5719 \\ \hline
1 & \multicolumn{1}{c|}{0} & 0 & \multicolumn{1}{c|}{0} & 1 & \multicolumn{1}{c|}{1} & 3 & \multicolumn{1}{c|}{0} & 18 & \multicolumn{1}{c|}{2} & 90 & \multicolumn{1}{c|}{0} & 504 & \multicolumn{1}{c|}{4} & 2856 & \multicolumn{1}{c|}{1} & 16688 \\ \hline
2 & \multicolumn{1}{c|}{0} & 0 & \multicolumn{1}{c|}{1} & 1 & \multicolumn{1}{c|}{0} & 5 & \multicolumn{1}{c|}{2} & 26 & \multicolumn{1}{c|}{1} & 140 & \multicolumn{1}{c|}{4} & 780 & \multicolumn{1}{c|}{2} & 4480 & \multicolumn{1}{c|}{7} & 26320 \\ \hline
3 & \multicolumn{1}{c|}{1} & 1 & \multicolumn{1}{c|}{0} & 1 & \multicolumn{1}{c|}{2} & 7 & \multicolumn{1}{c|}{1} & 31 & \multicolumn{1}{c|}{4} & 175 & \multicolumn{1}{c|}{3} & 981 & \multicolumn{1}{c|}{7} & 5719 & \multicolumn{1}{c|}{5} & 33922 \\ \hline
4 & \multicolumn{1}{c|}{} &  & \multicolumn{1}{c|}{1} & 1 & \multicolumn{1}{c|}{1} & 6 & \multicolumn{1}{c|}{3} & 33 & \multicolumn{1}{c|}{2} & 189 & \multicolumn{1}{c|}{6} & 1095 & \multicolumn{1}{c|}{5} & 6489 & \multicolumn{1}{c|}{11} & 39046 \\ \hline
5 & \multicolumn{1}{c|}{} &  & \multicolumn{1}{c|}{0} & 1 & \multicolumn{1}{c|}{1} & 5 & \multicolumn{1}{c|}{1} & 32 & \multicolumn{1}{c|}{4} & 186 & \multicolumn{1}{c|}{3} & 1120 & \multicolumn{1}{c|}{8} & 6776 & \multicolumn{1}{c|}{7} & 41524 \\ \hline
6 & \multicolumn{1}{c|}{} &  & \multicolumn{1}{c|}{1} & 1 & \multicolumn{1}{c|}{1} & 4 & \multicolumn{1}{c|}{3} & 28 & \multicolumn{1}{c|}{3} & 170 & \multicolumn{1}{c|}{7} & 1064 & \multicolumn{1}{c|}{7} & 6621 & \multicolumn{1}{c|}{13} & 41468 \\ \hline
7 & \multicolumn{1}{c|}{} &  & \multicolumn{1}{c|}{} &  & \multicolumn{1}{c|}{1} & 3 & \multicolumn{1}{c|}{1} & 21 & \multicolumn{1}{c|}{4} & 145 & \multicolumn{1}{c|}{4} & 945 & \multicolumn{1}{c|}{9} & 6105 & \multicolumn{1}{c|}{9} & 39235 \\ \hline
8 & \multicolumn{1}{c|}{} &  & \multicolumn{1}{c|}{} &  & \multicolumn{1}{c|}{0} & 2 & \multicolumn{1}{c|}{2} & 15 & \multicolumn{1}{c|}{2} & 115 & \multicolumn{1}{c|}{6} & 791 & \multicolumn{1}{c|}{6} & 5334 & \multicolumn{1}{c|}{13} & 35357 \\ \hline
9 & \multicolumn{1}{c|}{} &  & \multicolumn{1}{c|}{} &  & \multicolumn{1}{c|}{1} & 1 & \multicolumn{1}{c|}{1} & 10 & \multicolumn{1}{c|}{3} & 84 & \multicolumn{1}{c|}{4} & 625 & \multicolumn{1}{c|}{9} & 4424 & \multicolumn{1}{c|}{10} & 30436 \\ \hline
10 & \multicolumn{1}{c|}{} &  & \multicolumn{1}{c|}{} &  & \multicolumn{1}{c|}{} &  & \multicolumn{1}{c|}{1} & 6 & \multicolumn{1}{c|}{2} & 56 & \multicolumn{1}{c|}{5} & 465 & \multicolumn{1}{c|}{6} & 3486 & \multicolumn{1}{c|}{12} & 25060 \\ \hline
11 & \multicolumn{1}{c|}{} &  & \multicolumn{1}{c|}{} &  & \multicolumn{1}{c|}{} &  & \multicolumn{1}{c|}{0} & 3 & \multicolumn{1}{c|}{2} & 35 & \multicolumn{1}{c|}{2} & 324 & \multicolumn{1}{c|}{7} & 2611 & \multicolumn{1}{c|}{8} & 19740 \\ \hline
12 & \multicolumn{1}{c|}{} &  & \multicolumn{1}{c|}{} &  & \multicolumn{1}{c|}{} &  & \multicolumn{1}{c|}{1} & 1 & \multicolumn{1}{c|}{1} & 20 & \multicolumn{1}{c|}{4} & 210 & \multicolumn{1}{c|}{5} & 1855 & \multicolumn{1}{c|}{11} & 14868 \\ \hline
13 & \multicolumn{1}{c|}{} &  & \multicolumn{1}{c|}{} &  & \multicolumn{1}{c|}{} &  & \multicolumn{1}{c|}{} &  & \multicolumn{1}{c|}{1} & 10 & \multicolumn{1}{c|}{2} & 126 & \multicolumn{1}{c|}{5} & 1245 & \multicolumn{1}{c|}{7} & 10696 \\ \hline
14 & \multicolumn{1}{c|}{} &  & \multicolumn{1}{c|}{} &  & \multicolumn{1}{c|}{} &  & \multicolumn{1}{c|}{} &  & \multicolumn{1}{c|}{0} & 4 & \multicolumn{1}{c|}{2} & 70 & \multicolumn{1}{c|}{3} & 785 & \multicolumn{1}{c|}{8} & 7336 \\ \hline
15 & \multicolumn{1}{c|}{} &  & \multicolumn{1}{c|}{} &  & \multicolumn{1}{c|}{} &  & \multicolumn{1}{c|}{} &  & \multicolumn{1}{c|}{1} & 1 & \multicolumn{1}{c|}{1} & 35 & \multicolumn{1}{c|}{4} & 462 & \multicolumn{1}{c|}{5} & 4781 \\ \hline
16 & \multicolumn{1}{c|}{} &  & \multicolumn{1}{c|}{} &  & \multicolumn{1}{c|}{} &  & \multicolumn{1}{c|}{} &  & \multicolumn{1}{c|}{} &  & \multicolumn{1}{c|}{1} & 15 & \multicolumn{1}{c|}{2} & 252 & \multicolumn{1}{c|}{6} & 2947 \\ \hline
17 & \multicolumn{1}{c|}{} &  & \multicolumn{1}{c|}{} &  & \multicolumn{1}{c|}{} &  & \multicolumn{1}{c|}{} &  & \multicolumn{1}{c|}{} &  & \multicolumn{1}{c|}{0} & 5 & \multicolumn{1}{c|}{2} & 126 & \multicolumn{1}{c|}{3} & 1708 \\ \hline
18 & \multicolumn{1}{c|}{} &  & \multicolumn{1}{c|}{} &  & \multicolumn{1}{c|}{} &  & \multicolumn{1}{c|}{} &  & \multicolumn{1}{c|}{} &  & \multicolumn{1}{c|}{1} & 1 & \multicolumn{1}{c|}{1} & 56 & \multicolumn{1}{c|}{4} & 924 \\ \hline
19 & \multicolumn{1}{c|}{} &  & \multicolumn{1}{c|}{} &  & \multicolumn{1}{c|}{} &  & \multicolumn{1}{c|}{} &  & \multicolumn{1}{c|}{} &  & \multicolumn{1}{c|}{} &  & \multicolumn{1}{c|}{1} & 21 & \multicolumn{1}{c|}{2} & 462 \\ \hline
20 & \multicolumn{1}{c|}{} &  & \multicolumn{1}{c|}{} &  & \multicolumn{1}{c|}{} &  & \multicolumn{1}{c|}{} &  & \multicolumn{1}{c|}{} &  & \multicolumn{1}{c|}{} &  & \multicolumn{1}{c|}{0} & 6 & \multicolumn{1}{c|}{2} & 210 \\ \hline
21 & \multicolumn{1}{c|}{} &  & \multicolumn{1}{c|}{} &  & \multicolumn{1}{c|}{} &  & \multicolumn{1}{c|}{} &  & \multicolumn{1}{c|}{} &  & \multicolumn{1}{c|}{} &  & \multicolumn{1}{c|}{1} & 1 & \multicolumn{1}{c|}{1} & 84 \\ \hline
22 & \multicolumn{1}{c|}{} &  & \multicolumn{1}{c|}{} &  & \multicolumn{1}{c|}{} &  & \multicolumn{1}{c|}{} &  & \multicolumn{1}{c|}{} &  & \multicolumn{1}{c|}{} &  & \multicolumn{1}{c|}{} &  & \multicolumn{1}{c|}{1} & 28 \\ \hline
23 & \multicolumn{1}{c|}{} &  & \multicolumn{1}{c|}{} &  & \multicolumn{1}{c|}{} &  & \multicolumn{1}{c|}{} &  & \multicolumn{1}{c|}{} &  & \multicolumn{1}{c|}{} &  & \multicolumn{1}{c|}{} &  & \multicolumn{1}{c|}{0} & 7 \\ \hline
24 & \multicolumn{1}{c|}{} &  & \multicolumn{1}{c|}{} &  & \multicolumn{1}{c|}{} &  & \multicolumn{1}{c|}{} &  & \multicolumn{1}{c|}{} &  & \multicolumn{1}{c|}{} &  & \multicolumn{1}{c|}{} &  & \multicolumn{1}{c|}{1} & 1 \\ \hline
\end{tabular}%
}
\caption{Dimensions of $V^{\bl,L}$ (GE) and $\bV^{\bl,L}$ (GE-PI) for $\bl = (3,3,\ldots,3)$ with length $\Nat$, and for $L\in\{0,1,\ldots,3\Nat\}$.}
\label{tab:l3}
\end{table}

\begin{table}[htb!]
\setlength{\tabcolsep}{7pt}
\resizebox{\textwidth}{!}{%
\begin{tabular}{|c|cc|cc|cc|cc|cc|cc|cc|cc|}
\hline
\multirow{2}{*}{$L/l$} & \multicolumn{2}{c|}{0} & \multicolumn{2}{c|}{1} & \multicolumn{2}{c|}{2} & \multicolumn{2}{c|}{3} & \multicolumn{2}{c|}{4} & \multicolumn{2}{c|}{5} & \multicolumn{2}{c|}{6} & \multicolumn{2}{c|}{7} \\ \cline{2-17} 
 & \multicolumn{1}{c|}{GE-PI} & GE & \multicolumn{1}{c|}{GE-PI} & GE & \multicolumn{1}{c|}{GE-PI} & GE & \multicolumn{1}{c|}{GE-PI} & GE & \multicolumn{1}{c|}{GE-PI} & GE & \multicolumn{1}{c|}{GE-PI} & GE & \multicolumn{1}{c|}{GE-PI} & GE & \multicolumn{1}{c|}{GE-PI} & GE \\ \hline
0 & \multicolumn{1}{c|}{1} & 1 & \multicolumn{1}{c|}{0} & 1 & \multicolumn{1}{c|}{1} & 1 & \multicolumn{1}{c|}{0} & 1 & \multicolumn{1}{c|}{1} & 1 & \multicolumn{1}{c|}{0} & 1 & \multicolumn{1}{c|}{1} & 1 & \multicolumn{1}{c|}{0} & 1 \\ \hline
1 & \multicolumn{1}{c|}{} &  & \multicolumn{1}{c|}{1} & 3 & \multicolumn{1}{c|}{0} & 3 & \multicolumn{1}{c|}{1} & 3 & \multicolumn{1}{c|}{0} & 3 & \multicolumn{1}{c|}{1} & 3 & \multicolumn{1}{c|}{0} & 3 & \multicolumn{1}{c|}{1} & 3 \\ \hline
2 & \multicolumn{1}{c|}{} &  & \multicolumn{1}{c|}{0} & 2 & \multicolumn{1}{c|}{1} & 5 & \multicolumn{1}{c|}{0} & 5 & \multicolumn{1}{c|}{1} & 5 & \multicolumn{1}{c|}{0} & 5 & \multicolumn{1}{c|}{1} & 5 & \multicolumn{1}{c|}{0} & 5 \\ \hline
3 & \multicolumn{1}{c|}{} &  & \multicolumn{1}{c|}{1} & 1 & \multicolumn{1}{c|}{1} & 4 & \multicolumn{1}{c|}{2} & 7 & \multicolumn{1}{c|}{1} & 7 & \multicolumn{1}{c|}{2} & 7 & \multicolumn{1}{c|}{1} & 7 & \multicolumn{1}{c|}{2} & 7 \\ \hline
4 & \multicolumn{1}{c|}{} &  & \multicolumn{1}{c|}{} &  & \multicolumn{1}{c|}{1} & 3 & \multicolumn{1}{c|}{1} & 6 & \multicolumn{1}{c|}{2} & 9 & \multicolumn{1}{c|}{1} & 9 & \multicolumn{1}{c|}{2} & 9 & \multicolumn{1}{c|}{1} & 9 \\ \hline
5 & \multicolumn{1}{c|}{} &  & \multicolumn{1}{c|}{} &  & \multicolumn{1}{c|}{0} & 2 & \multicolumn{1}{c|}{1} & 5 & \multicolumn{1}{c|}{1} & 8 & \multicolumn{1}{c|}{2} & 11 & \multicolumn{1}{c|}{1} & 11 & \multicolumn{1}{c|}{2} & 11 \\ \hline
6 & \multicolumn{1}{c|}{} &  & \multicolumn{1}{c|}{} &  & \multicolumn{1}{c|}{1} & 1 & \multicolumn{1}{c|}{1} & 4 & \multicolumn{1}{c|}{2} & 7 & \multicolumn{1}{c|}{2} & 10 & \multicolumn{1}{c|}{3} & 13 & \multicolumn{1}{c|}{2} & 13 \\ \hline
7 & \multicolumn{1}{c|}{} &  & \multicolumn{1}{c|}{} &  & \multicolumn{1}{c|}{} &  & \multicolumn{1}{c|}{1} & 3 & \multicolumn{1}{c|}{1} & 6 & \multicolumn{1}{c|}{2} & 9 & \multicolumn{1}{c|}{2} & 12 & \multicolumn{1}{c|}{3} & 15 \\ \hline
8 & \multicolumn{1}{c|}{} &  & \multicolumn{1}{c|}{} &  & \multicolumn{1}{c|}{} &  & \multicolumn{1}{c|}{0} & 2 & \multicolumn{1}{c|}{1} & 5 & \multicolumn{1}{c|}{1} & 8 & \multicolumn{1}{c|}{2} & 11 & \multicolumn{1}{c|}{2} & 14 \\ \hline
9 & \multicolumn{1}{c|}{} &  & \multicolumn{1}{c|}{} &  & \multicolumn{1}{c|}{} &  & \multicolumn{1}{c|}{1} & 1 & \multicolumn{1}{c|}{1} & 4 & \multicolumn{1}{c|}{2} & 7 & \multicolumn{1}{c|}{2} & 10 & \multicolumn{1}{c|}{3} & 13 \\ \hline
10 & \multicolumn{1}{c|}{} &  & \multicolumn{1}{c|}{} &  & \multicolumn{1}{c|}{} &  & \multicolumn{1}{c|}{} &  & \multicolumn{1}{c|}{1} & 3 & \multicolumn{1}{c|}{1} & 6 & \multicolumn{1}{c|}{2} & 9 & \multicolumn{1}{c|}{2} & 12 \\ \hline
11 & \multicolumn{1}{c|}{} &  & \multicolumn{1}{c|}{} &  & \multicolumn{1}{c|}{} &  & \multicolumn{1}{c|}{} &  & \multicolumn{1}{c|}{0} & 2 & \multicolumn{1}{c|}{1} & 5 & \multicolumn{1}{c|}{1} & 8 & \multicolumn{1}{c|}{2} & 11 \\ \hline
12 & \multicolumn{1}{c|}{} &  & \multicolumn{1}{c|}{} &  & \multicolumn{1}{c|}{} &  & \multicolumn{1}{c|}{} &  & \multicolumn{1}{c|}{1} & 1 & \multicolumn{1}{c|}{1} & 4 & \multicolumn{1}{c|}{2} & 7 & \multicolumn{1}{c|}{2} & 10 \\ \hline
13 & \multicolumn{1}{c|}{} &  & \multicolumn{1}{c|}{} &  & \multicolumn{1}{c|}{} &  & \multicolumn{1}{c|}{} &  & \multicolumn{1}{c|}{} &  & \multicolumn{1}{c|}{1} & 3 & \multicolumn{1}{c|}{1} & 6 & \multicolumn{1}{c|}{2} & 9 \\ \hline
14 & \multicolumn{1}{c|}{} &  & \multicolumn{1}{c|}{} &  & \multicolumn{1}{c|}{} &  & \multicolumn{1}{c|}{} &  & \multicolumn{1}{c|}{} &  & \multicolumn{1}{c|}{0} & 2 & \multicolumn{1}{c|}{1} & 5 & \multicolumn{1}{c|}{1} & 8 \\ \hline
15 & \multicolumn{1}{c|}{} &  & \multicolumn{1}{c|}{} &  & \multicolumn{1}{c|}{} &  & \multicolumn{1}{c|}{} &  & \multicolumn{1}{c|}{} &  & \multicolumn{1}{c|}{1} & 1 & \multicolumn{1}{c|}{1} & 4 & \multicolumn{1}{c|}{2} & 7 \\ \hline
16 & \multicolumn{1}{c|}{} &  & \multicolumn{1}{c|}{} &  & \multicolumn{1}{c|}{} &  & \multicolumn{1}{c|}{} &  & \multicolumn{1}{c|}{} &  & \multicolumn{1}{c|}{} &  & \multicolumn{1}{c|}{1} & 3 & \multicolumn{1}{c|}{1} & 6 \\ \hline
17 & \multicolumn{1}{c|}{} &  & \multicolumn{1}{c|}{} &  & \multicolumn{1}{c|}{} &  & \multicolumn{1}{c|}{} &  & \multicolumn{1}{c|}{} &  & \multicolumn{1}{c|}{} &  & \multicolumn{1}{c|}{0} & 2 & \multicolumn{1}{c|}{1} & 5 \\ \hline
18 & \multicolumn{1}{c|}{} &  & \multicolumn{1}{c|}{} &  & \multicolumn{1}{c|}{} &  & \multicolumn{1}{c|}{} &  & \multicolumn{1}{c|}{} &  & \multicolumn{1}{c|}{} &  & \multicolumn{1}{c|}{1} & 1 & \multicolumn{1}{c|}{1} & 4 \\ \hline
19 & \multicolumn{1}{c|}{} &  & \multicolumn{1}{c|}{} &  & \multicolumn{1}{c|}{} &  & \multicolumn{1}{c|}{} &  & \multicolumn{1}{c|}{} &  & \multicolumn{1}{c|}{} &  & \multicolumn{1}{c|}{} &  & \multicolumn{1}{c|}{1} & 3 \\ \hline
20 & \multicolumn{1}{c|}{} &  & \multicolumn{1}{c|}{} &  & \multicolumn{1}{c|}{} &  & \multicolumn{1}{c|}{} &  & \multicolumn{1}{c|}{} &  & \multicolumn{1}{c|}{} &  & \multicolumn{1}{c|}{} &  & \multicolumn{1}{c|}{0} & 2 \\ \hline
21 & \multicolumn{1}{c|}{} &  & \multicolumn{1}{c|}{} &  & \multicolumn{1}{c|}{} &  & \multicolumn{1}{c|}{} &  & \multicolumn{1}{c|}{} &  & \multicolumn{1}{c|}{} &  & \multicolumn{1}{c|}{} &  & \multicolumn{1}{c|}{1} & 1 \\ \hline
\end{tabular}%
}
\caption{Dimensions of $V^{\bl,L}$ (GE) and $\bV^{\bl,L}$ (GE-PI) for $\bl=(l,l,l)$ with different $l$ and $L$.}
\label{tab:n3}
\end{table}

\begin{table}[htb!]
\setlength{\tabcolsep}{7pt}
\centering
\resizebox{\textwidth}{!}{%
\begin{tabular}{|c|cc|cc|cc|cc|cc|cc|cc|cc|}
\hline
\multirow{2}{*}{$L/l$} & \multicolumn{2}{c|}{0} & \multicolumn{2}{c|}{1} & \multicolumn{2}{c|}{2} & \multicolumn{2}{c|}{3} & \multicolumn{2}{c|}{4} & \multicolumn{2}{c|}{5} & \multicolumn{2}{c|}{6} & \multicolumn{2}{c|}{7} \\ \cline{2-17} 
 & \multicolumn{1}{c|}{GE-PI} & GE & \multicolumn{1}{c|}{GE-PI} & GE & \multicolumn{1}{c|}{GE-PI} & GE & \multicolumn{1}{c|}{GE-PI} & GE & \multicolumn{1}{c|}{GE-PI} & GE & \multicolumn{1}{c|}{GE-PI} & GE & \multicolumn{1}{c|}{GE-PI} & GE & \multicolumn{1}{c|}{GE-PI} & GE \\ \hline
0 & \multicolumn{1}{c|}{1} & 1 & \multicolumn{1}{c|}{1} & 3 & \multicolumn{1}{c|}{1} & 5 & \multicolumn{1}{c|}{2} & 7 & \multicolumn{1}{c|}{2} & 9 & \multicolumn{1}{c|}{2} & 11 & \multicolumn{1}{c|}{3} & 13 & \multicolumn{1}{c|}{3} & 15 \\ \hline
1 & \multicolumn{1}{c|}{} &  & \multicolumn{1}{c|}{0} & 6 & \multicolumn{1}{c|}{0} & 12 & \multicolumn{1}{c|}{0} & 18 & \multicolumn{1}{c|}{0} & 24 & \multicolumn{1}{c|}{0} & 30 & \multicolumn{1}{c|}{0} & 36 & \multicolumn{1}{c|}{0} & 42 \\ \hline
2 & \multicolumn{1}{c|}{} &  & \multicolumn{1}{c|}{1} & 6 & \multicolumn{1}{c|}{2} & 16 & \multicolumn{1}{c|}{2} & 26 & \multicolumn{1}{c|}{3} & 36 & \multicolumn{1}{c|}{4} & 46 & \multicolumn{1}{c|}{4} & 56 & \multicolumn{1}{c|}{5} & 66 \\ \hline
3 & \multicolumn{1}{c|}{} &  & \multicolumn{1}{c|}{0} & 3 & \multicolumn{1}{c|}{0} & 17 & \multicolumn{1}{c|}{1} & 31 & \multicolumn{1}{c|}{1} & 45 & \multicolumn{1}{c|}{1} & 59 & \multicolumn{1}{c|}{2} & 73 & \multicolumn{1}{c|}{2} & 87 \\ \hline
4 & \multicolumn{1}{c|}{} &  & \multicolumn{1}{c|}{1} & 1 & \multicolumn{1}{c|}{2} & 15 & \multicolumn{1}{c|}{3} & 33 & \multicolumn{1}{c|}{4} & 51 & \multicolumn{1}{c|}{5} & 69 & \multicolumn{1}{c|}{6} & 87 & \multicolumn{1}{c|}{7} & 105 \\ \hline
5 & \multicolumn{1}{c|}{} &  & \multicolumn{1}{c|}{} &  & \multicolumn{1}{c|}{1} & 10 & \multicolumn{1}{c|}{1} & 32 & \multicolumn{1}{c|}{2} & 54 & \multicolumn{1}{c|}{3} & 76 & \multicolumn{1}{c|}{3} & 98 & \multicolumn{1}{c|}{4} & 120 \\ \hline
6 & \multicolumn{1}{c|}{} &  & \multicolumn{1}{c|}{} &  & \multicolumn{1}{c|}{1} & 6 & \multicolumn{1}{c|}{3} & 28 & \multicolumn{1}{c|}{4} & 54 & \multicolumn{1}{c|}{5} & 80 & \multicolumn{1}{c|}{7} & 106 & \multicolumn{1}{c|}{8} & 132 \\ \hline
7 & \multicolumn{1}{c|}{} &  & \multicolumn{1}{c|}{} &  & \multicolumn{1}{c|}{0} & 3 & \multicolumn{1}{c|}{1} & 21 & \multicolumn{1}{c|}{2} & 51 & \multicolumn{1}{c|}{3} & 81 & \multicolumn{1}{c|}{4} & 111 & \multicolumn{1}{c|}{5} & 141 \\ \hline
8 & \multicolumn{1}{c|}{} &  & \multicolumn{1}{c|}{} &  & \multicolumn{1}{c|}{1} & 1 & \multicolumn{1}{c|}{2} & 15 & \multicolumn{1}{c|}{4} & 45 & \multicolumn{1}{c|}{6} & 79 & \multicolumn{1}{c|}{7} & 113 & \multicolumn{1}{c|}{9} & 147 \\ \hline
9 & \multicolumn{1}{c|}{} &  & \multicolumn{1}{c|}{} &  & \multicolumn{1}{c|}{} &  & \multicolumn{1}{c|}{1} & 10 & \multicolumn{1}{c|}{2} & 36 & \multicolumn{1}{c|}{3} & 74 & \multicolumn{1}{c|}{5} & 112 & \multicolumn{1}{c|}{6} & 150 \\ \hline
10 & \multicolumn{1}{c|}{} &  & \multicolumn{1}{c|}{} &  & \multicolumn{1}{c|}{} &  & \multicolumn{1}{c|}{1} & 6 & \multicolumn{1}{c|}{3} & 28 & \multicolumn{1}{c|}{5} & 66 & \multicolumn{1}{c|}{7} & 108 & \multicolumn{1}{c|}{9} & 150 \\ \hline
11 & \multicolumn{1}{c|}{} &  & \multicolumn{1}{c|}{} &  & \multicolumn{1}{c|}{} &  & \multicolumn{1}{c|}{0} & 3 & \multicolumn{1}{c|}{1} & 21 & \multicolumn{1}{c|}{3} & 55 & \multicolumn{1}{c|}{4} & 101 & \multicolumn{1}{c|}{6} & 147 \\ \hline
12 & \multicolumn{1}{c|}{} &  & \multicolumn{1}{c|}{} &  & \multicolumn{1}{c|}{} &  & \multicolumn{1}{c|}{1} & 1 & \multicolumn{1}{c|}{2} & 15 & \multicolumn{1}{c|}{4} & 45 & \multicolumn{1}{c|}{7} & 91 & \multicolumn{1}{c|}{9} & 141 \\ \hline
13 & \multicolumn{1}{c|}{} &  & \multicolumn{1}{c|}{} &  & \multicolumn{1}{c|}{} &  & \multicolumn{1}{c|}{} &  & \multicolumn{1}{c|}{1} & 10 & \multicolumn{1}{c|}{2} & 36 & \multicolumn{1}{c|}{4} & 78 & \multicolumn{1}{c|}{6} & 132 \\ \hline
14 & \multicolumn{1}{c|}{} &  & \multicolumn{1}{c|}{} &  & \multicolumn{1}{c|}{} &  & \multicolumn{1}{c|}{} &  & \multicolumn{1}{c|}{1} & 6 & \multicolumn{1}{c|}{3} & 28 & \multicolumn{1}{c|}{5} & 66 & \multicolumn{1}{c|}{8} & 120 \\ \hline
15 & \multicolumn{1}{c|}{} &  & \multicolumn{1}{c|}{} &  & \multicolumn{1}{c|}{} &  & \multicolumn{1}{c|}{} &  & \multicolumn{1}{c|}{0} & 3 & \multicolumn{1}{c|}{1} & 21 & \multicolumn{1}{c|}{3} & 55 & \multicolumn{1}{c|}{5} & 105 \\ \hline
16 & \multicolumn{1}{c|}{} &  & \multicolumn{1}{c|}{} &  & \multicolumn{1}{c|}{} &  & \multicolumn{1}{c|}{} &  & \multicolumn{1}{c|}{1} & 1 & \multicolumn{1}{c|}{2} & 15 & \multicolumn{1}{c|}{4} & 45 & \multicolumn{1}{c|}{7} & 91 \\ \hline
17 & \multicolumn{1}{c|}{} &  & \multicolumn{1}{c|}{} &  & \multicolumn{1}{c|}{} &  & \multicolumn{1}{c|}{} &  & \multicolumn{1}{c|}{} &  & \multicolumn{1}{c|}{1} & 10 & \multicolumn{1}{c|}{2} & 36 & \multicolumn{1}{c|}{4} & 78 \\ \hline
18 & \multicolumn{1}{c|}{} &  & \multicolumn{1}{c|}{} &  & \multicolumn{1}{c|}{} &  & \multicolumn{1}{c|}{} &  & \multicolumn{1}{c|}{} &  & \multicolumn{1}{c|}{1} & 6 & \multicolumn{1}{c|}{3} & 28 & \multicolumn{1}{c|}{5} & 66 \\ \hline
19 & \multicolumn{1}{c|}{} &  & \multicolumn{1}{c|}{} &  & \multicolumn{1}{c|}{} &  & \multicolumn{1}{c|}{} &  & \multicolumn{1}{c|}{} &  & \multicolumn{1}{c|}{0} & 3 & \multicolumn{1}{c|}{1} & 21 & \multicolumn{1}{c|}{3} & 55 \\ \hline
20 & \multicolumn{1}{c|}{} &  & \multicolumn{1}{c|}{} &  & \multicolumn{1}{c|}{} &  & \multicolumn{1}{c|}{} &  & \multicolumn{1}{c|}{} &  & \multicolumn{1}{c|}{1} & 1 & \multicolumn{1}{c|}{2} & 15 & \multicolumn{1}{c|}{4} & 45 \\ \hline
21 & \multicolumn{1}{c|}{} &  & \multicolumn{1}{c|}{} &  & \multicolumn{1}{c|}{} &  & \multicolumn{1}{c|}{} &  & \multicolumn{1}{c|}{} &  & \multicolumn{1}{c|}{} &  & \multicolumn{1}{c|}{1} & 10 & \multicolumn{1}{c|}{2} & 36 \\ \hline
22 & \multicolumn{1}{c|}{} &  & \multicolumn{1}{c|}{} &  & \multicolumn{1}{c|}{} &  & \multicolumn{1}{c|}{} &  & \multicolumn{1}{c|}{} &  & \multicolumn{1}{c|}{} &  & \multicolumn{1}{c|}{1} & 6 & \multicolumn{1}{c|}{3} & 28 \\ \hline
23 & \multicolumn{1}{c|}{} &  & \multicolumn{1}{c|}{} &  & \multicolumn{1}{c|}{} &  & \multicolumn{1}{c|}{} &  & \multicolumn{1}{c|}{} &  & \multicolumn{1}{c|}{} &  & \multicolumn{1}{c|}{0} & 3 & \multicolumn{1}{c|}{1} & 21 \\ \hline
24 & \multicolumn{1}{c|}{} &  & \multicolumn{1}{c|}{} &  & \multicolumn{1}{c|}{} &  & \multicolumn{1}{c|}{} &  & \multicolumn{1}{c|}{} &  & \multicolumn{1}{c|}{} &  & \multicolumn{1}{c|}{1} & 1 & \multicolumn{1}{c|}{2} & 15 \\ \hline
25 & \multicolumn{1}{c|}{} &  & \multicolumn{1}{c|}{} &  & \multicolumn{1}{c|}{} &  & \multicolumn{1}{c|}{} &  & \multicolumn{1}{c|}{} &  & \multicolumn{1}{c|}{} &  & \multicolumn{1}{c|}{} &  & \multicolumn{1}{c|}{1} & 10 \\ \hline
26 & \multicolumn{1}{c|}{} &  & \multicolumn{1}{c|}{} &  & \multicolumn{1}{c|}{} &  & \multicolumn{1}{c|}{} &  & \multicolumn{1}{c|}{} &  & \multicolumn{1}{c|}{} &  & \multicolumn{1}{c|}{} &  & \multicolumn{1}{c|}{1} & 6 \\ \hline
27 & \multicolumn{1}{c|}{} &  & \multicolumn{1}{c|}{} &  & \multicolumn{1}{c|}{} &  & \multicolumn{1}{c|}{} &  & \multicolumn{1}{c|}{} &  & \multicolumn{1}{c|}{} &  & \multicolumn{1}{c|}{} &  & \multicolumn{1}{c|}{0} & 3 \\ \hline
28 & \multicolumn{1}{c|}{} &  & \multicolumn{1}{c|}{} &  & \multicolumn{1}{c|}{} &  & \multicolumn{1}{c|}{} &  & \multicolumn{1}{c|}{} &  & \multicolumn{1}{c|}{} &  & \multicolumn{1}{c|}{} &  & \multicolumn{1}{c|}{1} & 1 \\ \hline
\end{tabular}%
}
\caption{Dimensions of $V^{\bl,L}$ (GE) and $\bV^{\bl,L}$ (GE-PI) for $\bl=(l,l,l,l)$ with different $l$ and $L$.}
\label{tab:n4}
\end{table}

\end{document}